\documentclass[12pt]{amsart}

\usepackage{amsmath} 
\usepackage{amssymb}
\usepackage{amsfonts}
\usepackage[T1]{fontenc}
\usepackage{comment}

\usepackage{ndstye}

\newcommand{\sub}{\operatorname{Sub}}
\newcommand{\st}{\Sigma _2}
\newcommand{\Cal}{\mathcal}
\newcommand{\re}{\operatorname{Re}}
\newcommand{\im}{\operatorname{Im}}
\newcommand{\w}[1]{\langle #1\rangle }
\newcommand{\set}[1]{\left\{\,#1\,\right\}}
\newcommand{\br}{\mathbf R}
\newcommand{\bc}{\mathbf C}
\newcommand{\bn}{\mathbf N}
\newcommand{\codim}{\operatorname{Codim}}
\newcommand{\wf}{\operatorname{WF}}
\newcommand{\restr}[1]{\big|_{#1}}
\newcommand{\ls}{\lesssim}
\newcommand{\gs}{\gtrsim}
\newcommand{\supp}{\operatorname{\rm supp}}
\newcommand{\wt}{\widetilde}
\newcommand{\hess}{\operatorname{Hess}}
\newcommand{\mn}[1]{\Vert#1\Vert}
\newcommand{\ol}{\overline}

\begin{document}

\hsize=160  mm
\vsize= 240 mm 

\baselineskip 17 pt  % 16 w. abstract
\lineskip 2pt
\lineskiplimit 2pt

\title[Subprincipal]{Solvability of subprincipal type operators} 
\author[NILS DENCKER]{{\textsc Nils Dencker}}
\address{Centre for Mathematical Sciences, University of Lund, Box 118,
	SE-221 00 Lund, Sweden}
\email{dencker@maths.lth.se}

%\date{January 21, 2018} 	

\numberwithin{equation}{section}

\begin{comment}
\begin{abstract}
In this paper we consider the solvability of
pseudodifferential operators in the case when the principal symbol vanishes of
order $k \ge 2 $ at a nonradial involutive manifold~$\st$.  We
shall assume that the operator is of subprincipal type, which means that the $ k$:th inhomogeneous blowup  at $\st$ of the refined principal symbol is of principal type with Hamilton vector field parallel to the base~$\st$, but
transversal to the symplectic leaves of $\st$ at the characteristics. When $k = \infty $ this blowup reduces to
the subprincipal symbol.
We also assume that the blowup is essentially constant on the leaves
of~$\st$, and does not satisfying the Nirenberg-Treves condition~(${\Psi}$). 
We also have conditions on the vanishing of the normal gradient and the Hessian of the
blowup at the characteristics. Under these conditions, we show that $P$ is not solvable.
\end{abstract} 
\end{comment}

\maketitle

\section{Introduction}

We shall consider the solvability for a classical pseudodifferential
operator $P \in  
{\Psi}^m_{\mathrm{cl}}(X)$ on a $C^\infty$ manifold $X$ of dimension $ n$. This means that $P$
has an expansion $p_m + p_{m-1} + \dots$ where $p_j \in S^{j}_{\mathrm{hom}}$
is homogeneous of 
degree $j$, $\forall\, j$, and $p_m = {\sigma}(P)$ is the principal
symbol of the operator. A pseudodifferential operator is said to be 
of principal type if the 
Hamilton vector field $H_{p_m}$ of the principal symbol does not
have the radial direction ${\xi}\cdot\partial_{\xi}$ on $p_m^{-1}(0)$,
in particular $H_{p_m} \ne 0$. We shall consider the case when the principal symbol
vanishes of at least second order at an involutive manifold~$\st$, thus~$P$
is not of principal type. 

$P$ is locally solvable at a compact set $K \subseteq X$
if the equation 
\begin{equation}\label{locsolv}
Pu = v 
\end{equation}
has a local solution $u \in \Cal D'(X)$ in a neighborhood of $K$
for any $v\in C^\infty(X)$
in a set of finite codimension.  
We can also define microlocal solvability of~$P$ at any compactly based cone
$K \subset T^*X$, see Definition~\ref{microsolv}.

For pseudodifferential operators of principal type,
local solvability is equivalent to condition (${\Psi}$) 
on the principal symbol, see~\cite{de:nt} and \cite{ho:nec}. This condition means that
\begin{multline}\label{psicond} \text{$\im ap_m$
    does not change sign from $-$ to $+$}\\
  \text{along the oriented
    bicharacteristics of $\re ap_m$}
\end{multline}
for any $0 \ne a \in C^\infty(T^*X)$. The oriented bicharacteristics are
the positive flow of the Hamilton vector field $H_{\re ap_m} \ne
0$ on which $\re ap_m =0$, these are also called
semibicharacteristics of~$p_m$.
Condition~\eqref{psicond} is invariant under multiplication of
~$p_m$ with nonvanishing factors, and symplectic changes of
variables, thus it is invariant under conjugation of~$P$ with elliptic
Fourier integral operators. Observe that the sign changes
in~\eqref{psicond} are reversed
when taking adjoints, and that it suffices to check~\eqref{psicond}
for some $a \ne 0$ for which $H_{\re ap} \ne 0$
according to~\cite[Theorem~26.4.12]{ho:yellow}.

For operators which are not of principal type, the situation is
more complicated and the solvability may depend on the lower order terms. 
Then the \emph{refined principal symbol} 
\begin{equation}\label{refpr}
p_{sub} = p_m + p_{m-1} + \frac{i}{2} \sum_j \partial_{\xi_j} \partial_{x_j} p_m
\end{equation} 
is invariantly defined modulo $S^{m-2}$ under changes of coordinates, see Theorem~18.1.33 in~\cite{ho:yellow}. 
In the Weyl quantization the refined principal symbol is given by $p_m + p_{m-1} $.

When~$\st$ is not involutive, there are examples where the
operator is solvable for any lower order terms. For 
example when~$P$ is effectively hyperbolic, then even the Cauchy problem
is solvable for any lower order term,
see~\cite{ho:cauchy}, \cite{Mel} and~\cite{Nishi}. 
There are also results in the cases when the principal symbol is a product of
principal type symbols not satisfying condition~(${\Psi}$),
see~\cite{CT},  \cite{GT},  \cite{Gold}, \cite{Treves} and~\cite{Yama1}.

In the case where the principal symbol is real and vanishes of at
least second order  at 
the involutive manifold there are several results, mostly in the case
when the principal symbol is a product of real symbols of principal
type. Then the operator is not solvable if the imaginary part of the
subprincipal symbol has 
a sign change of finite order on a bicharacteristic of one the factors of
the principal symbol, see~\cite{Ego}, \cite{Pop}, \cite{Wen1}
and~\cite{Wen2}.  

This necessary condition for solvability has been extended to some
cases when the principal symbol is real and vanishes of second order at the
involutive manifold. The conditions for solvability then involve the sign
changes of the imaginary part of the subprincipal symbol on the limits
of bicharacteristics 
from outside the manifold, thus on the leaves of the symplectic foliation of
the manifold, see~\cite{MU1}, \cite{MU2}, \cite{Men}
and~\cite{Yama2}. 

This has been extended to more general limit 
bicharacteristics of real principal symbols in~\cite{de:limit}. There we assumed that the bicharacteristics converge in $ C^\infty$ to a limit bicharacteristic. We also assumed that the linearization of the Hamilton vector field is tangent to and has uniform bounds on the tangent spaces of some Lagrangean manifolds at the bicharacteristics.  Then~$P$ is not solvable if condition Lim(${\Psi}$) is not satisfied on the limit bicharacteristics. This condition means that the quotient of the imaginary part of 
the subprincipal symbol with the norm of the Hamilton vector field
switches sign from $-$ to $+$ on the
bicharacteristics and becomes unbounded when converging to the limit
bicharacteristic. 
This was generalized in~\cite{de:compl} to operators with complex principal symbols.
There we assumed that the normalized complex Hamilton vector field of the principal symbol converges to a real
vector field. Then the limit bicharacteristics are uniquely defined, and one can invariantly define the imaginary part of the subprincipal symbol. Thus condition Lim(${\Psi}$) is well defined and we proved that it is necessary for solvability.

In~\cite{de:sub} we considered the case when the principal
symbol (not necessarily real valued) vanishes of at least second order at a
nonradial involutive manifold~$\st$. We assumed that the operator was of \emph{subprincipal type}, i.e., 
that the subprincipal symbol on $ \st $ is   
of principal type with Hamilton vector field tangent to~$\st$ at the
characteristics, but transversal to the symplectic leaves of~$\st$. Then
we showed that the operator is not solvable if the subprincipal 
symbol is essentially constant on the symplectic leaves of~$\st$ and does not satisfy condition (${\Psi}$), which we call Sub($ \Psi$). In the case when
the sign change is of infinite order, we also had conditions on the vanishing of both the Hessian of the principal symbol and the complex part of the gradient of the subprincipal 
symbol. 

The difference between \cite{de:compl} and \cite{de:sub} is that in the first case the Hamilton vector field of the principal symbol dominates, and in the second the Hamilton vector field of the subprincipal symbol dominates. 
In this paper, we shall study the case when condition~(${\Psi}$) is not satisfied for the refined principal symbol \eqref{refpr} which combines both the principal and subprincipal symbols. 
We shall assume that the principal symbol vanishes of at least order $k \ge 2$ on  at a
nonradial involutive manifold~$\st$.
When $k < \infty $ then the k:th jet of the principal symbol is well defined at $\st $, but since the refined principal symbol is inhomogeneous we make an inhomogeneous blowup, called \emph{reduced subprincipal symbol} by Definition~\ref{redsub}. We assume that the operator is of  \emph{subprincipal type}, i.e., the reduced subprincipal symbol is of principal type, see Definition~\ref{suprinctype}. We define condition $\sub_k(\Psi) $, which is condition~($ \Psi$) on the reduced subprincipal symbol, see Definition~\ref{subpsidef}. We assume that
the blowup of the refined principal symbol is essentially constant on the symplectic leaves of~$\st$, see~\eqref{cond1}.
We also have conditions on the rate of the vanishing of the normal gradient~\eqref{kcond}
and when $k = 2 $  of the Hessian of the reduced subprincipal symbol~\eqref{cond2}. 
When $ k= \infty$ all the Taylor terms vanish and condition $\sub_\infty (\Psi) $ reduces to condition $\sub (\Psi )$ on $\st $ from \cite{de:sub}.
Under these conditions, we show that if condition  $\sub_k(\Psi) $ is not satisfied near a bicharacteristic of the reduced subprincipal symbol then the operator is not solvable near the bicharacteristic, see
Theorem~\ref{mainthm} which is the main result of the paper. 
In the case when the sign change of  $\sub_k(\Psi) $ is on $\st $ we get a  different result than in~\cite{de:sub},  since now we localize the pseudomodes with the the phase function instead of the amplitude.

The plan of the paper is as follows. In Sect.~\ref{state} we make the definitions of the symbols we are going to use, state the conditions and the main result,  Theorem~\ref{mainthm}. In Sect.~\ref{exe} we present some examples, and in Sect.~\ref{normform} we develop normal forms of the operators, which are different in the case when the principal symbol vanishes of finite or infinite order at $ \st$. 
The approximate solutions, or pseudomodes, are defined in Sect.~\ref{modes}. In Sect.~\ref{eiksect} we solve the eikonal equation in the case when the principal symbol vanishes of finite order,  in Sect.~\ref{eta0} we solve it in the case when the bicharacteristics are on $\st $ and in Sect.~\ref{trans} we solve the transport equations. In order to solve the eikonal and transport equations uniformly we use the estimates of Lemma~\ref{lemclaim}, which is proved in Sect.~\ref{lempf}. Finally, Theorem~\ref{mainthm} is proved in Sect.~\ref{pfsect}.

\section{Statement of results}\label{state}

Let 
${\sigma}(P) = p  \in S^{m}_{\mathrm{hom}}$ be the homogeneous principal symbol of $ P$, we
shall assume that 
\begin{equation}
 {\sigma}(P) \ \text{vanishes of  at least second order at ${\Sigma}_2 \subset T^*X \setminus 0$}
\end{equation}
where
\begin{equation}
 {\Sigma}_2 \quad \text{is a nonradial involutive manifold of codimension $ d$}
\end{equation}
where $0 < d < n -1 $ with $n  = \dim X $. Here nonradial means that the radial direction
$\w{{\xi},\partial_{\xi}}$ is not in the span of
the Hamilton vector fields of the manifold, i.e., not equal to $H_{f}$ on
${\Sigma}_2$ for any $f \in C^1$ vanishing at ${\Sigma}_2$. 
Then by a change of local homogeneous symplectic coordinates we may assume that locally
\begin{equation}\label{sigma2} 
 {\Sigma}_2 = \set{\eta = 0} \qquad ({\xi}, \eta) \in
 \br^{n-d}\times \br^d \qquad \xi \ne 0
\end{equation}
for some $0 < d < n - 1$, which can be achieved by a conjugation with elliptic Fourier integral operators.

Now, since $p$ vanishes of at least second order at $\st$ we can define the \emph{order} of $p$ as
\begin{equation} \label{kdef}
2 \le \kappa(w) = \min \set{|\alpha|:\ \partial^\alpha p(w) \ne 0} \qquad w \in \st
\end{equation}
and $\kappa(\omega) = \min_{w \in \omega}  \kappa(w) $ for $\omega \subseteq \Sigma_2  $,
which is equal to $ \infty$ when $p $ vanishes of infinite order. This is an upper semicontinuous function on $\Sigma_2$, but since $\kappa(w) $ is has values in $\bn \cup \infty $, it attains its minimum $\kappa(\omega) $ on any set~$\omega \subseteq \Sigma_2 $.

If $P$ is of principal type near $\st$ then, since solvability
is an open property, we find that a necessary condition for $P$ to be
solvable at $\st$ 
is that condition $({\Psi})$ for the principal symbol is satisfied in some
neighborhood of $\st$. Naturally, this condition is empty
on $\st$ where we instead have conditions on the \emph{refined principal symbol:}
\begin{equation} \label{refprsymbol}
p_{sub} = p + p_{m-1} + \frac{i}{2}\sum_{j}^{}\partial_{x_j}\partial_{{\xi}_j}p_{}
\end{equation} 
(for the Weyl quantization, the refined principal symbol is given by $p + p_{m-1}  $).
The refined principal symbol is invariantly defined  as a
function on $T^*X$ modulo $ S^{m-2}$ under conjugation with elliptic Fourier integral operators, see \cite[Theorem~18.1.33]{ho:yellow} and \cite[Theorem~9.1]{ho:weyl}. (The latter result is for the Weyl quantization, but the result easily carries over to the Kohn--Nirenberg quantization for classical operators.)
The subprincipal symbol 
\begin{equation} \label{subsymbol}
p_{s} = p_{m-1} + \frac{i}{2}\sum_{j}^{}\partial_{x_j}\partial_{{\xi}_j}p_{}
\end{equation} 
is invariantly defined on $ \st $ under conjugation with elliptic Fourier integral operators.

\begin{rem}\label{subinv}
When $ \st = \set{\xi_1 = \xi_2 = \dots = \xi_j = 0}$ is involutive, the refined principal symbol is equal to $p_s = p_{m-1}$  at $ \st$.
\end{rem}

In fact, this follows since $ \partial_\xi p \equiv 0$ on $\st $. When composing~$P$ with an elliptic
pseudodifferential operator $C$, the value of the refined principal
symbol of $CP$ is equal to $cp_{sub} + \frac{i}{2}H_{p} c$ which is equal to $ cp_s$
at $\st$, where $c = {\sigma}(C)$.
Observe that the refined principal symbol is complexly
conjugated when taking the adjoint of the operator, see \cite[Theorem~18.1.34]{ho:yellow}.

The conormal bundle  $N^*\Sigma_2 \subset T^*(T^*X)$ of  $\Sigma_2$ is the dual of the normal bundle $T_{\st}T^*X/T\Sigma_2$. The conormal bundle can be parametrized by first choosing local homogeneous symplectic coordinates so that $ \st  $ is given by $ \set { \eta = 0}$. Then the fiber of  $N^*\Sigma_2$ can be parametrized by $ \eta \in \br^d$, $ d = \codim \st$, so that $N^*\st = \st \times \br^d$ and different parametrizations gives linear transformations on the fiber.

We define the $k$:th jet $J^k_w(f)$ of a $C^\infty$ function $f$ at $w \in \Sigma_2$ as the equivalence class of $f$ modulo functions vanishing of order $k+1$ at $w$. 
If $k = \kappa(\omega) < \infty$ is given by~\eqref{kdef} for the open neighborhood $ \omega \subset \st  $ then for $w = (x,y,\xi,0)\in \st$ we find that $J^k_w(p)$ is a well defined homogeneous function on  $N^*\Sigma_2$ given by 
 \begin{equation}
N_w^*\Sigma_2 \ni (w,\eta) \mapsto J^k_w(p)(\eta) = \partial^k_\eta p(w)(\eta)
 \end{equation}
since $\partial^{j} p \equiv 0  $ on~$\omega$, $j < k $. Here $ \partial^k_\eta p(w) $ is the $ k $-form given by the Taylor term of order $ k $ of~$ p $.
If  $ \kappa(\omega) = \infty $ then of course any jet of $ p$ vanishes identically on $ \omega$.
Here and in the following, the $ \eta $ variables will be treated as parameters.

\begin{defn}\label{redsub}
When  $k = \kappa(\omega) < \infty $ for some open set $ \omega \subset \st  $ we define the {\it reduced subprincipal symbol} by 
\begin{equation}\label{pskdef}
N^*\Sigma_2 \ni (w,\eta) \mapsto p_{s,k}(w,\eta) = J^k_w(p)(\eta) + p_s(w)  \qquad w \in \omega
\end{equation}
which is a polynomial in $ \eta $ of degree $ k $ and is given by the blowup of the refined principal symbol at $ \st$  see Remark~\ref{detarem}.  If  $\kappa(\omega) = \infty $ then we define $p_{s,\infty} =  p_{s}\restr \st $ so we have~\eqref{pskdef} for any $ k$. 
\end{defn}

\begin{rem}\label{invrem}
The reduced subprincipal symbol is well-defined up to nonvanishing factors under conjugation with elliptic homogeneous Fourier integral operators and under composition with classical elliptic pseudodifferential operator. 
\end{rem}

In fact, the reduced subprincipal symbol  is equal to the refined principal symbol modulo terms homogeneous of degree $ m$ vanishing at $\st  $ of order $k + 1 $ and terms homogeneous of degree $ m - 1$ vanishing at $\st  $. When composing with an elliptic pseudodifferential operator, both the terms in the refined subprincipal symbol gets multiplied with the same nonvanishing factor, and the terms proportional to $\partial p $ vanish on~$ \st $. Observe that if we multiply $p_{sub} $ with $ c$ then $p_{s,k} $ gets multiplied with $ c\restr \st $.

Since $p_{s}$ is only defined on~$\st$,
the Hamilton field $H_{p_{s,k}}  $ is only well defined modulo terms that are tangent to the sympletic  leaves of  $\st$, which are spanned by the Hamilton vector fields of functions vanishing on $ \st$.
Therefore, we shall assume that the reduced principal symbol essentially is constant on the leaves  of   $\st $ for fixed $\eta $ by assuming that
\begin{equation}\label{cond1a}
\left|d p_{s,k}\restr {TL} \right| \le C_0 |p_{s,k}| \qquad \text{at $\omega$  when $| \eta - \eta_0| \ll 1$ }
\end{equation}
for any leaf $L$ of $\st$ where $ \omega \subset  \st $.  Since $p_{s,k} $ is determined by the Taylor coefficients of the refined principal symbol at $ \st $ we find that \eqref{cond1a} is determined on $\st $.
When $\eta = 0 $ we get condition~\eqref{cond1a} on $p_s $ at $\st $ which was used in~\cite{de:sub}.
Condition \eqref{cond1a} is invariant under multiplication with nonvanishing factors and when $d p_{s,k} \ne 0$ on $p_{s,k}^{-1}(0) $ it is equivalent to the fact that $p_{s,k} $ is constant on the leaves up to nonvanishing factors by the following lemma.

\begin{lem}\label{constsublem}
Assume that  $ f(x,y,\zeta) \in C^\infty$ is a polynomial in $ \zeta$ of degree $ m$ for  $(x,y) \in \Omega  $, such that $\partial_x f  \ne 0 $ when $ f  = 0$. Assume that  $\Omega $ is an open bounded $C^\infty $ domain  such that $\Omega_{x_0} = \Omega \bigcap \set {x = x_0}$  is simply connected for all $ x_0$.  Let 
\begin{equation*}
\Xi = \set {(x, \zeta): \exists\, y \ (x,y) \in \Omega \text{ and } | \zeta - \zeta_0| < c}
\end{equation*}
be the projection on the $ (x,\zeta)$ variables and assume that there exist $ y_0(x) \in C^\infty$ such that $(x,y_0(x), \zeta) \in \Omega \times \set{|\zeta - \zeta_{0}|< c }$, $\forall  (x,\zeta) \in \Xi $.   Then
\begin{equation}\label{cond3}
| \partial_y f  | \le C_0 |f |  \quad \text{in $\Omega $ when $|\zeta -  \zeta_0| < c$}
\end{equation}
for $c > 0 $ implies that
\begin{equation}\label{qskmult}
f (x,y,\zeta) = c(x,y,\zeta)f_0(x,\zeta) \quad \text{in $\Omega $ when $|\zeta -  \zeta_0 | < c$}
\end{equation} 
where $0 < c_0 \le c(x,y,\zeta) \in C^\infty$ and $f_0(x,\zeta) =  f (x,y_0(x),\zeta)\in C^\infty $, which implies that $\left| \partial_y^\alpha f  \right| \le C_\alpha |f | $, $ \forall\, \alpha$,  in $\Omega $ when $|\zeta -  \zeta_0 | < c$.
\end{lem}

If $\partial_{w}p_{s,k} \ne 0  $ when $p_{s,k}= 0 $ and $ p_{s,k}$ satisfies \eqref{cond1a}, then we find from  Lemma~\ref{constsublem} after possibly shrinking $ \omega$ that $ p_{s,k}$ is constant on the leaves of $ \st$ in $ \omega$  when $| \eta - \eta_0|  < c_0 $ after multiplication with a nonvanishing factor.

\begin{proof}
Let  $ \Xi_0 =  \set {(x, \zeta) \in \Xi: f(x,y_0(x),\zeta) = 0 }$. We shall first prove the result when  $ (x, \zeta) \in \Xi \setminus  \Xi_0 $. Then $f  \ne 0$ at $(x, y_0(x),\zeta) $ and~\eqref{cond3} gives that $\partial_{y}\log f $ is
uniformly bounded near $(x, y_0(x),\zeta) $, where $\log f $ is a branch of the
complex logarithm. Thus, by integrating with respect
to~$y $ starting at $y = y_0(x)$ in the simply connected  $\Omega_x \times \set{|\zeta - \zeta_{0}|< c }$
we find that $\log f (x,y,\zeta) - \log f (x,y_0(x),\zeta)  \in C^ \infty$ is bounded  and by exponentiating we obtain
\begin{equation}\label{factcond}
f (x,y,\zeta) = c (x,y,\zeta)f_0( x,\zeta) \qquad \text {in $ \Omega_x$ for $| \zeta  -  \zeta_0 | < c$}
\end{equation}
when  $ (x, \zeta) \in \Xi \setminus  \Xi_0 $. Here  $f_0(x,\zeta) =  f (x,y_0(x),\zeta)\in C^\infty $ and $0 < c_0 \le  c(x,y,\zeta) \in C^\infty$ is uniformly bounded such that $c(x, y_0(x),\zeta) \equiv 1$.  This gives that $f ^{-1}(0)$ is constant in $y$ when $(x, \zeta) \notin \Xi_0 $.

Since  $\partial_x f  \ne 0 $ when $ f    = 0$ we find that $ \Xi_0 $ is nowhere dense.  Let $(x_0, \zeta_0) \in \Xi_0 $ and choose $z \in \bc $ such that  $ \partial_x\re zf (x_0,y_0(x_0),\zeta_0) \ne 0 $. Let  $  S_\pm = \set{ \pm \re zf (x,y_0(x),\zeta)  > 0} $  then
\begin{equation}\label{pmeq}
f (x,y,\zeta) = c_\pm(x,y,\zeta) f (x,y_0(x),\zeta) \quad \text{in $S_\pm $}
\end{equation}
where  $0 < c_0 \le  c_\pm (x,y,\zeta) \in C^\infty$ is uniformly bounded.
By taking the limit of  \eqref{pmeq} at $S =   \set{\re zf (x,y_0(x),\zeta)  = 0} $ we find that $ c_+   =  c_- $ when $  f  \ne 0 $ at $ S$. When  $  f   = 0$ at $ S$ then by differentiating  \eqref{pmeq}  in $ x$ we find that  $ c_+   =  c_- $.
By repeatedly differentiating \eqref{pmeq}  in $ x$ we obtain by recursion that  $c_\pm $ extends to  $c_0 < c(x,y,\zeta) \in  C^\infty $ in $\Omega $ when  $| \zeta  -  \zeta_0 |< c$ so that \eqref{qskmult} holds. 
\end{proof}

If $p_{s,k} $ is  constant in $y $ in a neighborhood of the semibicharacteristic, then the Hamilton field $H_{p_{s,k}}  $ will be constant on the leaves and defined modulo tangent vector to the leaves. Therefore we shall introduce a special symplectic structure on $ N^*\st$. 
Recall that the symplectic annihilator to a linear space consists of  the vectors that are symplectically orthogonal to the space. 
Let  $T{\Sigma}_2^{\sigma}$ be the symplectic annihilator to $T
{\Sigma}_2$, which spans the symplectic leaves of $\st$. If $\st =
\set{\eta = 0}$, 
$(x,y) \in \br^{n-d}\times \br^{d}$, then the leaves are
spanned by  $\partial_{y}$.  Let
\begin{equation}
T^{\sigma}{\Sigma}_2 = T{\Sigma}_2/T{\Sigma}_2^{\sigma}
\end{equation}
which is a symplectic space over ${\Sigma}_2$ which in these
coordinates is parametrized by
\begin{equation}\label{tsigmadef}
T^{\sigma}\st =  
\set{((x_0,y_0;{\xi}_0,0); (x,0;{\xi}, 0)) \in T \st:\ (x,{\xi}) \in T^*\br^{n-d}}
\end{equation} 
This is isomorphic to the symplectic manifold $T^*\br^{n-d} $ with $ w \in \st $ as parameter.

We define the symplectic structure of $ N^*\st $ by lifting  the structure of $\st$ to the fibers, so that the leaves of $ N^*\st $ are given by $L \times \set {\eta_0} $ where $ L$ is a leaf of  $ \st$ for $\eta_0 \in \br^d $.
In the chosen coordinates, these leaves are parametrized by  $ \set{(x_0,y;{\xi}_0,0)\times \set {\eta_0}:\ y \in \br^{d}}$.
The radial direction in  $N^*\st $ will be the radial direction  in $\st $, i.e.\ $\w{\xi,\partial_\xi} $, lifted to the fibers.
Similarly, a vector field $ V \in T(N^*\st) $ is parallel to the base of $N^*\st $ if  it is in $T\st  $, which means that  $V \eta = 0$. 

If $ p_{s,k}$ is constant in $ y$ then $H_{p_{s,k}}$ coincides with the Hamilton vector field of $p_{s,k}$ on $p_{s,k}^{-1}(0) \subset  N^*\st $ with respect the symplectic structure on the symplectic manifold
$N^*\st $. In fact, in the chosen coordinates we obtain from~\eqref{cond1a} that
\begin{equation}\label{hpsk}
H_{p_{s,k}} = 
\partial_{\xi}p_{s,k} \partial_{x} - \partial_{x}p_{s,k}
\partial_{{\xi}}
\end{equation}
modulo $\partial_{y}$, which is nonvanishing if
$\partial_{x,{\xi}}p_{s,k}  \ne 0$. Thus $H_{p_{s,k}}$ is well-defined  modulo terms containing
$\partial_{y}$  making it well defined on $T^{\sigma}{\Sigma}_2 \times \br^d$.
Now, if $p_ {s,k}= 0$ then by~\eqref{cond1a} we find that $ d p_{s,k}\restr {T\st}$ vanishes on $T{\Sigma}_2^{\sigma}$ so $dp_{s,k}\restr {T\st}$
is well defined on $T^{\sigma}{\Sigma}_2 $. 
We may identify $T (N^*{\Sigma}_2)$ with $T \st \times \br^d $ since the fiber $\eta $ is linear.

\begin{defn}\label{suprinctype}
We say that the operator $P$ is of \emph{subprincipal type} on $ N^* \st $ if the
following hold when $p_{s,k} = 0$ on $N^*\st$: $H_{p_{s,k}}
$ is parallel to the base,  
\begin{equation}\label{subprinc}
d p_{s,k}\restr{T^{\sigma}{\Sigma}_2} \ne 0 
\end{equation}
and the corresponding Hamilton vector field  $H_{p_{s,k}}$ of~\eqref{subprinc}
does not have the radial direction. 
The (semi)bicharacteristics of $ p_{s,k} $ with respect to the symplectic structure of $ N^* \st $ are called the {subprincipal (semi)bicharacteristics}.
\end{defn}

Clearly, if coordinates are chosen so that~\eqref{sigma2} holds, then \eqref{subprinc} gives that $\partial_{x,{\xi}}p_{s,k} \ne 0$ when $ p_{s,k} = 0$ and the condition that the Hamilton vector field does not have the radial direction means that $ \partial_{\xi}p_{s,k} \ne 0 $ 
or $\partial_{x}p_{s,k} \nparallel {\xi}$ when $p_{s,k} = 0$.
Because of~\eqref{subprinc} we find that $H_{p_{s,k}} $ is transversal to the foliation of $N^*\st $ and by~\eqref{cond1a} it is parallel to the base at the characteristics.  
The semibicharacteristic of $ p_{s,k}$ can be written  $\Gamma = \Gamma_0 \times \set{\eta_0} \subset T(N^*\st)$, where $\Gamma_0 \subset \st$ is transversal to the leaves of~$\st$ and $\eta_0$ is fixed.
The definition can be localized to  an open set $ \omega  \subset N^*\st$. It is a generalization of the definition of subprincipal type in~\cite{de:sub}, which is the special case when $ \eta = 0$.
When $ P$ is of subprincipal type and satisfies \eqref{cond1a}, then we find from Lemma~\ref{constsublem}  that $ p_{s,k}$ is constant on the leaves of $ \st$ near a semibicharacteristic after multiplication with a nonvanishing factor.
We can now state a condition corresponding to ($ \Psi$) on the reduced subprincipal symbol.

\begin{defn}\label{subpsidef}
If $k = \kappa(\omega) $ for an open set $ \omega \subset N^*\st  $, then
we say that $P$ satisfies condition $\sub_k({\Psi})$ if $\im a p_{s,k}$ does not change sign from $-$ to $+$ when going in the 
positive direction on the subprincipal bicharacteristics of $\re
a p_{s,k}$ in~$ \omega$   for any $0 \ne a \in C^\infty$. 
\end{defn}

Observe that when $k < \kappa(\omega)  $  or $k = \kappa(\omega)  = \infty $ then  $p_{s,k} = p_s\restr \st$ on $\omega $ and $\sub_k({\Psi})$  means that the subprincipal symbol $p_s $ satisfies condition $(\Psi) $ on $T^\sigma \st$, which is condition  $\sub({\Psi})$ in~\cite{de:sub}. 
In general, we have that condition $\sub_k({\Psi})$  is condition (${\Psi}$) given by~\eqref{psicond} on the reduced subprincipal symbol $ p_{s,k} $ with respect to the symplectic structure of $N^*\st $. 
But it is equivalent to the condition  (${\Psi}$) on the reduced subprincipal symbol $ p_{s,k} $ with respect to the standard  symplectic structure. In fact, condition $\sub_k({\Psi})$ means that condition   
(${\Psi}$) holds for $ p_{s,k} \restr {\eta = \eta_0} $ for any $\eta_0 $. 
By using Lemma \ref{constsublem} we may assume  that $ p_{s,k} $ is independent of~$ y$ after multiplying with $ 0 \ne a \in C^\infty$. In that case, the conditions are equivalent and
both are invariant under multiplication with nonvanishing smooth factors.

By the invariance of  condition (${\Psi}$) given by~\cite[Theorem~26.4.12]{ho:yellow} it suffices to check condition  $\sub_k({\Psi})$ for some $ a$ such that $H_{\re ap_{s,k}} \ne 0$.
We also find that condition $\sub_k({\Psi})$  is invariant under symplectic
changes of variables,
thus it is invariant under conjugation of the operator by elliptic
homogeneous Fourier integral operators. Observe that the sign
change is reversed when taking the adjoint of the operator.

Next, we assume that condition  $\sub_k({\Psi})$  is not satisfied on a
semibicharacteristic ${\Gamma}$ of $p_{s,k}$, i.e., that $\im ap_{s,k}$ changes
sign from $-$ to $+$ on the positive flow of $H_{\re
a p_{s,k}} \ne 0$ for some  $0 \ne a \in C^\infty$, where  $\eta $ is constant on $ \Gamma$. Thus, by Lemma~\ref{constsublem} we may assume that $p_{s,k} $ is constant on the leaves in a  neighborhood $\omega $ of $ \Gamma$, and by multiplying with $ a$ we may assume that $a \equiv 1 $ and that $ y$ is constant on the semibicharacteristic.

\begin{defn}\label{def}
Let $p $ be of subprincipal type on $ N^* \st $ and $\Gamma $ a subprincipal semibicharacteristic of $ p$.
We say that a $ C^\infty $ section of
spaces $L \subset T (N^*\st )$ is \emph{gliding} for $\Gamma $ if  $L $ is symplectic  of maximal dimension  $2n - 2(d+1) \ge 2$ so that $ L$ is the symplectic annihilator  of $T\,\Gamma$ and the foliation of $\st $, which gives  $L \subset  T\, \st $ since $\eta $  is constant on $ L$.
We say that a $  C^\infty $ foliation of $N^*  \Sigma_2$  with symplectic leaves  $ M $ is gliding  for $\Gamma $  if the section of tangent spaces $TM $ is a gliding section for~$\Gamma$.
\end{defn}

Actually, he gliding 
foliation $ M$ for a   subprincipal semibicharacteristic $ \Gamma$  is uniquely defined near $\Gamma$, since it is determined by the unique annihilator $TM $ and $\Gamma$ is transversal to the foliation of $\st $ when $p= 0 $ by~\eqref{subprinc}. This definition can be localized to a neighborhood of  a subprincipal semibicharacteristic.

\begin{exe}
Let $p $ be  of subprincipal type on $ N^* \st $. Assume that
$ \st = \set {\eta = 0}$, $\partial_y p = \set {\eta, p } = 0$ and $\partial_{x, \xi}$ spans $ TM$ of the gliding foliation $M $ of $N^* \st $ for the bicharacteristic of $H_{\re p} \ne 0$. Then we may complete $x $, $ \xi $, $\tau = \re p $ and $ \eta $ to a symplectic coordinate system $(t,x,y;\tau, \xi, \eta) $ so that the foliation $ M$ is given by intersection of the level sets of $\tau$, $t$ ,  $y $ and $ \eta$. In fact, in that case we have $\partial \re p \ne 0 $ but $\partial_x\re  p =  \partial_\xi \re  p = 0$.
\end{exe}

In the case when $ \eta_0 \ne 0$ and $ k = \kappa(\omega)< \infty $ we will have estimates on the rate of
vanishing of $ \partial_\eta p_{s,k} $ on the subprincipal semibicharacteristic.  Recall that the semibicharacteristic can be written $\Gamma \times \set {\eta_0} $. Observe that  
\begin{equation}
\partial_\eta p_{s,k} =\Cal J^{k-1}(\partial_\eta  p) = \Cal J^{k-1}(\partial  p) 
\end{equation} 
since $ p$ vanishes of at least order $ k$ at $\st $ and that  the normal derivatives $\partial_\eta $ is well-defined modulo nonvanishing factors at $\eta = 0 $.  Let  $\omega \subset \st $ be a neighborhood of  the subprincipal semibicharacteristic $ \Gamma$ and let  $ M$ be the local $  C^\infty $  foliation of $N^*\st  $ at $ \omega$ which is  gliding for $\Gamma$. 
When $\eta_{0} \ne 0 $ we shall assume that there exists $ \varepsilon > 0$ so that
\begin{equation}\label{kcond}
\left|  V_1\cdots V_\ell \, \partial_\eta p_{s,k} \right| \le C_\ell | p_{s,k} |^{{1}/{k} + \varepsilon} \qquad \text{on $\omega$  when $| \eta - \eta_0| \ll 1$ } 
\end{equation}
for any vector fields  $V_j \in TM$, $0 \le j \le \ell $ and any  $\ell $.
Condition  \eqref{kcond} gives that  $V_1\cdots V_\ell \, \partial_\eta p_{s,k} $ vanishes when $ p_{s,k} = 0$.
This definition is invariant under symplectic changes of coordinates and multiplication with nonvanishing factors.
Observe that we have $ V_1\cdots V_\ell \, \partial_\eta p_{s,k} = 0$ when $ \eta_0 = 0$ since then $p =  \partial_\eta  p = 0$ and $ V_j \in TM \subset T\st $.  
Condition \eqref{kcond} with $\ell = 0 $ gives that $\eta \mapsto |p_{s,k}(w,\eta)|^{{(k-1)}/{k} - \varepsilon} $ is Lipschitz continuous, thus $\eta \mapsto p_{s,k}(w,\eta) $ vanishes at $\eta_{0} $ of order 3 when $ k = 2$ and  order 2 when $ k > 2$.

In the case $k  = \kappa(\omega) =2$ we shall also have a similar condition on the rate of
vanishing of $ \partial^2_\eta p_{s,k} $ on the subprincipal semibicharacteristic. 
Then 
\begin{equation}\label{hessdef}
\partial^2_\eta p_{s,k} = \Cal J^{0}(\partial^2  p) = \hess p\restr \st
\end{equation}  
is the Hessian of the principal symbol  $ p$ at $\st $, which is well defined on
the normal bundle $N \st$ since it vanishes on $T\st$.
Since $p = \partial_\eta p = 0 $ on $ \st$,  we find that $\hess p$ is invariant  modulo nonvanishing smooth factors  under symplectic 
changes of variables and multiplication of $P$ with elliptic
pseudodifferential operators.  
With the gliding  $  C^\infty $ foliation $M$ of $N^* \st $ for $\Gamma$ 
we shall assume that there exists
${\varepsilon} > 0$ so that
\begin{equation}\label{cond2}
\mn{V_1\cdots V_\ell \, \hess p} \le C'_\ell | p_{s,k}|^\varepsilon
 \qquad \text{on $\omega$}
\end{equation}
for any vector fields  $V_j \in TM $, $0 \le j \le \ell $  and any  $\ell $.
This definition is invariant under symplectic changes of coordinates and multiplication with nonvanishing factors.

\begin{rem}
Conditions \eqref{kcond} and \eqref{cond2} are well defined and invariant under multiplication with elliptic pseudodifferential operators and conjugation with elliptic Fourier integral operators.
\end{rem}

Examples~\ref{Miz}--\ref{CPTgen} show that conditions~\eqref{kcond} and~\eqref{cond2} are essential for the 
necessity of  $\sub_k({\Psi})$ when $ k= 2$. 
 
\begin{exe}\label{finex}
 If  $\re p_{s,k} = \tau $, $\st = \set {\eta = 0} $, $TM $ is spanned by $\partial_{x,\xi} $ and $t \mapsto \im p_{s,k} $ vanishes of order $3 \le  \ell < \infty $ at $t = t_0(y,\eta) \in C^\infty$ then~\eqref{kcond}  and~\eqref{cond2} hold. 
 If $t_0(y) $ is independent of $\eta $ then conditions~\eqref{kcond}   and~\eqref{cond2} hold for any finite $\ell > 0 $.
 \end{exe}

 In fact, if $0 < \ell < \infty  $ then we can write $\im p_{s,k} = a(t-t_0(y,\eta))^\ell $ with $a \ne 0 $.  If $\ell > \frac{k}{k-1} $ then for any $\alpha$ we find that  $\partial_{x,\xi}^\alpha  \partial_\eta \im p_{s,k}$ vanishes of order $\ell - 1  > \ell/k $ at $ t = t_0$, and if $\ell > 2 $ then $\partial_{x,\xi}^\alpha  \partial_\eta^2 \im p_{s,k}$ vanishes of order $\ell - 2  > 0$ at $ t = t_0$. 
 If $ t_0$ is independent of $ \eta $ then $\partial_{x,\xi}^\alpha   \partial_\eta^j \im p_{s,k}$ vanishes of order $\ell  $ for any $ j$ and $ \alpha$.

Since $ \partial_\eta p_{s,k}$ is homogeneous of degree $ k-1$ in $ \eta$, we find from Euler's identity that $\partial_\eta p_{s,k}(w,\eta) = (k-1)\, \eta \cdot \hess p(w,\eta)$. Thus~\eqref{cond2} implies that   $| V_1\cdots V_\ell \, \partial_\eta p_{s,k}| \ls | p_{s,k}|^\varepsilon$ when $\eta \ne 0 $, 
 but we shall only use condition~\eqref{cond2}  when~\eqref{kcond} holds, see Theorem~\ref{mainthm}. 
 Here $a \ls b $ means $a \le Cb  $ for some constant $ C$, and similarly for $a \gs b $.

Now, by \eqref{cond1a} we have assumed that the reduced  subprincipal symbol~$p_{s,k}$ is constant on the leaves of~$\st$ near $\Gamma $
up to multiplication with nonvanishing factors, but  when $\kappa < \infty $ we will actually have that condition on the following symbol.

\begin{defn}\label{defqsk}
If $ k  = \kappa(\omega) < \infty$ is the order of $ p$ on an  open set $\omega \subseteq \st $ then we define the \emph{extended subprincipal symbol} on  $N^*\omega$ by
\begin{multline}\label{qsk}
N^*\Sigma_2 \ni (w,\eta) \mapsto q_{s,k}(w,\eta,\lambda)  =  \lambda J^{2k-1}_w(p)(\eta/ \lambda^{1/k} )  \\ + 
 J^{k-1}_w(p_{s})(\eta/ \lambda^{1/k}) 
  \cong  p_{s,k}(w,\eta) +  \Cal O (\lambda^{-{1}/{k}}) 
\end{multline}
which is a weighted polynomial in $ \eta $ of degree $ 2k -1$. 
When $ \kappa(\omega) = \infty $ we define $ q_{s, \infty} \equiv p_{s} $. 
\end{defn}

By the invariance of $ p$ and $ p_s$, the extended subprincipal symbol transforms as jets under homogeneous symplectic changes of coordinates that preserve the base~$\st$. It is well defined up to nonvanishing factors and terms proportional to the jet $ J^{k-1}_w(\partial_\eta p )(\eta/ \lambda^{1/k}) \cong \lambda^{{1}/{k} -1}\partial_\eta^k p $ modulo $\Cal O(\lambda^{-1} )$ under multiplication with classical elliptic pseudodifferential operators.
The extended and the reduced subprincipal symbols are complexly conjugated when taking adjoints.

\begin{rem}\label{detarem}
The extended subprincipal symbol \eqref{qsk} is given by the blowup of the reduced principal symbol at $ \eta = 0$ so that
\begin{multline}\label{qskexp}
\lambda^{2 -m} p_{sub}(x,y;  \lambda  \xi, \lambda^{1 - {1}/{k}} \eta) \cong   \lambda q_{s,k}(x,y;\xi,\eta,\lambda) \\ =  \lambda p_{s,k}(x,y;\xi,\eta) +  \Cal O (\lambda^{1-{1}/{k}}) \qquad \text{modulo $\Cal O(1)$}
 \end{multline}
We also have that
\begin{multline}
\lambda^{2-m} \partial_\eta p_{sub}(x,y;  \lambda  \xi, \lambda^{1 - {1}/{k}} \eta) \cong  \lambda^{{1}/{k}} \partial_\eta q_{s,k}(x,y;\xi,\eta,\lambda) \\ =  \lambda^{{1}/{k}} \partial_\eta p_{s,k}(x,y;\xi,\eta) +  \Cal O (1) 
 \end{multline}
modulo $\Cal O( \lambda^{{1}/{k} - 1})$ and
\begin{multline}
\lambda^{2 -m} \partial_\eta^2 p_{sub}(x,y;  \lambda  \xi, \lambda^{1 - 1/k} \eta) \cong  \lambda^{{2}/{k} -1} \partial_\eta^2 q_{s,k}(x,y;\xi,\eta,\lambda)  \\  =  \lambda^{{2}/{k} -1} \partial_\eta^2 p_{s,k}(x,y;\xi,\eta) + \Cal O( \lambda^{{1}/{k} - 1})
\end{multline}
modulo $\Cal O( \lambda^{{2}/{k} - 2})$.
Observe that if $ $P  is of subprincipal type then $ d q_{s,k}\restr{T^{\sigma}{\Sigma}_2} \ne 0  $ when  $q_{s,k} = 0$ for $ \lambda \gg 1 $ since this holds for $p_{s,k}$. 
\end{rem}

In fact, $d q_{s,k} \cong  d p_{s,k} $ modulo  $\Cal O(\lambda^{-1/k}) $ and since $| d p_{s,k}| \ne 0  $ the distance between $q_{s,k}^{-1}(0) $ and  $p_{s,k}^{-1}(0) $ is  $\Cal O(\lambda^{-1/k}) $ for $ \lambda \gg 1$.
Observe that composition of the operator $ P$ with elliptic pseudodifferential operators gives factors proportional to $ J^{k-1}_w(\partial_\eta p )(\eta/ \lambda^{1/k}) $ which we shall control with~\eqref{kcond}.

By \eqref{kcond} we have that $\partial_{{\eta}} p_{s,k} = 0 $ when $p_{s,k} = 0$ at $\omega $. We shall also assume this for the next term in the expansion of $ q_{s,k}$,
\begin{equation}\label{dqcond}
\partial_{{\eta}}   q_{s,k} = \Cal O(\lambda^{-2/k}) \text{ when $p_{s,k} = 0$ at $ \omega$ for $| \eta - \eta_0|  < c_0 $  and $\lambda \gg 1 $}
\end{equation}
Actually, we only need this where $p_{s,k}\wedge d \ol p_{s,k} $ vanishes of infinite order at  $p_{s,k}^{-1}(0)$ in $\omega $, where $d p_{s,k}\wedge d \ol p_{s,k} $ is the complex part of $d p_{s,k} $.

We shall also assume a condition similar to~\eqref{cond1a} on  the extended subprincipal symbol
\begin{equation}\label{cond1}
\left|d q_{s,k}\restr {TL} \right| \le C_0 |q_{s,k}| \qquad \text{at $\omega$  when $| \eta - \eta_0|  < c_0 $  and $\lambda \gg 1 $}
\end{equation}
for any leaf $L$ of $N^* \st$ where $\omega \subset \st  $. 
By letting $ \lambda \to \infty$ we obtain that~\eqref{cond1a} holds, since $q_{s,k} \cong p_{s,k}$ modulo  $ \Cal O(\lambda^{-1/k})$.  Also, 
multiplication of $ p_{sub}$ by $a \cong a_0 + a_{-1} + \dots$  with $a_j $ homogeneous of degree~$ j$ and $a_0 \ne 0 $ gives that $q_{s,k} $ gets multiplied by the expansion of $\eta \mapsto a_0(x,y;\xi,\eta/\lambda^{1/k}) $ since $a p_{sub} \cong a_0 p_{sub} + a_{-1}p  \cong a_0 p_{sub} $ modulo terms in $ S^{m-1}$ vanishing of order $ k$ at $ \st $. Thus, condition~\eqref{cond1} is invariant under multiplication  of $ p_{sub}$ with classical elliptic symbols.
Also,  \eqref{cond1} is invariant under changes of homogeneous symplectic coordinates that preserves $\st = \set {\eta = 0 } $ and $TL $. 
Now, we have $ \partial_{x,\xi} q_{s,k} \ne 0  $ when $q_{s,k} = 0$ and $ \lambda \gg 1$ since $ P$ is of subprincipal type. 

\begin{rem}\label{qskrem}
Since the semibicharacteristic is transversal to the leaves of $\st $ and condition~\eqref{cond1}  holds near the semibicharacteristic, Lemma~\ref{constsublem} gives that 
\begin{equation}\label{qconst}
q_{s,k}(x,y;\xi,\eta,\lambda) = c(x,y;\xi,\eta,\lambda) \wt  q_{s,k}(x;\xi,\eta,\lambda) \qquad \lambda \gg 1
\end{equation}
for $| \eta - \eta_0|  < c_0 $ near the semibicharacteristic.
Here $\wt  q_{s,k}(x;\xi,\eta,\lambda)$ is the value of $ q_{s,k}$ at the intersection of the semibicharacteristic and the leaf.
In fact, the proof of the lemma extends to symbols depending uniformly on the parameter $0 < \lambda^{-1/k} \ll 1$.

Condition~\eqref{cond1} is \emph{not} invariant under multiplication of $P$ with elliptic pseudodifferential operators or conjugation with elliptic Fourier integal operators. In fact, if $A$ has symbol $a$ then the refined principal symbol of the composition $AP$ is equal to $ap_{sub} + \frac{1}{2i}\set{a, p}$ which adds $ \frac{i}{2}\lambda^{{1}/{k} - 1} \partial_y a \partial_\eta p_{s,k} $ to $q_{s,k}$.  But \eqref{dqcond} is invariant, since the term containing the factor $ \partial_\eta p_{s,k} $ is $\Cal O(\lambda^{-2/k}) $ when $k > 2 $ and has vanishing $\eta $ derivative  at $ p_{s,k}^{-1}(0)$ by \eqref{cond2} when $k = 2 $.
\end{rem}

This is one reason why we have to control the terms with $\partial_\eta p_{s,k} $ with~\eqref{kcond}.
When $ k < \infty$, $ q_{s,k}$ is a polynomial in $\eta/\lambda^{1/k}$ of degree $2k-1$ and $c$ in \eqref{qconst} is
an analytic function in $\eta/\lambda^{1/k}$ on $\omega$ when $| \eta - \eta_0|  < c_0 $. Actually, it suffices to expand $ c$  in $\eta/\lambda^{1/k}$ up to order $ k$ in order to obtain~\eqref{qconst} modulo $\Cal O(\lambda^{-1}) $. 
If $ C(x,y;\xi, \eta/ |\xi|^{1/k}) = c(x,y;\xi, \eta, |\xi|)  $ in~\eqref{qconst} then we obtain that $Cp_{sub} $ is constant in $ y$  modulo $ S^{m-2}$ in $ \omega$ when $\big |\eta - \eta_0  |\xi|^{1 - {1}/{k}} \big | < c_0 |\xi|^{1 - {1}/{k}} $.

In the case when the principal symbol~$p$ is real, a necessary
condition for solvability of the operator is that the
imaginary part of the subprincipal symbol does not change 
sign from $-$ to $+$ when going in the positive direction on a $ C^\infty $ limit
of normalized bicharacteristics of the principal symbol~$p$ at $ \st $,
see~\cite{de:limit}. When $p$ vanishes of exactly order $k$ on~$\st  = \set{\eta =0}$
and the
localization 
\begin{equation*}
{\eta} \mapsto \sum_{|{\alpha}| = k}
\partial_{\eta}^{\alpha} p(x,y;0,{\xi}) {\eta}^{\alpha}/\alpha!
\end{equation*}
is of principal type when ${\eta} \ne 0$ such limit
bicharacteristics are tangent to the leaves of~$\st$. 
In fact, then 
$
|\partial_{\eta} p(x,y;{\xi},\eta)|
\cong |\eta|^{k-1}
$
and $|\partial_{x,y,{\xi}} p(x,y;{\xi},\eta)| = \Cal
O(|\eta |^{k})$, which gives $H_p = \partial_{\eta} p \partial_{y} +
\Cal O(|\eta|^k)$. Thus the normalized Hamilton vector field is equal to $A\partial_{y}$, $ A \ne 0 $, modulo terms that are $ \Cal O(|\eta|)$, so the normalized Hamilton vector fields have limits that are tangent to the leaves.
That the $ \eta$ derivatives dominates $\partial p $ can also be seen from Remark~\ref{detarem}.
When the principal symbol is proportional to a real valued symbol, this
gives examples of nonsolvability when the subprincipal symbol is not
constant on the leaves of~$\st$, see Example~\ref{etaex} and~\cite{de:limit} in general. Thus condition~\eqref{cond1} is natural for the the study of the  necessity of  $\sub_k({\Psi})$ if there are no other conditions on the principal symbol.

We shall study the microlocal solvability of the operator, which is
given by the following definition. Recall that $H^{loc}_{(s)}(X)$ is
the set of distributions that are locally in the $L^2$ Sobolev space
$H_{(s)}(X)$. 

\begin{defn}\label{microsolv}
If $K \subset S^*X$ is a compact set, then we say that $P$ is
microlocally solvable at $K$ if there exists an integer $N$ so that
for every $f \in H^{loc}_{(N)}(X)$ there exists $u \in \Cal D'(X)$ such
that $K \bigcap \wf(Pu-f) = \emptyset$. 
\end{defn}

Observe that solvability at a compact set $K \subset X$ is equivalent
to solvability at $S^*X\restr K$ by~\cite[Theorem 26.4.2]{ho:yellow},
and that solvability at a set implies solvability at a subset. Also,
by~\cite[Proposition 26.4.4]{ho:yellow} the microlocal solvability is
invariant under conjugation by elliptic Fourier integral operators and
multiplication by elliptic pseudodifferential operators.
We can now state the main result of the paper.

\begin{thm}\label{mainthm}
Assume that $P \in {\Psi}^m_{\mathrm{cl}}(X)$ has principal symbol that
vanishes of at least second order at a nonradial involutive manifold $\st \subset T^* X \setminus 0$.  We  
assume that $ P$  is of subprincipal type, satisfies conditions~\eqref{dqcond} and~\eqref{cond1} but does not satisfy condition $\sub_k({\Psi})$  near the subprincipal semibicharacteristic~$ \Gamma \times \set {\eta_0}$ in $N^*\st $ 
where $\Gamma  \subset \omega \subset \st$ and  $k = \kappa(\omega) $.  

In the case when $ \eta_0 \ne 0$  we assume that $ P$ satisfies conditions~\eqref{kcond} and when $k=2$  we also assume condition~\eqref{cond2} for a gliding symplectic foliation~$ M$ of $N^*\st $ for the subprincipal semibicharacteristic. 

In the case $ \eta_0 = 0$ and
$k = 2 $ we assume condition \eqref{cond2}  for a gliding symplectic foliation $ M$  of $N^*\st $  for  the subprincipal semibicharacteristic, and when $ k > 2$ we assume no extra condition.

Under these conditions, $P$ is not locally solvable near\/~${\Gamma} \subset \st$.
\end{thm}

Examples \ref{Miz}--\ref{CPTgen} show that conditions~\eqref{kcond} and~\eqref{cond2} are essential for the necessity of  $\sub_k({\Psi})$ when $ k= 2$. 
Due to the results of~\cite{de:limit}, condition~\eqref{cond1} is natural if there are no other conditions on the principal symbol, see Example~\ref{etaex}. 
Observe that for effectively hyperbolic operators, which are always solvable, $ \st$ is not an involutive manifold, see Example~\ref{effhyp}.

\begin{rem}\label{thmrem}
It follows from the proof that we don't need condition~\eqref{dqcond} in the case when condition \eqref{cond1} holds on the leaves of $\st $ that intersect the semibicharacteristic.
In the  case when $\eta = 0 $ on the subprincipal semibicharacteristics,  condition~\eqref{cond2} only involves $\hess p $ at $ \st $. This gives a  different result than Theorem 2.7 in~\cite{de:sub}, since in that result  condition~\eqref{cond2} is not used, condition \eqref{cond1} only involves $ p_s$ but we also have conditions on $| dp_s \land d\ol {p}_s | $ and $\hess p $ on $ \st$.
\end{rem}

Now let $S^*X
\subset T^*X$ be the cosphere bundle where $|{\xi}| = 1$, and let
$\mn{u}_{(k)}$ be the $L^2$ Sobolev norm of order $k$ for $u \in C_0^\infty$.
In the following, $P^*$ will be the $L^2$ adjoint of $P$.
To prove Theorem~\ref{mainthm} we shall use the following result.

\begin{rem}\label{solvrem}
If $P$ is microlocally solvable at ${\Gamma}\subset S^*X$,
then Lemma 26.4.5 in~\cite{ho:yellow} gives that for any $Y \Subset
X$ such that ${\Gamma} \subset S^*Y$ there exists an integer ${\nu}$
and a pseudodifferential operator $A$ so that
$\wf(A) \cap {\Gamma} = \emptyset$ and
\begin{equation}\label{solvest}
 \mn {u}_{(-N)} \le C(\mn{P^*{u}}_{({\nu})} + \mn {u}_{(-N-n)} +
 \mn{Au}_{(0)}) \qquad u \in C_0^\infty(Y)
\end{equation}
where~$N$ is given by Definition~\ref{microsolv}.
\end{rem}

We shall prove Theorem~\ref{mainthm} in Sect.~\ref{pfsect} by
constructing localized  
approximate solutions to $P^*u \cong 0$ and
use~\eqref{solvest} to show that $P$ is 
not microlocally solvable at~${\Gamma}$.

\section{Examples}\label{exe}

\begin{exe}\label{Miz}
Consider the operator
\begin{equation}\label{Mizeq}
P = D_t + i a(t)\Delta_y  \qquad (x,y) \in  \br^{n-d}  \times \br^{d}  
\end{equation}
where $0 < d < n $, $a(t) $ is real and has a sign change from $ -$ to $ +$.
This operator is equal to the Mizohata operator when $a(t) = t$.
We find that $P$ is of subprincipal type,  $k = 2$ and $ p_{s,2}(t,\tau,\eta) = \tau + i a(t)|\eta |^2 $ is constant on the leaves of $\st = \set {\eta = 0} $. Condition  \eqref{cond1} hold but $ \sub_2 (\Psi)$ does not hold since $t\mapsto a(t) |\eta|^2$ changes sign  from $ - $ to $ + $ when $\eta \ne 0$. 
Since $ |\partial_{\eta} p_{s,2}| \cong \mn{\hess p_{s,2}} \cong |a(t)|$ when $\eta \ne 0$ and 
$p_{s,2} $ is independent of $(x,\xi) $ we find that conditions~\eqref{kcond} and \eqref{cond2} hold. 
Theorem~\ref{mainthm} gives that $P$ is not locally solvable. 
\end{exe}

\begin{exe}\label{CPT}
The operator 
\begin{equation}\label{CPTeq}
P = D_t + i(D_{x_1}D_{x_2} + t D_{x_2}^2) \qquad x \in \br^{n} \quad n \ge 3
\end{equation}
is solvable, see \cite{CPT}. 
We find that $P$ is of subprincipal type, $k = 2$, $\st = \set {\xi_1 = \xi_2 = 0} $ and $ p_{s,2}(t,\tau,\xi) = \tau + i(\xi_1\xi_2 + t \xi_2^2)$. Condition $ \sub_2 (\Psi) $ does not hold since $ t \mapsto \xi_1\xi_2 + t \xi_2^2$ changes sign  from $ - $ to $ + $ when $ \xi_1 = - t\xi_2 $ and  $\xi_2 \ne 0$.        
Since $ |\partial_{\xi} p_{s,2}| \cong \mn{\hess p_{s,2}} \cong 1 \gg |p_{s,2}|\cong |t|$ when  $\xi_2 \ne 0$ and  $\tau = \xi_1 = 0 $, we find that conditions~\eqref{kcond} and \eqref{cond2} do not hold. 
\end{exe}

\begin{exe}\label{CPTgen}
Consider the following generalization of Example~\ref{CPT} given by
\begin{equation}\label{CPTgeneq}
P = D_t + i(D_{x_1}D_{x_2} + t^{2j + 1} D_{x_2}^2) + i (2j^2 + j)  t^{2j - 1} x_1^2 \qquad x \in \br^{n} 
\end{equation}
for $ j > 0$ and $  n \ge 3$. We find that $P$ is of subprincipal type, $k = 2$, $\st = \set {\xi_1 = \xi_2 = 0} $ and $ p_{s,2}(t,\tau,\xi) = \tau + i(\xi_1\xi_2 + t^{2j + 1} \xi_2^2)$. Thus $ \sub_2 (\Psi) $ does not hold since $ t \mapsto \xi_1\xi_2 + t^{2j + 1} \xi_2^2$ changes sign  from $ - $ to $ + $ when $\xi_1 = - t^{2j + 1}\xi_2 $ and  $\xi_2 \ne 0$.        
Since $ |\partial_{\xi} p_{s,2}|  \cong \mn{\hess p_{s,2}}  \cong 1 $ and $ |p_{s,2}|\cong |t|^{2j + 1}$ when $\tau = \xi_1 = 0 $ and $\xi_2 \ne 0 $, we find that conditions~\eqref{kcond} and \eqref{cond2} do not hold. 
By choosing $x_2 -t^{2j + 1} x_1 $ as new $ x_2$ coordinate we obtain that 
\begin{equation}
P = D_t + i\big (D_{x_1} + i (2j + 1) t^{2j}x_1 \big )D_{x_2}  + i (2j^2 + j)  t^{2j - 1} x_1^2
\end{equation}
Then by conjugating $ P$ with $e^{(2j + 1) t^{2j}x_1^2/2} $ we obtain $ P  = D_t + iD_{x_1}D_{x_2}$ which has constant coefficients and is solvable. 
\end{exe}

\begin{exe}\label{etaex}
Consider the operator
\begin{equation}
P = D_t + f(t,y,D_x) + i\, \Box_y  \qquad (x,y) \in \br^{n-2} \times \br^2
\end{equation}
where  $f(t,y,\xi)\in S^1_{\mathrm{hom}} $ is real and $\Box_y= \partial_{y_1}\partial_{y_2}$ is the wave operator in $ y \in \br^2$.
We find that $P$ is of subprincipal type, $k = 2$, $\st  = \set {\eta = 0} $ and $ p_{s,2}(t,y,\tau,\xi,\eta ) = \tau + f(t,y,\xi)  -i \eta_1   \eta_2$ so   \eqref{cond1} is not satisfied if $\partial_{y}f \ne 0 $. 
Since $-iP =  \Box_y - i D_t -i f(t,y,D_x)$ it follows from Theorem 1.2 in \cite{MU1} that 
$ P$ is not solvable if $\partial_{y}f \ne 0 $.
\end{exe}

\begin{exe}
Consider the operator
\begin{equation}
P = D_t + i f(t,x,D_x) + B(t,x,D_y)  \qquad (x,y) \in \br^{n-d} \times \br^d
\end{equation}
where  $0 < d < n $, $f(t,x,\xi) \in S^1_{\mathrm{hom}} $ is real and $ B(t,x,\eta) \in S^2_{\mathrm{hom}}$ vanishes of degree $ k \ge 2$ at $ \st = \set {\eta = 0}$. Then $p_{s,k} = \tau +if(t,x,\xi) + B_k(t,x,\eta) $ where $ B_k$ is the $ k$:th Taylor term at $ \st $ of the principal symbol of $ B$, so \eqref{cond1} is satisfied everywhere. 

Assume that $B(t,\eta)$ is independent of $ x$ and the sign change in $t \mapsto f (t,x,\xi) +\im B_k(t,\eta) $ is from $ -$ to $ +$ of order $\ell < \infty $ at $ t = t_0$. If $t \mapsto \partial_\eta B_k(t,\eta) $ vanishes of order greater than~$\ell/k$ at $ t = t_0$ then~\eqref{kcond} holds. If $ k = 2$ and $t \mapsto \partial_\eta^2 B_k(t,\eta) $ vanishes  at $ t = t_0$
then~\eqref{cond2} holds. Then $P$ is not solvable by Theorem~\ref{mainthm} and Remark~\ref{thmrem}. 

If  $ \im B(x,\eta)  \ne 0$ is constant in $ t$ and $k $ is odd  with  $\im B_k(x,\eta) \gtrless  0 $, $ \forall \, x$, then condition $\sub_k(\Psi) $ implies that $t \mapsto f(t,x,\xi)  $ is nonincreasing. In fact, Sard's theorem gives for almost all values $f_0 $ of  $ f$ that there exists $(t,x,\xi) $ so that $ f(t,x,\xi) = f_0$ and $ \partial_t f(t,x,\xi) \ne 0$. Then one can choose $\eta$ so that $f(t,x,\xi) +  \im B_k(x,\eta) = 0$  so $\sub_k(\Psi) $  gives  $\partial_t f(t,x,\xi) \le 0 $. 

If  $t \mapsto f(t,x,\xi)  $ is nonincreasing,  $ B(x,\eta)$ is constant in $ t$ and $\re B \equiv 0 $,  then  $ P$ is solvable. In fact, then $ [P^*, P]= 2i [\re P, \im P] = 2\partial_t f \le 0 $ so $ \mn{\re P u}^2 \ls \mn{Pu}^2+ \mn{P^*u}^2 \ls \mn{P^*u}^2 + \mn{u}^2$ and  $\mn{u} \ll \mn{\re P u}$ if $|t| \ll 1 $ in the support of $u \in C^\infty_0 $.
\end{exe}

\begin{exe}\label{LNS}
The linearized Navier-Stokes equation 
\begin{equation}
\partial_t u + \sum_j a_j(t,x) \partial_{x_j}u + \Delta_x u = f   \qquad a_j(x) \in C^\infty
\end{equation}
is of subprincipal type. The symbol is
\begin{equation}
i\tau + i\sum_j a_j(t,x)\xi_j -  | \xi |^2
\end{equation}
so $P$ is of subprincipal type,  $k = 2$, ${\Sigma}_2 = \set{{\xi} = 0}$ and $ p_{s,2}(\tau,\xi) = i\tau - |\xi |^2 $. Thus $ \sub_2 (\Psi) $ holds since $ -|\xi|^2$ does not change sign when $t$ changes. 
\end{exe}

\begin{exe}\label{effhyp}
Effectively hyperbolic operators $P$ are weakly hyperbolic operators for which the fundamental matrix $F$ has two real eigenvalues, here $ F = \Cal J \hess p\restr \st $ with $p = \sigma(P)$ and $\Cal J(x,\xi) = (\xi, -x)$ is the symplectic involution. Then  $P$ is solvable for any subprincipal symbol by (see \cite{Mel} and \cite{Nishi}) but in this case $ \st $ is \emph{not} an involutive manifold.
\end{exe}

\section{The normal form} \label{normform}

We are going to prepare the operator microlocally near the semibicharacteristic.
We have assumed that $P^*$ has the symbol expansion $p_m +
p_{m-1} + \dots$ where $p_j \in S^j_{\mathrm{hom}}$ is homogeneous of degree~$j$. By
multiplying  $P^*$ with an elliptic classical pseudodifferential operator, we may
assume that $m=2$ and $p = p_2 $. 
By chosing local homogeneous symplectic coordinates $(x,y;{\xi},{\eta})$ we may assume that $X = \br ^n $ and $\st =
\set{{\eta} = 0} \subset T^* \br ^n \setminus 0$ with the symplectic foliation by leaves spanned by $\partial_y$.
If $p$ vanishes of  order $ k < \infty$ at $\omega \subset \st$ we find that 
\begin{equation}\label{p2fact}
 p(x,y;{\xi},{\eta}) = \sum_{|\alpha| = k}
 B_{\alpha}(x,y;{\xi},{\eta}) {\eta}^\alpha/\alpha! \qquad (x,y,\xi) \in \omega 
\end{equation}
where $B_{\alpha}$ is homogeneous of degree $2-k$, and ${B_{\alpha}(x,y;{\xi},0)} \not \equiv 0$ for some $|\alpha| = k $ and some  $(x,y,{\xi},0) \in \omega$.
When $p$ vanishes of  infinite order we get~\eqref{p2fact} for any $ k$.

We shall first consider the case when $k = \kappa(\omega) < \infty $. Recall the reduced subprincipal symbol 
$p_{s,k}(w,\eta) = J^k_w(p)(\eta) + p_s(w)  $, $w \in \st $, by Definition~\ref{redsub}, and the extended subprincipal symbol 
$q_{s,k}(w,\eta,\lambda)  =  \lambda J^{2k-1}_w(p)(\eta/ \lambda^{1/k} ) + 
J^{k-1}_w(p_s)(\eta/ \lambda^{1/k}) $ by Definition~\ref{defqsk}.
Observe that these are invariantly defined and are the complex conjugates of the corresponding symbols of~$P$ by  Remark~\ref{detarem}.
We also find from Remark~\ref{detarem}  that
\begin{multline}\label{psubdef}
p_{sub} (x,y;\xi,  \eta/ |\xi|^{{1}/{k}}) 
\cong |\xi| q_{s,k}(x,y; \xi_0, \eta_0, |\xi|) \\ =  |\xi| p_{s,k}(x,y; \xi_0, \eta_0) +  \Cal O (|\xi|^{1-1/k}) \qquad \text{modulo $\Cal O(1)$}
\end{multline} 
where $(\xi_0,\eta_0) =   |\xi| ^{-1}(\xi,\eta)$. We also have
\begin{equation}
\partial_\eta p_{sub} (x,y; \xi, \eta/ |\xi|^{{1}/{k}}) 
 \cong  |\xi|^{1/k} \partial_\eta p_{s,k}(x,y; \xi_0, \eta_0)  \qquad \text{modulo $\Cal O (1)$}
\end{equation} 
and
\begin{equation}
\partial_\eta^2 p_{sub}(x,y; \xi,\eta/ |\xi|^{{1}/{k}}) \\  \cong  |\xi|^{{2}/{k} -1} \partial_\eta^2 p_{s,k}(x,y;\xi_0,\eta_0) \qquad \text{modulo $\Cal O ( |\xi|^{{1}/{k} - 1})$}
\end{equation}
When $k < \infty $ we shall localize with respect to the metric
\begin{equation}\label{gdef}
g_k(dx,dy; d\xi, d\eta) = |dx|^2+|dy|^2 +
|d{\xi}|^2/\Lambda^2 + |d{\eta}|^2/\Lambda^{2- {2}/{k}}
\end{equation}
where  $\Lambda = ( |\xi |^2 + 1 )^{ 1/2}$.
If $g_{\varrho, \delta}$ is the metric corresponding to the symbol classes $S^m_{\varrho, \delta} $  we find that 
\begin{equation*}
g_{1,0} \le g_k \le g_{1 -{1}/{k},0}
\end{equation*}
When $k = \infty $ we shall let $g_\infty =  g_{1,0}$ which is the limit metric when $k \to \infty $.

We shall use the Weyl calculus symbol notation $ S(m,g_k) $ where $m$ is a weight for $g_k$, one example is  $\Lambda^m = ( |\xi |^2 + 1 )^{ m/2}$. Observe that we have the usual asymptotic expansion when composing $ S(\Lambda^m, g_k) $ with $ S^j_{\varrho,\delta} $ when $\varrho > 0 $ and $ \delta < 1 -\frac{1}{k}$.

\begin{rem}
If $k < \infty $, $f $ is homogeneous of degree $m$ and vanishes of order $j$ at $\st $ then $f \in  S(\Lambda^{m - {j}/{k}}, g_k)$ when $|\eta| \ls |\xi| ^{1 -{1}/{k}}$.
\end{rem} 

One example is $ p = \sigma(P^*) \in S(\Lambda, g_k)$ in $ \omega$  when $|\eta| \ls |\xi| ^{1 - {1}/{k}}$ for $k = \kappa(\omega) $. 
In fact, when $|\eta| \ls |\xi| ^{1 -{1}/{k}}$ we have $|f| \ls |\xi|^{m-j}|\eta|^{j} \ls |\xi|^{m - {j}/{k}} $. 
Differentiation in $x$ or $y$ does not change this estimate, differentiation in $\xi$  lowers the homogenity by one and when taking derivatives in $\eta$ we may lose a factor $\eta_j = \Cal O(|\xi| ^{1 - {1}/{k}})$. We shall prepare the symbol in 
domains of the type
\begin{equation}\label{wtgamma}
\wt \Omega = \set {(x,y, \lambda \xi, \lambda^{1-  {1}/{k}} \eta ): (x,y,\xi,\eta) \in \Omega  \subset S^*\br^n, \quad \lambda > 0}
\end{equation}
which is a $g_k$ neighborhood consisting of the inhomogeneous rays through $\Omega$.

Now for $k < \infty $ we shall use the blowup mapping
\begin{equation}\label{blowup}
\chi: \ (x,y; \xi, \eta) \mapsto (x,y; \xi, \eta/|\xi |^{{1}/{k}} ) 
\end{equation}
which is a bijection when $ |\xi| \ne 0 $. The pullback by $ \chi$ maps symbols in $ S(\Lambda^m, g_k) $ where  $|\eta| \ls |\xi |^{1- {1}/{k}} $ 
to symbols in $S^m_{1,0} $ where  $|\eta| \ls |\xi | $, see for example~\eqref{psubdef}. Also Taylor expansions  in $ \eta$ where $|\eta| \ls |\xi |^{1- {1}/{k}} $ get mapped by~$\chi^* $ to polyhomogeneous expansions, and  a conical neighborhood $\Omega$ is mapped by $ \chi $ to the $g_k$ neighborhood $\wt \Omega $.

The blowup 
\begin{multline}\label{qdef}
p_{sub} \circ \chi(x,y;\xi,\eta) =  q(x,y;\xi,\eta) =  |\xi|q_{s,k}(x,y; \xi/|\xi|, \eta/|\xi|, |\xi|) \\ \cong  |\xi|p_{s,k}(x,y; \xi/|\xi|, \eta/|\xi|)\in S^1_{1,0} \qquad \text{modulo $ S^{1 - {1}/{k}}_{1,0}$}
\end{multline} 
is a sum of terms homogeneous of degree $ 1 - j/k$ 
for $j \ge 0 $ by Definition~\ref{defqsk}. We shall prepare the blowup  $ q_{s,k}$ and get it on a normal form after multiplication with pseudodifferential operators and conjugation with elliptic Fourier integral operators.

We have assumed that $P$ is of subprincipal type and does not satisfy condition $\sub_k({\Psi})$   near a subprincipal semicharacteristic  ${\Gamma}\times \set {\eta_0} \subset N^*\st$, which is transversal to the leaves of $N^*\st $.  
By changing $\Gamma $ and $ \eta_{0}$ we may obtain that  $\im a p_{s,k}$ changes
sign from $+$ to $-$ on the bicharacteristic~${\Gamma}\times \set {\eta_0} $ of ${\re
	a p_{s,k}}$ for some  $0 \ne a \in C^\infty$. 
The differential inequality~\eqref{cond1} in these coordinates means that
\begin{equation}\label{cond1q}
|\partial_{y}q_{s,k}| \le C|q_{s,k}|
\qquad \text{when $|\xi | \gg 1$ }
\end{equation}
in a conical neighborhood $\omega $ in $ N^*\st $  containing $ {\Gamma}\times \set {\eta_0} $.
By shrinking $\omega $ we may obtain that the intersections of  $\omega$ and the leaves of $ \st$ are simply connected.
Then by putting $|\xi | = \lambda $ we obtain from Remark~\ref{qskrem} that 
\begin{equation} \label{qcq}
\wt   q_{s,k}(x,{\xi},\eta) \cong
 c(x,y,{\xi},\eta) q_{s,k}(x,y,{\xi},\eta) \qquad \text{at $\omega$ when $| \xi| \gg 1$ }
\end{equation}
modulo  $S^0_{1,0} $. 
Here $\wt  q_{s,k}$ is the value of $ q_{s,k}$ at the intersection of the semibicharacteristic and the leaf.
Here $0 \ne  c \in S^0_{1,0}$ is a sum of terms homogeneous of degree $ - j/k$ for $j \ge 0$ such that $|c| > 0$ when $| \xi| \gg 1$. In fact, $ c$ has an expansion in $\eta/|\xi|^{1/k} $ and it suffices to take terms up to order $ k$ in $c $ to get~\eqref{qcq} modulo  $S^0_{1,0} $.
Thus the term homogeneous of degree 0 in $ c$ is nonvanishing in the conical neighborhood $ \omega$. By cutting off the coefficients of the lower order terms of $ c$ where $|\xi| \gg 1 $, we may assume that $ c \ne 0$ in $ \omega$.

By multiplying $P$ with a pseudodifferential operator with symbol $C = c \, \circ \chi^{-1} \in S(1, g_k)$  when  $|\eta| \ls |\xi | ^{1 - {1}/{k}}$, we obtain by Remark~\ref{qskrem} the refined principal symbol 
\begin{equation}
\wt p_{sub} + \frac{1}{2i}\partial_y C \partial_\eta p_{sub} \qquad \text{modulo $S(1,g_k)$ in  $\chi (\omega) $}
\end{equation}
for $\wt P = CP $. Here $\wt p_{sub} = C p_{sub} $ is constant on the leaves of $N^* \st$ modulo $S(1,g_k)$ in  $\chi (\omega) $. We have that  
\begin{equation*}
\partial_\eta p_{sub} = C^{-1}\partial_\eta \wt  p_{sub} + \partial_\eta  C^{-1} \wt p_{sub} 
\end{equation*}
where $\partial_\eta  C^{-1} \in  S(\Lambda^{ {1}/{k} -1},g_k)$ when  $|\eta| \ls |\xi | ^{1 -{1}/{k}} $. Thus by multiplying $\wt  P$ with a  pseudodifferential operator with symbol $1 - \frac{1}{2i} \partial_\eta  C^{-1}\partial_y  C \in S(1,g_k)$, we obtain the refined principal symbol 
\begin{equation}\label{rps}
\wt p_{sub}  + c_0 \partial_\eta \wt p_{sub}\qquad \text{modulo $S(1,g_k)$ in  $\chi (\omega) $}
\end{equation} 
for some $c_0 \in S(1,g_k) $ which may depend on $ y$.
Then we find that \begin{equation*}
\partial_y (\wt p_{sub}  + c_0 \partial_\eta \wt p_{sub})  \cong  \partial_y c_0 \partial_\eta \wt p_{sub}  \qquad \text{modulo $S(1,g_k)$}
\end{equation*} 
By putting $q =\wt  p_{sub} \circ \chi $ obtain that $\partial_y q = 0 $ in $\omega$ when $| \xi| \gg 1$. We shall control the term proportional to $ \partial_\eta \wt p_{sub}\circ \chi =  |\xi|^{1/k} \partial_\eta q \in  S^{{1}/{k}}$ by using condition~\eqref{kcond}, see Lemma~\ref{lemclaim}. Observe that $q \cong q_1 = p_{s,k}  \circ \chi$ modulo $ S^{1- {1}/{k}}$, which is homogeneous and independent of $y $ near $ \omega$. By the invariance of the condition, we may assume that $ a$ is independent of~$y$. Then the semibicharacteristics are constant in $ \eta$ so we may choose $ a$  independent of~$(y,\eta) $.

Observe that changing $ a$ changes $ \Gamma$  and $ \eta_0$ by the invariance, but we may assume that $ \Gamma\times \set {\eta_0}$ is arbitrarily close to the original semibicharacteristic by \cite[Theorem~26.4.12]{ho:yellow}.  Since $\im aq_1$ changes sign on ${\Gamma}\times \set {\eta_0}$ there is a maximal semibicharacteristic ${\Gamma}' \times \set {\eta_0}$ on which $\im aq_1 = 0$ and because of the sign change we may shrink ${\Gamma}$ so that
it is not a closed curve. Here
${\Gamma}'$ could be a point, which is always the case if the sign
change is of finite order. By continuity, $\partial_{x,\xi} \re a q_1 \ne 0$ near $\Gamma'\times \set {\eta}$ for $\eta $ close to $\eta_0$  and we may extend $ a$ to a nonvanishing symbol that is  homogeneous of degree~0 near $\Gamma $.
Multiplying $ P$ with an elliptic  pseudodifferential operator with symbol $a = a \circ \chi^{-1}$  we may assume that $ a \equiv 1$.

Recall that conditions  \eqref{kcond}  (and \eqref{cond2} when $k = 2 $) holds in some
neighborhood of~${\Gamma} \times \set {\eta_0}$  with the gliding foliation $ M$ of $N^*\st  $. 
By using Darboux' theorem we can choose local coordinate functions $ (x,\xi)$ such that $ TM$ is spanned by $\partial_x $ and $\partial_\xi $ for the leaves of~$ M$. 
Now $0 \ne H_{\re q_1} $ is tangent to $\Gamma' \times \set {\eta_0}$, transversal to the symplectic foliation of $ \st $, constant in $ y$ and in the symplectic annihilator of $ TM$. 
Since $TM $ is symplectic, this gives that $\re q_1 $ and $\eta $ are constant on the leaves $ M$.
Now take  ${\tau}= \re  q_1$ when $ {\eta = \eta_0}$ and extend it is so that $\tau $ is independent of $\eta $.
Then we can complete ${\tau}$, $ y$ and  $ \eta$ to a homogeneous symplectic coordinate system $(t,x,y;\tau,\xi,\eta)$ in a conical neighborhood $\omega$ of ${\Gamma}' $ in $\st $
so that  $ (x,\xi)\restr {\eta = \eta_0}$ is preserved.
Since the change of variables preserves the $ (y,\eta)$ variables, it  preserves  $ \st = \set {\eta = 0}$ and its symplectic  foliation and the fact that  $\partial_y q = \set {\eta, q } = 0 $.  When $\eta = \eta_0 $  we have that  $\re q_1 = \tau $ and the  leaves $TM $  of the foliation $M  $ is spanned by $ \partial_x = H_\xi$ and $ \partial_\xi = - H_x$ modulo $\partial_y $ when $\eta = \eta_0 $. Since $q $ is independent of $y $ we may assume that  $ V_\ell  $ is in the span of $ \partial_{x,\xi}$ in  \eqref{kcond} and \eqref{cond2}.
Since the $ \eta$ variables are preserved, the blowup map $\chi $ and the inhomogeneous rays are preserved and the  coordinate change is an isometry with respect to  the metric $ g_k$.

By conjugating with elliptic Fourier integral operators in the variables $(t,x) $ independently of  $y $ microlocally  near $\Gamma' \subset  \st $, we obtain that $\re q_1 = \tau $ in a conical neighborhood of~${\Gamma} $ when $\eta = {\eta_0}$. This gives  
$ q_1 = \tau + \varrho(t,x,\tau, \xi, \eta) $ in a neighborhood of $\Gamma' \times \set{\eta_0} $, where $\varrho$ is  homogeneous and $\re  \varrho \equiv 0$ when $\eta = \eta_0 $. Since this is a change of symplectic variables $(t,x;\tau,\xi)$ for fixed $(y,\eta) $ we find by the invariance that $t \mapsto \im \varrho(t,x,0, \xi,\eta_0) $ changes sign from $ +$ to $- $ near $ \Gamma'$.
Observe that the reduced principal symbol is invariant under the conjugation  by Remark \ref{invrem}  so
the condition that $\partial _y q =   \set { \eta, q} = 0 $ is preserved, but we may also have a term $c\partial_\eta q \in S^{1/k}$ where $ c$ could depend on $ y$.

Next, we shall use the Malgrange preparation theorem on $ q_1$.
Since $ \partial_\tau q_1 \ne 0$ near  $ \Gamma' \times \set {\eta_0}$ we obtain that
\begin{equation}\label{Malprep}
\tau = c(t,x,\tau,\xi,\eta)q_1(t,x,\tau,\xi,\eta) + r(t,x,\xi,\eta)
\end{equation}
locally for $\eta $ close to $\eta_0$  when $|\xi| = 1 $, and by a partition of unity near $ \Gamma' \times \set {\eta_0}$, which can be extended by  homogeneity so that $ c$ is homogeneous of degree 0 and $ r$ is homogeneous of degree 1. 
Observe that this gives that $\partial_\eta q_1 =  \partial_\eta c^{-1}(\tau - r) - c^{-1}\partial_\eta r$ by \eqref{Malprep}.
By taking the $ \tau$ derivative of~\eqref{Malprep}  using that $ q_1 = 0$ and  $ \partial_\tau q_1 = 1$ at $ \Gamma' \times \set {\eta_0}$ we obtain that   $c = 1$ on $  \Gamma' \times \set {\eta_0}$.
Multiplying the operator $ P^*$ with a pseudodifferential operator with symbol $c \circ \chi^{-1} \in S(1,g_k)$ when $|\eta| \ls |\xi|^{1 - {1}/{k}} $ we obtain that 
$q_1(t,x,\tau,\xi,\eta) = \tau - r(t,x,\xi,\eta)$ in  a conical neighborhood of  $ \Gamma' \times \set {\eta_0}$.

Writing $r = r_1 + i r_2$ with $r_j$ real, we may complete $\tau - r_1(t,x,\xi, \eta_0)$,   $t$,   $ y$ and  $\eta$  to a homogeneous  symplectic coordinate system $(t,x,y;\tau,\xi,\eta)$  near $\Gamma' \times \set {\eta_0}$ so that $(x,\xi)\restr {t=0} $ is preserved. 
This is a  change of coordinates in $(t,x;\tau, \xi) $ which as before is independent of the variables $(y,\eta) $.
We find that $\partial_\tau r = \set {r, t} = 0$  and $\partial_y r = \set {\eta, r} = 0 $ are preserved and $\re r \restr {\eta = \eta_0} \equiv 0 $. By the invariance we find that $t \mapsto  r_2(t,x, \xi,\eta_0) $ changes sign from $ +$ to $- $ near $ \Gamma'$. As before,  the blowup map $ \chi$ and the inhomogeneous rays are preserved and the  coordinate change is an isometry with respect to  the metric $ g_k$.

By conjugating with elliptic Fourier integral operators in $ (t,x)$ which are constant in $y$ microlocally near $\Gamma' \subset  \st $, the calculus gives as before that 
\begin{equation}\label{q1f}
q_1(t,x,\tau,\xi, \eta) = \tau + f(t,x,\xi, \eta) 
\end{equation}
where  $f  = - r$. 
When $\eta =  {\eta_0} $ we have $\re f \equiv 0 $ and $\im f $ has a sign change from $ +$ to $ -$ as $ t$ increases near $\Gamma' $  by the invariance of condition  $({\Psi})$.
In fact,  the reduced principal symbol is invariant under the conjugation 
so the condition that $\partial _y q =   \set { \eta, q} = 0 $ is preserved, but we may also have a term $c\partial_\eta q \in S^{1/k}$ where $ c$ could depend on~$ y$.
Observe that $\tau $ is invariant under the blowup mapping $ \chi $  given by~\eqref{blowup}.
By the invariance, we find that 
\eqref{kcond}  (and \eqref{cond2} if $k = 2 $)  holds for $q_1$  with $TM $ spanned by $\partial_x $ and $ \partial_\xi $ when  $\eta =  {\eta_0} $. In fact, since $y $, $\eta$ and $t$  are independent of $(x,\xi) $ we find that the span of $\partial_{x,\xi} $ is invariant modulo terms proportional to $\partial_\tau $. Since $ \partial_\eta q_1$ is independent of $ \tau$ by~\eqref{q1f}  we may take $V_j $ in the span of $\partial_{x,\xi}$  in  \eqref{kcond} (and \eqref{cond2} if $ k= 2$) when $\eta =  {\eta_0} $.
Thus, by putting $ \tau = - \re f$ in \eqref{kcond} 
we obtain that there exists $\varepsilon > 0 $ so that
\begin{equation}\label{deta0}
|\partial_{x,\xi}^\alpha \partial_{{\eta}} f|  \ls | \im f|^{{1}/{k} + \varepsilon} \qquad \forall\, \alpha \quad |\eta - \eta_0| < c
\end{equation}
near $ \Gamma'$ in $ \st$. 
Observe that on $\Gamma' $,  where $ \im f$ vanishes, \eqref{deta0} gives that $ \partial_{x,\xi}^\alpha \partial_{{\eta}} f$ vanishes $  \forall\, \alpha$.
Similarly, it follows from~\eqref{cond2} 
that  there exist  $\varepsilon > 0 $ so that
\begin{equation}\label{ddeta0}
|\partial_{x,\xi}^\alpha \partial_{{\eta}}^2 f|  \ls | \im f|^{\varepsilon}  \qquad \forall\, \alpha 
\end{equation}
near $ \Gamma'$ in $ \st$. As before,  \eqref{ddeta0} gives that $ \partial_{x,\xi}^\alpha \partial_{{\eta}}^2 f $ vanishes $  \forall\, \alpha$, when $ \im f$ vanishes. 

We shall next consider  on the lower order terms in the expansion $q = q_1 + q_2 + \dots $ where $q_j \in S^{1- ({j-1})/{k}}$. Observe that $\partial_{{\eta}} q_2 = 0$ when $ q_1= 0$ by condition \eqref{dqcond}. (Actually, since $dq_1 \wedge d\ol q_1 = 2i df \wedge d\tau $ it suffices that this holds when $ f$ vanishes of infinite order.) 
We shall use the  Malgrange preparation theorem on $q_j $, $  j \ge 2 $, in a conical neighborhood of  $ \Gamma' \times \set {\eta_0}$. Since $ \partial_\tau q_1 \ne 0$ we obtain
\begin{equation}\label{Maldiv}
q_j(t,x,\tau,\xi,\eta)= c_j(t,x,\tau,\xi,\eta)q_1(t,x,\tau,\xi,\eta) + r_j(t,x,\xi,\eta) \qquad j \ge 2
\end{equation}
locally for $\eta $ close to $\eta_0$  when $|\xi| = 1 $, and by a partition of unity near $ \Gamma' \times \set {\eta_0}$. This can be extended to a  conical neighborhood of $\Gamma' \times \set {\eta_0}$ so that $r_j \in S^{1- ({j-1})/{k}}$ is independent of  $\tau$ and $ c_j \in S^{-({j-1})/{k}}$. Multiplying the operator $ P^*$ with a pseudodifferential operator with symbol $1 - c_j \circ \chi^{-1} \in S(1,g_k)$ where  $ c_j \circ \chi^{-1} \in S( \Lambda^{-({j-1})/{k}} ,g_k)$ when  $|\eta| \ls |\xi|^{1 - {1}/{k}} $ we obtain that $q_j(t,x,\tau,\xi,\eta) = r_j(t,x,\xi,\eta)$. Since $q_2 $ is  now independent of $ \tau$ we find by putting $\tau = - \re f $ that  $\partial_{{\eta}} q_2 = 0$ when $ \im f = 0$ (of infinite order) by condition \eqref{dqcond}.

When $ j = k$ then we find from~\eqref{rps} and~\eqref{Malprep} that $q_k \in S^{1/{k}}$ also contains the term $c_0 \partial_{{\eta}} c^{-1}(\tau + f) +  c_0 c^{-1}\partial_{{\eta}} f$ modulo $ S^0$,  where $c^{-1} \in S^0 $  and  $ c_0  \in S^{1/{k}}$ may depend on~$ y$. 
By using \eqref{Maldiv} and multiplying the operator $ P^*$ with a pseudodifferential operator with symbol $1 - (c_k + c_0\partial_{{\eta}}  c^{-1}) \circ \chi^{-1} \in S(1,g_k) $ when  $|\eta| \ls |\xi|^{1 - {1}/{k}} $ we obtain that $q_k = r_k + c_0c^{-1} \partial_\eta f$, where   $ c_0  \in S^{1/{k}}$   may depend on $ y$.

This preparation can be done for all lower order terms of pullback of the full symbol of $P^*$ given by $\sigma(P^*) \circ  \chi$, where the terms are in $ S^{-{j}/{k}}$ for $j \ge 0$. These terms may depend on~$y$, but that does not change the already prepared terms since  $c_j \in S^{-1- {j}/{k}}$ in~\eqref{Maldiv}. 
We shall cut off in a $g_k $ neighborhood of the bicharacteristic and then we have to  measure the error terms of the preparation.

\begin{defn}\label{wfdef}
In the case $k < \infty $ and  $R \in {\Psi}^{\mu}_{{\varrho},{\delta}}$  where $\varrho + \delta \ge 1 $, ${\varrho} > 0$ and ${\delta} < 1 - \frac{1}{k}$
we say that $T^*X \ni
(t_0,x_0,y_0;\tau_0, {\xi}_0,\eta_0) \notin \wf_{g_k} (R)$ if the 
symbol of $R$ is $\Cal O(|{\xi}|^{-N})$, $\forall\, N$, when
the $g_k$ distance to the inhomogeneous ray 
$\set{(t_0,x_0,y_0;\varrho\tau_0,\varrho {\xi}_0,\varrho^{1 - {1}/{k}}\eta_0):\
{\varrho} \in \br_+}$ is less than $c > 0$.
If $ k = \infty$ and  $R \in {\Psi}^{\mu}_{{\varrho},{\delta}}$, for ${\varrho} > 0$ and ${\delta} < 1 $, then $\wf_{g_k} (R) = \wf (R)$.
\end{defn}

For example, $(t_0,x_0,y_0;\tau_0, {\xi}_0,\eta_0) \notin \wf_{g_k} \big (R(D) \big ) $ if $ R$ is the
cutoff  function  
\begin{equation*}
R(\tau, \xi,\eta) = 1- \chi\big (c |(\tau, \xi)|^{{1}/{k} - 1} (\eta - \eta_0)\big )\in S(1, g_k)
\end{equation*}
with $\chi \in C^\infty_0 $ such that  $0 \notin \supp (1 - \chi) $ and $ c > 0$.
By the calculus, Definition \ref{wfdef} means that there exists $A\in
S(1,g_k)$ so that $A \ge c > 0$ in a $g_k$
neighborhood of the inhomogeneous ray such that $AR \in {\Psi}^{-N}$ for any $N$.
By the conditions on ${\varrho}$ and ${\delta}$, 
it  follows from the calculus that Definition  \ref{wfdef}  is invariant
under composition with classical elliptic pseudodifferential operators and under conjugation
with elliptic homogeneous Fourier integral operators  preserving the fiber and 
$\st = \set {\eta = 0}$. 
We also have that $\wf_{g_k}(R) $ grows when $k $ increases and $\wf_{g_k}(R) \subseteq \wf (R) $,  with equality when $k = \infty $.

Cutting off where $|\eta - \eta_0| \ls |\xi |^{1- {1}/{k}} $ we obtain that
\begin{equation}
P^* = D_t + F_1(t,x,y,D_x,D_y) +  F_0(t,x,y, D_x,D_y) + R(t,x,y, D_x,D_y) 
\end{equation}
where $ R\in S(\Lambda^{2},g_k)$ such that $ \Gamma'  \times \set {\eta_0} \bigcap \wf_{g_k}(R) =  \emptyset$, $F_j \in S(\Lambda^{j},g_k)$ such that 
\begin{equation}
F_1\circ \chi (t,x,y, \xi,\eta)  \cong  f(t,x,\xi, \eta)  + r(t,x,\xi, \eta)\in S^1
\end{equation} 
modulo $S^{1- {2}/{k}}$, where $r \in  S^{1- {1}/{k}}$ and $\partial_{{\eta}} r = 0 $ when $ \im f$ vanishes (of infinite order). Also there exists $ c \in  S( 1, g_k)$ so that $F_1 - c \partial_{{\eta}}F_1$ is constant in $ y$ modulo $S(1,g_k) $.

Next, we study the case when $k = \infty $. We have $q_{s,\infty} = p_s $, $p_{s,\infty} = p_s \restr \st  $ and $g_\infty = g_{1,0} $. Then we shall not prepare the principal symbol, which vanishes of infinite order at $ \st $. Instead, we shall prepare the lower order terms starting with $p_1 $, which is homogeneous of degree~1. 
We shall prepare $p_1 $ in a similar way as~$q $ near the subprincipal semicharacteristic $ \Gamma \subset \st$.  First we may as before use the  differential inequality~\eqref{cond1} to obtain that $p_s $ is constant in $ y$ near $ \Gamma $ after multiplication with  an nonvanishing homogeneous $c \in S^0_{1,0} $. By multiplication with an elliptic pseudodifferential operator with symbol $ c$ we obtain that $p_1 $ is constant in $ y$ modulo terms vanishing of infinite order at $ \st$. In fact, the composition of $ P$ with a classical elliptic pseudodifferential operator can only give terms in $p_{sub} $ vanishing of infinite order at $ \st$.

By assumption condition $\sub_\infty (\Psi) $ is not satisfied, so there exists $0\ne a \in   S^0_{1,0}$ so that  $\im ap_1\restr \st$ changes
sign from $+$ to $-$ on the bicharacteristic~${\Gamma} \subset \st  $ of ${\re
a p_1}$ for some  $0 \ne a \in C^\infty$ which can be assumed to be homogeneous and constant in~$y$ and $\eta  $. By multiplication with an elliptic pseudodifferential operator with symbol $ a$ we may assume that $a \equiv 1 $.  
Let $\Gamma' \subset \Gamma$ be the subset on which $ p_1 $ vanishes.
Since $0 \ne H_{\re p_1} $ is tangent to $\Gamma \subset \st$ and $\partial_y p_1 = \set {p_1, \eta} =  0$ we can complete ${\tau}= \re p_1$ and $\eta $ to a homogeneous symplectic coordinate
system $(t,x,y;\tau,\xi,\eta)$ in a conical neighborhood $\omega$ of ${\Gamma}'$, which preserves the foliation of $ \st $. 
Then conjugating with an elliptic Fourier integral operators, we obtain $ p_1 = \tau + i \im p_1$ modulo terms vanishing of infinite order at~$\st $. The conjugation also gives terms proportional to $\partial_\eta p \in S^1$ which vanish of infinite order at $\st $.
As before,  we find that condition   $\sub_\infty ({\Psi})$ is not satisfied in any neighborhood of $ \Gamma' $ in $\st  $ by the invariance.

Since $ \partial_\tau p_1 \ne 0$ we can use the Malgrange preparation theorem  as before to obtain
\begin{equation}\label{Malprep0}
\tau = c(t,x,\tau,\xi,\eta)p_1(t,x,\tau,\xi,\eta) + r(t,x,\xi,\eta)
\end{equation}
locally, and by a partition of unity near $ \Gamma' \subset \st$. This may be extended by homogeneity to a conical neighborhood of $\Gamma' $, thus for $\eta $ close to $0$. 
Then $cp_1 = \tau - r $ where
$r \in S^1$ is constant in $\tau$ and $0 \ne c \in S^0$ near~$ \Gamma'$.  
In fact, this follows by taking the $ \tau$ derivative of~\eqref{Malprep0} and using that $ p_1 = 0$ and  $ \partial_\tau p_1 \ne 0$ at $ \Gamma' $. 

Multiplying the operator $ P^*$ with an elliptic pseudodifferential operator with symbol $c$ we obtain that 
$p_1(t,x,\tau,\xi,\eta) = \tau - r(t,x,\xi,\eta)$ in  a conical neighborhood of  $ \Gamma' $. By writing $r = r_1 + i r_2$ with $r_j$ real, we may complete $\tau - r_1(t,x,\xi, \eta)$, $ \eta$ and $t$ to a homogeneous symplectic coordinate system $(t,x,y;\tau,\xi,\eta)$  in a conical neighborhood of $\Gamma'$. By conjugating with elliptic Fourier integral operators we obtain that
\begin{equation}\label{p1f0}
p_1(t,x,\tau,\xi, \eta) = \tau + f(t,x,\xi, \eta)
\end{equation}
near $ \Gamma' \subset \st $, where $f = - ir_2 $ modulo terms vanishing of infinite order at $ \st$.  
By the invariance we find that $t \mapsto \im f(t,x,\xi, 0 ) $ changes sign from $+ $ to $- $ near $\Gamma' $. 

We shall next consider  on the lower order terms in the expansion $p + p_1 + p_0+ \dots $ where $p_j \in S^{j}$ near $ \Gamma'$ may depend on $ y$ when $j \le 0 $.  Observe that $\partial_{{\eta}} p_0 = 0$ when $ p_1 = 0$ by condition \eqref{dqcond}. (Actually, since $dp_1 \wedge d\ol p_1 = 2i df \wedge d\tau $ it suffices that this holds when $ f$ vanishes of infinite order.) 
We shall use the  Malgrange preparation theorem on $p_j $, $  j \le 0 $, in a conical neighborhood of  $ \Gamma'$. Since $ \partial_\tau p_1 \ne 0$ we obtain for $j \le 0 $  that
\begin{equation}\label{Maldiv0}
p_j(t,x,y,\tau,\xi,\eta)= c_j(t,x,y,\tau,\xi,\eta)p_1(t,x,\tau,\xi,\eta) + r_j(t,x,y,\xi,\eta)
\end{equation}
locally and by a partition of unity near $\Gamma' $. Extending by homogenity we obtain that $ r_j \in S^{j}$ and  $c_j \in S^{j-1} \subset S^{-1}$ near $ \Gamma' $ for $\eta $ close to $0$. After multiplication with an elliptic pseudodifferential operator with symbol $1 - c_j $ we obtain that $p_j \cong r_j $ is independent of  $ \tau$ modulo terms vanishing of infinite order at $ \st$. 
Since $p_0 $ is now independent of $ \tau$ we find by putting $\tau = - \re f $ that  $\partial_{{\eta}} p_0 = 0$ when $ \im f = 0$ (of infinite order) by condition \eqref{dqcond}.
Continuing in this way, we can make any lower order term in the expansion of $ P$  independent of  $ \tau$ modulo terms vanishing of infinite order at $ \st$.

Since condition $\sub_k({\Psi})$, $k \le \infty $, is \emph{not}
satisfied, we find that  $t \mapsto \im f(t,x_0,\xi_0,\eta_0)$ changes sign from $+$ to $-$ as $t \in I$ increases and we assume that $\im f(t,x_0,\xi_0,\eta_0) = 0$ when $ t \in I' \subset I$.
Observe that we shall keep $\eta_0 $ fixed and when $k = \infty $ we have $\eta_0 = 0 $.
If \eqref{deta0} holds then we find that $ \partial_{x,\xi}^\alpha \partial_{{\eta}} f = 0$ on  ${\Gamma}'\times \set {\eta_0}$,  $\forall \, \alpha\, \beta$, and if \eqref{ddeta0} holds then we find that $ \partial_{x,\xi}^\alpha \partial_{{\eta}}^2 f = 0$ on  ${\Gamma}'\times \set {\eta_0}$,  $\forall \, \alpha\, \beta$. Observe that we have $\partial_x^{\alpha}\partial_{\xi}^{\beta}\re f$,  $\forall \, \alpha\, \beta$, when $ \eta = \eta_0$.

Now if $| I' | \ne 0$, then by reducing to {\em
minimal bicharacteristics} near which $\im f $ changes sign as in~\cite[p.\ 75]{ho:nec}, we may assume that $\partial_x^{\alpha}\partial_{\xi}^{\beta}\im f$ vanishes on a bicharacteristic ${\Gamma}'\times \set {\eta_0}$,  $\forall \, \alpha\, \beta$, which is
arbitrarily close to the original bicharacteristic  (see~\cite[Sect.~2]{Witt} for a
more refined analysis).

In fact, if $\im f(a,x,{\xi},\eta_0) > 0 > \im f(b,x,{\xi},\eta_0)$ for some $(x,\xi) $ near  $(x_0,\xi_0) $ and $a <
b$, then we can define
\begin{equation*}
 L(x,{\xi}) = \inf \{\, t-s:\ a < s < t < b  \text{ and }  
    \im  f(s,x,{\xi},\eta_0) > 0 > \im f(t,x,{\xi},\eta_0) \,  \}
\end{equation*}
when $(x,{\xi})$ is close to $(x_0,{\xi}_0)$, and we put
$L_0 = \liminf_{(x,{\xi}) \to (x_0,{\xi}_0)} L(x,{\xi})$. Then for
every ${\varepsilon} > 0$ there exists an open neighborhood~$V_{\varepsilon}
$ of $(x_0,{\xi}_0)$ such that the diameter of~$V_{\varepsilon}$ is
less than~${\varepsilon}$ and $L(x,{\xi}) > L_0 - {\varepsilon}/2$
when $(x,{\xi}) \in V_{\varepsilon}$. By definition, there exists
$(x_{\varepsilon}, {\xi}_{\varepsilon} ) \in V_{\varepsilon}$ and $a < s_{\varepsilon} <
t_{\varepsilon} < b$ so that $t_{\varepsilon} - s_{\varepsilon} < L_0
+ {\varepsilon}/2$ and
$\im f(s_{\varepsilon},x_{\varepsilon},{\xi}_{\varepsilon},\eta_0) > 0 > 
\im f(t_{\varepsilon},x_{\varepsilon},{\xi}_{\varepsilon},\eta_0)$. Then it is
easy to see that  
\begin{equation}\label{mincharcond}
 \partial_x^{\alpha}\partial_{\xi}^{\beta} \im f(t, x_{\varepsilon},
 {\xi}_{\varepsilon},\eta_0) = 0 \quad \forall\,{\alpha}\,{\beta} \quad
 \text{when} \quad s_{\varepsilon}+ 
 {\varepsilon} < t  < t_{\varepsilon} - {\varepsilon} 
\end{equation}
since else we would have a sign change in an  interval of length less than $L_0
- {\varepsilon}/2$ in~$V_{\varepsilon}$. We may then choose a sequence
${\varepsilon}_j \to 0$ so that $s_{{\varepsilon}_j} \to s_0$ and
$t_{{\varepsilon}_j} \to t_0$, then $L_0 = t_0 - s_0$
and~\eqref{mincharcond} holds at~$(x_0,{\xi}_0,\eta_0)$ for $s_0 < t < t_0$.

\begin{prop}\label{prepprop}
Assume that $P$ satisfies the conditions in
Theorem~\ref{mainthm} with  $ k = \kappa(\omega)$. Then by conjugating with elliptic Fourier integral operators and multiplication with an elliptic
pseudodifferential operator we may assume that 
\begin{equation}\label{Pexp}
 P^* = D_t + F(t,x,y,D_x,D_y) +  R(t,x,y,D_t,D_x,D_y)
\end{equation}
microlocally near ${\Gamma} =  \set {(t,x_0,y_0; 0, \xi_0,0):\  t \in I}\subset \st$.
In the case $k < \infty $ we have
$ R \in S(\Lambda^2,g_k) \subset S^2_{1- {1}/{k},0}$ such that $ \Gamma  \times \set {\eta_0} \bigcap \wf_{g_k} (R) = \emptyset$, and $ F = F_1 + F_0$ with $F_1 \in S(\Lambda, g_k)$ and $F_0 \in S( 1, g_k)$.  
Here 
\begin{equation}\label{f1comp}
F_1 \circ \chi (t,x,y, {\xi}, {\eta})  \cong  f(t,x, {\xi},\eta)  + r(t,x,\xi, \eta) \quad \text{modulo $S^{1 - {2}/{k}} $} 
\end{equation}
where $ \chi $ is the blowup map \eqref{blowup}, $r \in  S^{1- {1}/{k}}$, $\re f(t,x,\xi, \eta_0) \equiv 0 $ and $\im f = \im p_{s,k}  \in S^1$ is given by~\eqref{pskdef} such that $ t \mapsto \im f(t,x_0,\xi_0, \eta_0)$ changes sign from $ +$ to $ -$ when $ t \in I$ increases. Also,  $\partial_{{\eta}} r = 0 $ when $ f$ vanishes (of infinite order), and there exists $ c \in  S( 1, g_k)$ so that $F_1 - c \partial_{{\eta}} F_1$ is constant in $ y$ modulo $S( 1, g_k) $, where 
$ \partial_{{\eta}}F_1 \in S( \Lambda^{{1}/{k}}, g_k)$.

If $ \eta_0 \ne  0$, then condition~\eqref{deta0} holds near $ \Gamma \times \set {\eta_0}$. 
If  $ k= 2$ then condition~\eqref{ddeta0} also holds near $ \Gamma \times \set {\eta_0}$ if $\eta_0 \ne 0 $ and near\/ $\Gamma' $ in $\st  $  if $\eta_0 =  0 $.
If $ f = 0$ on\/ $\Gamma' \times \set {\eta_0}$ where $\Gamma' \subset \Gamma $ and $ |\Gamma'| \ne 0$ we may assume that $\partial_x^\alpha \partial_\xi^\beta \partial_\eta^\gamma  f = 0 $ on $ \Gamma' \times \set {\eta_0}$ for any $ \alpha, \beta$  and $| \gamma| \le 1$, and when $k = 2 $  that $\partial_x^\alpha \partial_\xi^\beta \partial_\eta^\gamma  f = 0 $ on $ \Gamma' \times \set {\eta_0}$ for any  $ \alpha, \beta$  and $| \gamma| \le 2$.

In the case when $k = \infty $ we obtain~\eqref{Pexp} with $R \in  S^2_{1,0} $ vanishing of infinite order on~$ \st$, $F = F_1 + F_0 $ where $F_0(t,x,y;\xi, \eta) \in S^0_{1,0}  $ and
\begin{equation}
F_1(t,x,{\xi}, {\eta}) = f(t,x;\xi,\eta) \in S^1_{1,0} 
\end{equation}
where $ t \mapsto \im f(t,x_0,\xi_0, 0)$ changes sign from $ +$
to $ -$ when $ t $ increases.
If $ \im f = 0$ on\/ $\Gamma' \times \set {0}$ with $ |\Gamma'| \ne 0$ we may assume that $\partial_x^\alpha \partial_\xi^\beta  f = 0 $ on $ \Gamma' \times \set {0}$ for any $ \alpha, \beta$.
\end{prop}

\section{The Pseudomodes}\label{modes}

For the proof of Theorem~\ref{mainthm} we shall modify the Moyer-H\" ormander
construction of approximate solutions (or pseudomodes) of the type  
\begin{equation}\label{udef}
 u_{\lambda}(t,x,y) = e^{i{\lambda}{\omega}_\lambda(t,x,y)}
 \sum_{j\ge 0} {\phi}_j(t,x,y) {\lambda}^{-j\kappa} \qquad {\lambda}\ge 1
\end{equation}
with $ \kappa > 0$, phase function $ {\omega}_\lambda$ and amplitudes $ \phi_j$. 
Here the phase function ${\omega}_\lambda(t,x,y)$ will be uniformly bounded in $ C^\infty$ and complex valued, such that
$\im {\omega}_\lambda \ge 0$ and $\partial \re {\omega}_\lambda \ne 0$ when
$\im{\omega}_\lambda = 0$.  The amplitude functions $ \phi_j \in C^\infty$ may depend uniformly on $ \lambda$.
Letting $z = (t,x,y)$ we have the formal expansion
\begin{equation}\label{trevesexp}
  p(z,D_z)  (\exp(i{\lambda}{\omega}_\lambda){\phi})
  \sim \exp(i{\lambda}{\omega}_\lambda) \sum_{{\alpha}}
   \partial_{{\zeta}}^{\alpha }
  p(z,{\lambda} \partial_z{\omega}_\lambda(z))\Cal
  R_{\alpha}({\omega}_\lambda,{\lambda},D_z){\phi}(z)/{\alpha}! 
\end{equation}
where $\Cal R_{\alpha}({\omega}_\lambda,{\lambda},D_z){\phi}(z) =
D_w^{\alpha}(\exp(i{\lambda} \wt
{\omega}_\lambda(z,w)){\phi}(w))\restr{w=z}$ 
and 
$$
\wt {\omega}_\lambda(z,w) = {\omega}_\lambda(w) - {\omega}_\lambda(z) +
(z-w)\partial {\omega}_\lambda(z)
$$
The error term in~\eqref{trevesexp}  is of the same order in $ \lambda $ as the last term in the expansion.
Observe that since the phase is complex valued, the values of the symbol are given by an
almost analytic extension at the real parts, see Theorem  3.1 in Chapter VI and Chapter
X:4 in~\cite{T2}. If $P^* =  D_t + F(t,x,y,D_{x,y})$  we find from \eqref{trevesexp} that
\begin{multline}\label{exp0}
  e^{-i{\lambda}{\omega}_\lambda}P^* e^{i{\lambda}{\omega}_\lambda}{\phi} \\ =
\left({\lambda}\partial_t{\omega}_\lambda +
   F(t,x,y,{\lambda} \partial_{x,y}{\omega}_\lambda) -    
   i  \lambda \partial_{\xi,\eta}^2 F(t,x,y,{\lambda}\partial_{x,y}{\omega}_\lambda)
 \partial^2_{x,y}{\omega}_\lambda/2 \right){\phi}  \\ 
    + D_t{\phi} + \partial_{\xi,\eta}F(t,x,y,{\lambda}\partial_{x,y} {\omega}_\lambda)
   D_{x,y} {\phi} + \partial^2_{\xi,\eta}F(t,x,y,{\lambda}\partial_{x,y} {\omega}_\lambda)
   D_{x,y}^2 {\phi}/2 \\  +  \sum_{j \ge
  0}{\lambda}^{-j}R_{j}(t,x,y,D_{x,y}){\phi} 
\end{multline}
Here the values of the symbols at $(t,x,y,\lambda \partial_{t,x,y} {\omega}_\lambda)$ will
be replaced by finite Taylor expansions at $(t,x,y,\lambda \re \partial_{t,x,y} 
{\omega}_\lambda)$, which determine the almost analytic extensions.

Now assume that $P^* = D_t + F + R $ is given by Proposition~\ref{prepprop}. 
In the case $k = \kappa(\omega) < \infty$  in a open neighborhood $\omega$ of the bicharacteristic $ \Gamma$ and $ \eta_0 \ne 0$ we have $F = F_1 + F_0 $ with $F_j \in S(\Lambda^j, g_k) $ and $R \in S(\Lambda^2, g_k)  $ with $\Gamma  \notin \wf_{g_k} (R)$.
In this case, we shall
use a nonhomogeneous phase function given by~\eqref{omegaexp}:
\begin{multline}\label{omega1}
{\omega}_\lambda(t,x, y) =  \w {x-x_0(t), \xi_0} + \lambda^{-1/k}\w {y -y_0(t), \eta_0} \\ + \Cal O(|x-x_0(t)|^2) + \lambda^{\varrho - 1}\Cal O(|y-y_0(t)|^2)
\end{multline}
such that  $\partial_y {\omega}_\lambda =  \lambda^{-1/k} \eta_0 +  \Cal O(\lambda^{\varrho - 1})$ with some $0 <  \varrho < 1/2$. We find by Remark~\ref{detarem} that $F_1(t,x,y,{\lambda} \partial_x {\omega}_\lambda, \lambda\partial_y {\omega}_\lambda) \cong F_1(t,x,y,{\lambda}\xi_0, \lambda^{1 - {1}/{k}}\eta_0) $ gives an approximate blowup of $F_1 \in S(\Lambda, g_k)$. Since  $\partial^\alpha \partial_{t,x}\omega_\lambda = \Cal O(1)$ and $\partial^\alpha \partial_{y}\omega_\lambda = \Cal O(\lambda^{- {1}/{k}})$ for any~$ \alpha$ we obtain the following result from the chain rule.

\begin{rem}\label{symbrem}
If $0 < \varrho \le 1/2 $, $\omega_\lambda(t,x,y) $ is given by~\eqref{omegaexp} and $ a(t,x,y,\tau, \xi,\eta) \in S(\Lambda^m ,g_k)$ then $\lambda^{-m} a(t,x,y,\lambda\partial \omega_\lambda)
\in C^\infty$ uniformly.
\end{rem}

This gives that   $R_0(t,x,y) = F_0(t,x,y,{\lambda}\partial_{x,y}
{\omega}_\lambda)  $ is bounded in~\eqref{exp0}  and
$R_m(t,x,y,D_{t,x,y})$ are bounded differential operators of order $i$ in $t$,
order $j$ in $x$ and order $\ell$ in $y$, where $i + j + \ell \le m+2$ for $m > 0$.
In fact, derivatives in $\tau$ and $\xi$ of $ F_1 \in S(\Lambda, g_k)$ lowers the order of $\lambda$ by one, but derivatives in $\eta$ lowers the order only by $1 - 1/k$ until we have taken $k$ derivatives, thereafter by $1$. 
Thus for $R_m$, which is the coefficient for $\lambda^{-m}$, we find that $-m \le 1 - i  - j - \ell(1 - 1/k) $ so that $i + j + \ell \le m + 1 + \ell/k \le m + 2$ for $\ell \le k$, else $-m \le 1 - i  - j - k(1 - 1/k) - (\ell -k) = 2 -i-j-\ell $  which also gives $i + j + \ell \le m + 2$.   For the term  $ R$ we shall use  the following result when $k < \infty $.

\begin{rem}\label{Rrem}
If $ R \in S(\Lambda^m,g_k) \subset S^m_{1- {1}/{k},0}$, $u_\lambda $ is given by~\eqref {udef} with phase function $\omega_\lambda $ in~\eqref{omegaexp} and
\begin{equation}
\Big \{(t,x,y,\lambda \partial_{t,x,y}\omega_\lambda) : (t,x,y) \in \bigcup_j \supp \phi_j \Big \} \bigcap \wf_{g_k}(R) = \emptyset \qquad \lambda \gg 1
\end{equation} 
then $R u_\lambda = \Cal O(\lambda^{-N})$, $\forall\,  N$. 
\end{rem}

In fact, by using the expansion~\eqref{trevesexp} we find that $\partial^\alpha R(t,x,y,\lambda\partial_{t,x,y} {\omega}_\lambda) = \Cal O(\lambda^{-N}) $ for any $ \alpha$ and $ N$  in a neighborhood of  the support of $ \phi_j $ for any~$ j$ when $  \lambda \gg 1$.

In the case $k = \kappa(\omega) = \infty$ or $\eta_0 = 0 $ we shall
use the phase function given by~\eqref{omegaexp01}, then
\begin{multline}\label{omega2}
{\omega}_\lambda(t,x, y) =  \w {x-x_0(t), \xi_0} + \lambda^{\varrho - 1}\w {y -y_0(t), \eta_0} \\ + \Cal O(|x-x_0(t)|^2) + \lambda^{\varrho - 1}\Cal O(|y-y_0(t)|^2)
\end{multline}
such that  $\partial_y {\omega}_\lambda = \lambda^{\varrho - 1} \big ( \eta_0 + \Cal O(|y - y_0(t)|)\big)$ with some $0 <  \varrho < 1$. If $R\in S^2 $ vanishes of infinite order at $\eta = 0 $
then  $\partial^\alpha R(t,x,y,\lambda\partial_{t,x,y} {\omega}_\lambda) = \Cal O(\lambda^{-N}) $ for any $ \alpha$ and $ N$.  
Thus, we get the expansion~\eqref{exp0} with bounded
$R_0= F_0(t,x,y,\lambda \partial_{x,y}{\omega}_\lambda)$ and
bounded differential operators $R_m(t,x,y,D_{t,x,y})$ of order $i$ in $t$,
order $j$ in $x$ and order $\ell$ in $y$, where $i + j + \ell \le m+2$ for $m > 0$. When $k < \infty $ this follows as before, and 
in the case $k = \infty $ we have that derivatives in $\tau,\, \xi$ and $\eta $ of $ F_1 \in S^1_{1,0}$ lowers the order of $\lambda$ by one. 
In that case, we find for  $R_m$ that $-m \le 1 - i  - j - \ell$ so that $i + j + \ell \le m + 1 $.

\begin{rem}\label{symbrem0}
If $0 < \varrho \le 1 $, $\omega_\lambda(t,x,y) $ is given by~\eqref{omegaexp01} and $ a(t,x,y,\tau, \xi,\eta) \in S^m_{1,0}$ then $\lambda^{-m} a(t,x,y,\lambda\partial \omega_\lambda)
\in C^\infty$ uniformly.
\end{rem}

This follows from the chain rule since $\partial^\alpha \partial \omega_\lambda = \Cal O(1)$  for any~$ \alpha$.

\section{The Eikonal Equation}\label{eiksect}

We shall solve the eikonal equation approximately, first in the case when $k = \kappa(\omega) < \infty $ and $ \eta_0 \ne 0 $. This equation is given by the highest order terms of~\eqref{exp0}: 
\begin{equation}\label{eikeq}
\lambda  \partial_t\omega_\lambda +  F_1(t,x,y, \lambda \partial_{x,y} {\omega}_\lambda)  
- i\lambda\partial_{\eta}^2 F_1(t,x,y,{\lambda}\partial_{x,y} {\omega}_\lambda)  \partial_{y}^2{\omega}_\lambda  = 0
\end{equation}
modulo $ \Cal O(1)$. 
Here $F_1 \in  S(\Lambda, g_k) $  satisfies $F_1 \circ \chi = f \in S^1 $ modulo $ S^{1 - {1}/{k}}$ when $ |\eta| \ls | \xi| ^{1 - {1}/{k}}$  by \eqref{f1comp} in Proposition~\ref{prepprop}. Thus if $\partial_y \omega_\lambda  = \Cal O(\lambda ^{- {1}/{k}})$ we obtain the blowup
\begin{equation}
 F_1(t,x,y, \lambda \partial_{x} {\omega}_\lambda,\lambda  \partial_{y} {\omega}_\lambda) \cong  \lambda  f(t,x,\partial_{x} {\omega}_\lambda, \lambda^{1/k} \partial_{y} {\omega}_\lambda)
\end{equation}
modulo terms that are $ \Cal O(\lambda^{1- {1}/{k}}) $. Now $\re f \equiv 0 $ when $ \eta = \eta_0$, $f$ vanishes on ${\Gamma}' = \set{(t,x_0, {\xi}_0,\eta_0): \ t \in
I'}$ and $t \mapsto \im f(t,x_0,{\xi}_0,\eta_0) \in S^1$ changes sign from $+$ to $-$ as $t$ increases in a neighborhood of $ I'$. We may choose coordinates so that $0 \in I' $ thus  $f(0,x_0,\xi_0,\eta_0) = 0 $. 
Observe that \eqref{deta0} (and  \eqref{ddeta0} if $k = 2 $) holds near $\Gamma' $.
If $|I' | \ne 0$ then by Proposition~\ref{prepprop} we
may assume that~$\partial_x^\alpha \partial_\xi ^\beta\partial_\eta ^\gamma f$ vanishes 
at~${\Gamma}'$, $ \forall \,  \alpha \, \beta$ and $|\gamma| \le 1$   when $k > 2 $ and  for $ \forall \,  \alpha \, \beta$ and $|\gamma| \le 2$  when $k = 2 $.  
We also have that
$F_1 - c\partial_\eta F_1 $ is constant in $y$ modulo  $S(1, g_k)$ when $ |\eta| \ls | \xi| ^{1 - {1}/{k}}$, where $c \in S(1,g_k)$ may depend on $y$.
The case when $\eta_0 =0 $, for example when $ k = \infty$, will be treated in Sect.~\ref{eta0}.

We shall choose the phase
function so that $\im {\omega}_\lambda \ge 0$, $\partial_x \re {\omega}_\lambda \ne 0 $ and $\partial_{x,y}^2 \im{\omega}_\lambda >
0$ near the interval. We shall adapt the method by
H\" ormander~\cite{ho:nec} to inhomogeneous
phase functions. 
The phase function ${\omega}_\lambda(t,x,y)$ is given by the expansion
\begin{multline}\label{omegaexp}
 {\omega}_\lambda (t,x,y) = w_0(t) + \w{{\xi}_0(t), x-x_0(t)} 
  + \lambda^{- {1}/{k}}\w{{\eta}_0(t), y- y_0(t)}  \\ + \sum_{2 \le i
   \le K} w_{i,0}(t) (x-x_0(t))^{i}/i!   + \lambda^{\varrho - 1} \sum_{\substack{2 \le i + j 
   \le K \\ j \ne 0}} w_{i,j}(t) (x-x_0(t))^{i}(y- y_0(t))^j/i! j!
\end{multline}
for sufficiently large $ K$, where we will choose $0 <  \varrho < 1/2$, $\xi_0(0) = \xi_0 \ne 0$,  $\im w_{2,0}(0)> 0$,  $\im w_{1,1}(0) = 0$ and $\im w_{0,2}(0) > 0$. This gives $ \partial_{x,y}^2 \im{\omega}_\lambda >
0$ when $ t= 0$ and $|x - x_0(0)| + | y- y_0(0)| \ll 1$ which then holds in a neighborhood.
Here we use the multilinear forms $w_{i,j} = \set {w_{\alpha,\beta}i!j!/\alpha ! \beta !}_{|\alpha| = i, |\beta| = j}$,  $ (x-x_0(t))^j = \set{ (x-x_0(t))^\alpha}_{|\alpha| = j}$ and $ (y-y_0(t))^j = \set{ (y-y_0(t))^\alpha}_{|\alpha| = j}$ to simplify the notation. Observe that $ x_0(t)$, $ y_0(t)$, $ \xi_0(t)$, $ \eta_0(t)$ and $ w_{j,k}(t)$ will depend uniformly  on $ \lambda$.

Putting ${\Delta}x = x-x_0(t)$ and $\Delta y = y - y_0(t)$ we find that 
\begin{multline}\label{dtomega}
 \partial_t{\omega}_\lambda(t,x,y) = w_0'(t)  - \w{x_0'(t) ,{\xi}_0(t) } -  \lambda^{-1/k}\w{y_0'(t) ,{\eta}_0(t) } \\+ \w{{\xi}'_0(t) - w_{2,0}(t)x_0'(t)- \lambda^{\varrho - 1} w_{1,1}(t)y_0'(t), \Delta x}  
 \\ + \w{\lambda^{- {1}/{k}}{\eta}'_0(t) - \lambda^{\varrho - 1} w_{1,1}(t)x_0'(t) - \lambda^{\varrho - 1} w_{0,2}(t)y_0'(t), \Delta y}
 \\ + \sum_{2 \le i  \le K}(w_{i,0}'(t)- w_{i+1,0}(t)x_0'(t) - \lambda^{\varrho - 1} w_{i,1}(t)y_0'(t)) (\Delta x)^{i}/i! 
 \\ + \lambda^{\varrho - 1} \sum_{\substack{2 \le i + j 
  \le K \\ j\ne 0}} (w_{i,j}'(t) - w_{i+1,j}(t)x_0'(t) - w_{i,j+1}(t)y_0'(t))  (\Delta x)^{i}(\Delta y)^j/i! j!
\end{multline}
where the terms $w_{i,j}(t) \equiv 0 $ for $i + j > K $. We have
\begin{multline}\label{dxomega}
\partial_{x} {\omega}_\lambda(t,x,y) = {\xi}_{0}(t) +
\sum_{1 \le i \le K-1}   w_{i + 1,0}(t)({\Delta}x)^{i}/i! \\  + \lambda^{\varrho - 1} \sum_{\substack{1 \le i +  j 
\le K -1  \\ j \ne 0}} w_{i+1,j}(t) ({\Delta}x)^{i}({\Delta}y)^j/i! j! 
 ={\xi}_{0}(t)  +  {\sigma}_{0}(t,x) +  \lambda^{\varrho - 1}{\sigma}_{1}(t,x,y)
\end{multline}
Here  ${\sigma}_0$ is a finite expansion in powers of
${\Delta}x$ and  ${\sigma}_1$ is a finite expansion in powers of
${\Delta}x $ and ${\Delta}y $.
Also
\begin{multline}\label{dyomega}
\partial_{y} {\omega}_\lambda(t,x,y) =  \lambda^{-{1}/{k}}{\eta}_0
+ \lambda^{\varrho - 1} \sum_{1 \le  i + j \le K-1 } w_{i ,j+1}(t) ({\Delta}x )^{i}({\Delta}y)^j/i! j! \\ =
 \lambda^{-{1}/{k}}\left ({\eta}_{0}(t)  +   \lambda^{{1}/{k} +  \varrho - 1}{\sigma}_{2}(t,x,y)\right)
\end{multline}
where  ${\sigma}_2$ is a finite expansion in powers of
${\Delta}x$ and $\Delta y $.

Since the phase function is complex valued, the values of the symbol will be given by a formal Taylor expansion at the real values.
Recall that  $ F_1\circ \chi \cong f$ modulo $S^{1 - {1}/{k}} $ so by
the expansion in Remark~\ref{detarem} we find
\begin{multline}\label{F1exp}
F_1(t,  x,y,  \lambda\partial_x{\omega}_\lambda, \lambda \partial_y{\omega}_\lambda)
\cong \lambda f(t,x,{\xi}_0+ {\sigma}_0 + \lambda^{\varrho - 1} \sigma_1, \eta_0 + \lambda^{{1}/{k} +\varrho - 1} \sigma_2)  \\
=  \lambda \sum_{\alpha,\beta } \partial_{\xi}^\alpha  \partial_{\eta}^\beta f(t, x,{\xi}_0 + \sigma_0, \eta_0 )(\lambda^{\varrho - 1} \sigma_1) ^\alpha(\lambda^{{1}/{k} +\varrho - 1} \sigma_2) ^\beta/\alpha! \beta!
\end{multline}
modulo $\Cal O(\lambda^{1 - {1}/{k}} )$,
which can then be expanded in $\Delta x $ and $\Delta y$. This expansion can be done for any derivative of $F_1 $, see for example~\eqref{detaf1}.
Observe that $F_1 - c\partial_\eta F_1 $ is constant in $y$ modulo  $S(1, g_k)$, where $ \partial_\eta F_1 \in  S(\Lambda^{1/k}, g_k) $ when $ |\eta| \ls | \xi| ^{1 - {1}/{k}}$ and $c \in S(1, g_k)$ may depend on~$y$.
Remark~\ref{detarem}  also gives that $\partial_\eta ^2F_1(t,  x,y,  \lambda\partial_{x,y}{\omega}_\lambda) $ is bounded and by~\eqref{dyomega} we find that $ \partial_y^2 \omega_\lambda = \Cal O(\lambda^{\varrho- 1})$ so the last term in~\eqref{eikeq} is  $ \Cal O(\lambda^{\varrho})$.

When $x = x_0$ and $y = y_0$ we obtain that
$
 F_1(t,  x_0,y_0, \lambda \partial_x{\omega}_\lambda, \lambda \partial_y{\omega}_\lambda)  \cong \lambda f(t,x_0,\xi_0,\eta_0)
$
modulo  $\Cal O(\lambda^{1-{1}/{k}}) $. 
Thus, taking the value of~\eqref{eikeq} and dividing by $\lambda $,  we obtain the equation
\begin{equation}
 w_0'(t) - \w{x_0'(t),{\xi}_0(t)}  +  f(t, x_0(t), {\xi}_0(t), \eta_0(t)) = 0
\end{equation}
modulo  $\Cal O(\lambda^{-1/k}) + \Cal O(\lambda^{{1}/{k} +  \varrho - 1})$.
By taking real and imaginary parts we obtain the equations
\begin{equation}\label{2a}
\left\{
\begin{aligned} 
&\re w'_0(t) =  \w{x_0'(t),{\xi}_0(t)} -  \re f(t, x_0(t), {\xi}_0(t),\eta_0(t))\\
&\im w_0'(t) = - \im f(t, x_0(t), {\xi}_0(t),\eta_0(t))
\end{aligned}
\right.
\end{equation}
modulo  $\Cal O(\lambda^{-1/k}) + \Cal O(\lambda^{{1}/{k} +  \varrho - 1})  = \Cal O(\lambda^{-\kappa})$ for some $\kappa > 0 $ since $\varrho < 1/2 $.
After choosing $w_0(0) $ this will determine $w_0$ when we have determined $(x_0(t) ,{\xi}_0(t),\eta_0(t) )$. 
In the following, we shall solve equations like~\eqref{2a} modulo $\Cal O(\lambda^{-\kappa}) $ for some $\kappa > 0 $, which will give the asymptotic solutions when  $ \lambda \to \infty$.

Using \eqref{F1exp}, we find that the first order terms in $\Delta x$ of~\eqref{eikeq} will similarly be zero if 
\begin{multline}\label{22}
{\xi}_0'(t) -  w_{2,0}(t)x_0'(t) +  \partial_xf(t,x_0(t),{\xi}_0(t),\eta_0(t)) \\ +
\partial _{\xi}f(t,x_0(t),{\xi}_0(t),\eta_0(t))w_{2,0}(t) = 0 
 \end{multline}
modulo  $ \Cal O(\lambda^{-\kappa})$ for some $\kappa > 0 $. By taking real and imaginary parts we find that~\eqref{22} gives that
\begin{equation}\label{2}
\left\{
\begin{aligned} 
&{\xi}_0' = \re w_{2,0} x_0' - \re \partial_x  f+ \im \partial _{\xi} f \im w_{2,0} -  \re \partial _{\xi} f \re w_{2,0} 
\\
&\im w_{2,0}x_0' = \im \partial_x  f + \im \partial
_{\xi} f \re w_{2,0} + \re  \partial
_{\xi}f \im w_{2,0}
\end{aligned}
\right.
\end{equation}
modulo  $ \Cal O(\lambda^{-\kappa})$. 
Here and in what follows, the values of the symbols are taken at $(t, x_0(t), y_0(t), {\xi}_0(t),\eta_0(t)) $.
We shall put $(x_0(0), {\xi}_0(0)) = (x_0, {\xi}_0)$, which will determine $x_0(t)$ and
${\xi}_0(t)$ if $|\im w_{2}(t)| \ne 0$. 

Similarly, the second order terms in $\Delta x$  of~\eqref{eikeq} vanish if we solve
\begin{equation*}
w'_{2,0}/2 - w_{3,0}x_0'/2
+  \partial_{\xi}f w_{3,0}/2 +  \partial_x^2 f/2 + 
\Re  \left( \partial_x\partial_{\xi}f w_{2,0}\right) + w_{2,0} \partial_{\xi}^2 f w_{2,0}/2 = 0
\end{equation*}
modulo  $ \Cal O(\lambda^{-\kappa})$ for some $\kappa > 0 $ where $\Re A = \frac{1}{2}(A + A^t)$ is the symmetric part of $ A$.
This gives
\begin{equation}\label{w2eq}
 w'_{2,0}  =  w_{3,0}x_0'
- \left(\partial_{\xi}f w_{3,0}
 + \partial_x^2 f  + 
2\Re \left(\partial_x\partial_{\xi}f w_{2,0}\right) + w_{2,0} \partial_{\xi}^2 f w_{2,0}\right) 
\end{equation}
with initial data $ w_{2,0}(0) $ such that 
$\im w_{2,0}(0) > 0$, which then holds in a neighborhood. 
Similarly, for $ j> 2 $ we obtain 
\begin{equation}\label{wjeq}
w'_{j,0}(t)  =  w_{j+1,0}(t) x_0'(t) 
- \Big(f\big(t, x, \xi_0(t) + \sigma_0(t,x) , \eta_0(t) \big)\Big)_j 
\end{equation}
modulo  $ \Cal O(\lambda^{-\kappa})$, where we have taken the $ j$:th term of the expansion in $ \Delta x$. Observe that \eqref{2}--\eqref{wjeq} only involve $x_0 $, $ \xi_0$ and $w_{j,0} $ with $j \le K $.

Next, we will study the $ y$ dependent terms. Then, we have to expand $F_1 \circ \chi \cong f + r $ modulo $S^{1 - {2}/{k}} $ where $ r \in S^{1 - {1}/{k}}$ is independent of $ y$ and $\partial_{{\eta}} r = 0 $ when $ f$ vanishes (of infinite order).
By expanding $ r$, we get \eqref{F1exp} with $ f$ replaced by $ r$ and $ \lambda $ replaced by $ \lambda^{1 - {1}/{k}}$.
Since  $\partial_\eta F_1  \circ \chi \cong   \lambda^{{1}/{k}} \partial_\eta f $ modulo $S^0$ we obtain modulo bounded terms  that
\begin{multline}\label{detaf1}
\partial_\eta F_1(t,x,y, \lambda \partial_{x,y} {\omega}_\lambda)  \cong  \lambda^{{1}/{k}} \partial_\eta f(t,x_0, \xi_0+ \sigma_0 + \lambda^{\varrho -1}\sigma_1, \eta_0 + \lambda^{{1}/{k} + \varrho -1}\sigma_2) 
\\ \cong \lambda^{{1}/{k} } \partial_\eta f(t,x_0, \xi_0+ \sigma_0, \eta_0) +  \lambda^{{2}/{k}
 + \varrho - 1}\partial_\eta^2 f(t,x_0, \xi_0+ \sigma_0 , \eta_0) \sigma_2 
\end{multline}
modulo $\Cal O(\lambda^{{1}/{k}+  2\varrho - 1})$. Similarly we obtain that 
\begin{equation*}
\partial_\eta^2 F_1(t,x,y, \lambda \partial_{x,y} {\omega}_\lambda)  \cong  \lambda^{{2}/{k} - 1}\partial_\eta^2 f(t,x_0, \xi_0+ \sigma_0 , \eta_0) 
\end{equation*}
modulo  $\Cal O(\lambda^{{1}/{k}+  \varrho - 1})$.
Recall that  $\partial_y F_1 \cong  \partial_y c \partial_\eta F_1  \in S(\Lambda^{{1}/{k}} , g_k)$ modulo $ S(1,g_k)$ when $ |\eta| \ls | \xi| ^{1 - {1}/{k}}$, which gives that  $\partial_y^\alpha F_1 \cong  \partial_y^\alpha c \partial_\eta F_1 $ modulo $ S(1,g_k)$ for any $\alpha $, see \eqref{dyf1}.
Using \eqref{F1exp} and \eqref{detaf1}  we find that the coefficients of the first order terms in $\Delta y$ of~\eqref{eikeq} are
\begin{multline}\label{etaeq}
\lambda^{1- {1}/{k}}\eta_0'   - \lambda^{\varrho }w_{0,2}y_0' - \lambda^{\varrho }w_{1,1}x_0'    + \lambda^{\varrho } \partial_\eta r w_{0,2}  +  \lambda^{\varrho } \partial_\xi f w_{1,1} 
\\ + \lambda^{{1}/{k} + \varrho } \partial_\eta f w_{0,2} +
\lambda^{{1}/{k}} \partial_y c \partial_\eta f   - i \lambda^{{2}/{k} + \varrho -1}\partial_\eta^2 f w_{0,3}
\end{multline}
modulo  $\Cal O(1) $  for some $c \in S^0 $. Here and in what follows, the values of the symbols are taken at $(t, x_0(t), y_0(t), {\xi}_0(t) + \sigma_0(t,x),\eta_0(t)) $
By taking the real and imaginary parts we obtain the equations
\begin{equation}\label{detaeq}
\eta_0'  = 
- \lambda^{{2}/{k} + \varrho -1}\re \partial_\eta f  w_{0,2} -  \lambda^{{2}/{k} - 1} \re  \partial_y c \partial_\eta f  
\end{equation}
modulo  $\Cal O(\lambda^{{1}/{k}+ \varrho - 1} )$, and
\begin{multline}\label{detaeq0}
 \im w_{0,2}y_0' = - \im w_{1,1}x_0'   +  \im  \partial_\eta r w_{0,2} + \im \partial_\xi f  w_{1,1}   \\ 
 + \lambda^{1/k}\im  \partial_\eta f  w_{0,2}   + \lambda^{{1}/{k}- \varrho}\im \partial_y c \partial_\eta f    - \lambda^{{2}/{k} -1}\re \partial_\eta^2 f w_{0,3}
\end{multline}
modulo  $\Cal O(\lambda^{-\varrho}) $.
This gives that $ \eta'_0(0) = \Cal O(\lambda^{-\kappa})$ for some $ \kappa > 0$, since~\eqref{deta0} gives  $ \partial_\eta f(0,x_0,\xi_0,\eta_0) = 0$. 
We will choose initial data $\eta_0(0) = \eta_0 $, $y_0(0) = y_0 $ and $ \im w_{0,2}(0) > 0$, then $y_0' $ is well defined in a neighborhood.
In order to control the unbounded terms in~\eqref{detaeq} and \eqref{detaeq0}, we shall use scaling and  \eqref{deta0}. Therefore we let 
\begin{equation}
\zeta_0(t) = \lambda^{1 - {1}/{k}- \varrho}\big (\eta_0(t) - \eta_0 \big )
\end{equation}
where $\eta_0 = \eta_0(0) $. Then we get from \eqref{detaeq} that
\begin{equation}\label{dzetaeq}
\zeta_0'(t) = - \lambda^{{1}/{k}}\re \partial_\eta f  w_{0,2} -  \lambda^{{1}/{k} - \varrho} \re  \partial_y c \partial_\eta f  
\end{equation}
modulo bounded terms. Expanding~\eqref{dzetaeq} in $\eta_0(t) =  \lambda^{{1}/{k} + \varrho -1} \zeta_0(t) + \eta_0 $ we find
\begin{multline}\label{zeta0eq}
\zeta_0'(t) = - \lambda^{{1}/{k}}\re \partial_\eta f_0  w_{0,2} -  \lambda^{{1}/{k} - \varrho} \re  \partial_y c \partial_\eta f_0  - \lambda^{{2}/{k} + \varrho - 1} \re \zeta_0 \partial_\eta^2 f_0  w_{0,2}
\\ -  \lambda^{{2}/{k} -1} \re \big( \zeta_0 \partial_\eta\partial_y c \partial_\eta f_0   +  \partial_y c \partial_\eta^2 f_0 \zeta_0 \big)
\end{multline}
modulo $\Cal O(\lambda^{-\kappa}) $ for some $\kappa > 0 $ where $\partial_\eta^j f_0 = \partial_\eta^j f\restr {\eta = \eta_0}$ for $j \ge 0 $.
We shall use the following result.

\begin{lem}\label{lemclaim}
Assume $ k = \kappa(\omega) < \infty$, $ \varepsilon$ is given by \eqref{deta0}--\eqref{ddeta0} and $ \im w_0(t) \ge  0$ is the solution to $ \im w_0'(t) = - \im f (t, x_0(t), \xi_0(t), \eta_0) $ with $ \im w_0(0) = 0$. 

If  \eqref{deta0} holds, then for any $ \delta < \min \left ({\varepsilon},1 - \frac{1}{k} \right ) $, $ \alpha $ and  $\beta $, there exists $ \kappa > 0$ and $C \ge 1$ with the property that if
\begin{equation}\label{detaass1}
\left|\int_0^t \big| \lambda^{1/k}\partial_x^\alpha\partial_\xi^\beta \partial_\eta f(s,x_0(s),\xi_0(s),\eta_0)
\big| \, ds\right| \ge C\lambda^{-\delta} 
\end{equation}
with $ \lambda \ge C$, then $\lambda  \im w_0(s) \ge
{\lambda}^{\kappa}/C$ for some $s$ in the interval
connecting $0$ and $t$.

If  $ k = 2$  and \eqref{ddeta0} holds, then for any $ \delta < \varepsilon $, $ \alpha $ and  $\beta $, there exists $ \kappa > 0$ and $C \ge 1$ with the property that if
\begin{equation}\label{detaass2}
\left| \int_0^t \big| \lambda^{{2}/{k}-1} \partial_x^\alpha\partial_\xi^\beta \partial_\eta^2 f(s,x_0(s),\xi_0(s),\eta_0) \big| \, ds\right| \ge C \lambda^{-\delta}  \qquad \forall\,\alpha\, \beta
\end{equation}
with $ \lambda \ge C$, then $\lambda  \im w_0(s) \ge
{\lambda}^{\kappa}/C$ for some $s$ in the interval
connecting $0$ and $t$.
\end{lem}

Observe that $\eta_0 $ is constant in Lemma~\ref{lemclaim}, but we have to show that $\im w_0(t)$ has the minimum 0 at $t= 0 $.
Lemma~\ref{lemclaim} will be proved in Sect.~\ref{lempf}.
Clearly, \eqref{detaass2} cannot hold if $k > 2 $ and $\delta < 1/3 $.
We shall use Lemma~\ref{lemclaim} for a fixed $ \delta > 0$ and for $ |\alpha| + | \beta| \le N$.
Since we only have to integrate the eikonal equations in the interval where $\lambda \im w_0 \ls {\lambda}^{\kappa}$ for some $ \kappa > 0 $ and sufficiently large  $ \lambda $, we may assume the integrals in \eqref{detaass1} and~\eqref{detaass2}  are $ \Cal O(\lambda^{-\delta}) $ when $ \lambda \to \infty$.
Using this and integrating \eqref{zeta0eq} we find that $ \zeta_0 = \Cal O( \lambda^{-\kappa})$ for some $ \kappa > 0$
if $\varrho \ll 1$, $\im w_0(t) \ge 0 $  and the coefficients $w_{i,j} $ are bounded. This gives a constant asymptotic solution $ \eta_0$.

\begin{rem}\label{rrem}
If $\zeta_0(t) =  \lambda^{1 - {1}/{k}- \varrho}\big (\eta_0(t) - \eta_0 \big)$ then we have
\begin{multline}
\partial_x^\alpha\partial_\xi^\beta f(t, x_0(t), {\xi}_0(t),\eta_0(t)) \cong \partial_x^\alpha\partial_\xi^\beta f(t, x_0(t), {\xi}_0(t),\eta_0)  \\+ \lambda^{{1}/{k} + \varrho - 1} \partial_x^\alpha\partial_\xi^\beta \partial_\eta f(t,x_0(t),\xi_0(t),\eta_0)\zeta_0(t) \\ + \lambda^{{2}/{k} 
+ 2\varrho - 2}\partial_x^\alpha\partial_\xi^\beta \partial_\eta^2 f(t,x_0(t),\xi_0(t),\eta_0) \zeta_0^2(t)/2
\end{multline}
for $\varrho \ll 1 $ and $  t  \in I  $ modulo $\Cal O \big (\lambda^{- 1 - \kappa} |\zeta_0(t)|^3 \big )$ for some $\kappa > 0 $. We also  obtain that 
\begin{multline}
\partial_x^\alpha\partial_\xi^\beta \partial_\eta f(t,x_0(t),\xi_0(t),\eta_0(t)) \cong \partial_x^\alpha\partial_\xi^\beta \partial_\eta f(t,x_0(t),\xi_0(t),\eta_0) \\ + \lambda^{{1}/{k} + \varrho - 1}\partial_x^\alpha\partial_\xi^\beta \partial_\eta^2 f(t,x_0(t),\xi_0(t),\eta_0) \zeta_0(t) 
\end{multline}
for $\varrho \ll 1 $ and $  t  \in I  $ modulo $\Cal O \big (\lambda^{2\varrho - 1} |\zeta_0(t)|^2\big )$  and 
\begin{equation}
\partial_x^\alpha\partial_\xi^\beta \partial_\eta^2 f(t,x_0(t),\xi_0(t),\eta_0(t)) \cong \partial_x^\alpha\partial_\xi^\beta \partial_\eta^2 f(t,x_0(t),\xi_0(t),\eta_0) 
\end{equation}
when $\varrho \ll 1 $ and $  t  \in I  $ modulo $\Cal O \big (\lambda^{ \varrho - {1}/{2}}  |\zeta_0(t)| \big )$.
\end{rem}

If $ \zeta_0$ is bounded and the integrals in \eqref{detaass1} and~\eqref{detaass2} are $ \Cal O(\lambda^{-\delta}) $   then for $\delta $ and $\varrho  $ small enough we may replace $\eta_0 $ with $\eta_0(t) $ in these integrals with a smaller $ \delta > 0$. In fact, then the change in~\eqref{detaass2} is $ \Cal O(\lambda^{\varrho -1/2}) $  and the change in~\eqref{detaass1} is $ \Cal O(\lambda^{\varrho -\delta} + \lambda^{2\varrho - {1}/{2}}) $ using~\eqref{detaass2}. 
Observe that we only need that $\delta > 0 $ for the proof, but we have to show that  $\im w_0(t)$ has the minimum 0 at $t= 0 $ which will be done later.

Using  \eqref{F1exp} and \eqref{detaf1} we find that the second order terms in $\Delta  y$ of~\eqref{eikeq} vanish if
\begin{multline}
\lambda^{\varrho } \big (w_{0,2}'  - w_{1,2}x_0' - w_{0,3}y_0'  +  \partial_\eta r w_{0,3}  +  \partial_\xi f w_{1,2} \big)  +  \lambda^{{1}/{k}+ \varrho  } \partial_\eta f w_{0,3} 
\\ + \lambda^{{2}/{k} + 2\varrho - 1} w_{0,2} \partial_\eta^2 f w_{0,2}
+   \lambda^{{1}/{k}} \partial_y^2 c \partial_\eta f  
 + 2\lambda^{{2}/{k} + \varrho - 1}    \partial_y c \partial_\eta^2 f w_{0,2} - i \lambda^{{2}/{k} + \varrho -1} \partial_\eta^2 f w_{0,4}= 0
\end{multline}
modulo $ \Cal O(1) $  if $\varrho \ll 1 $.
This gives
\begin{multline}\label{w02eq}
w_{0,2}'  = w_{1,2}x_0' +  w_{0,3}y_0'  -  \partial_\eta r w_{0,3} -  \partial_\xi f w_{1,2}
  - \lambda^{{1}/{k}} \partial_\eta f w_{0,3} 
\\ -  \lambda^{{2}/{k} + \varrho -1} w_{0,2} \partial_\eta^2 f w_{0,2} 
 -   \lambda^{{1}/{k}- \varrho} \partial_y^2 c \partial_\eta f 
  - 2\lambda^{{2}/{k} - 1}  \partial_y c \partial_\eta^2 f w_{0,2} + i \lambda^{{2}/{k} -1} \partial_\eta^2 f w_{0,4}
\end{multline}
modulo $ \Cal O(\lambda^{-\varrho})$ if $\varrho  \ll 1 $. 
By using Lemma~\ref{lemclaim} we may assume that the coefficients in the right hand side are uniformly integrable when $\varrho $ is small enough.
We choose the initial value so that $\im w_{0,2}(0) >0$ which then holds in a neighborhood.

Similarly, the coefficients for the term $ \Delta x^j \Delta y^\ell $ in~\eqref{eikeq} can be found from the expansion in  $\Delta x $ and $ \Delta y$ of
\begin{multline}
\lambda^{\varrho}\left( \sum_{j ,\ell  \ne 0}(w_{j ,\ell }'  - w_{j +1,\ell }x_0'  - w_{j ,\ell +1}y_0') \Delta x^j  \Delta y^\ell /j !\ell ! + \partial_\eta r \sigma_2 + \partial_\xi f \sigma_1 \right) 
\\ + \lambda^{{1}/{k} + \varrho } \partial_\eta f \sigma_2 
+ \lambda^{{2}/{k} + 2\varrho -1} \sigma_2 \partial_\eta^2 f \sigma_2/2 
+   \lambda^{{1}/{k}}   \partial_y c \Delta y   \partial_\eta f 
\\ + \lambda^{{2}/{k} + \varrho - 1}    \partial_y c \Delta y \partial_\eta^2 f \sigma_2 - i\lambda^{{2}/{k} + \varrho -1}\partial_\eta^2 f \partial_y \sigma_2 
\end{multline}
modulo $ \Cal O(1)$ if $\varrho  \ll 1 $. 
Here the values of the symbols are taken at $(t, x_0(t), y_0(t), {\xi}_0(t) + \sigma_0(t,x),\eta_0(t)) $, so the last terms can be expanded in~$ \Delta x$ and  $ \Delta y$ which also  involves the $\xi $ derivatives.
Taking the coefficient for $ \Delta x^j  \Delta y^\ell $ and dividing by $ \lambda^{\varrho}$ we obtain that these terms vanish if
\begin{multline}\label{wjleq}
w_{j ,\ell }'  =  w_{j +1,\ell }x_0' + w_{j ,\ell +1}y_0'  -  j ! \, \ell ! \Big(  \partial_\eta r \sigma_2  + \partial_\xi f \sigma_1 + \lambda^{{1}/{k} } \partial_\eta f \sigma_2  \\
+ \lambda^{{2}/{k} + \varrho -1} \sigma_2 \partial_\eta^2 f \sigma_2/2 
+   \lambda^{{1}/{k} - \varrho}    \partial_y c \Delta y  \partial_\eta f  
 + \lambda^{{2}/{k} - 1}    \partial_y c \Delta y \partial_\eta^2 f \sigma_2  - i \lambda^{{2}/{k} -1}\partial_\eta^2 f \partial_y \sigma_2\Big)_{j ,\ell }
\end{multline}
modulo  $\Cal O(\lambda^{-\varrho})$, where in the right hand side we have taken the coefficient of $  \Delta x^j  \Delta y^\ell $, expanding $ \partial_yc$, $\partial_\eta r $, $\partial_\xi f $,  $\partial_\eta f $ and  $\partial_\eta ^2 f $ in $ \Delta x$ and $ \Delta y$ which also involves the $\xi $ derivatives. We choose  initial values  $w_{j ,\ell }(0) = 0 $ for  $j  = \ell = 1 $ and $j  + \ell  > 2 $. Observe that the Lagrange error term in the Taylor expansion of   \eqref{F1exp} is $ \Cal O(\lambda (|x-x_0(t)| + |y - y_0(t)|)^{K+1})$.

Now assume that $\im w_0(t) \ge 0$ is a solution to $ \im w_0'(t) = - \im f (t, x_0(t), \xi_0(t), \eta_0) $ with $ \im w_0(0) = 0$, thus $ \im w_0(t)$ has a minimum at $t= 0 $.
The  equations~\eqref{2a}, \eqref{2}--\eqref{wjeq}, \eqref{detaeq0}, \eqref{zeta0eq}, \eqref{w02eq}  and~\eqref{wjleq} 
form a system of  nonlinear ODE for $x_0(t),$ $ y_0(t),$ $ \xi_0(t) ,$ $\zeta_0(t) $ and~$ w_{j,\ell }(t)$ for $ j + \ell \le K $. 
By Remark~\ref{rrem} we can then replace $\eta_0(t)$ in  $ f$  by $ \eta_0 = \eta_0(0)$ when $ \varrho \ll 1$.
Since we only have to integrate where $\im w_0(t) \ls \lambda^{\kappa -1}$ for some $ \kappa > 0$, this system  has uniformly integrable coefficients by  Lemma~\ref{lemclaim} for $\varrho \ll 1 $, which gives a local solution near $(0,x_0,y_0,{\xi}_0, \eta_0) $. 

In the case when ${\Gamma}' = \set{(t,x_0,y_0,{\xi}_0,\eta_0): \ t \in I'}$ for  $|I' |\ne 0$,  we shall use the following definition.

\begin{defn}\label{defI}
	For $ a(t) \in L^\infty(\br) $ and $ \kappa \in \br $ we say that $a(t) \in I(\lambda^\kappa)$ if $\int_0^t
	a(s)\,ds = \Cal O (\lambda^\kappa)$ uniformly for all $t \in I$ when $\lambda \gg 1 $.
\end{defn}

We have assumed that 
$ \partial_{x}^{{\alpha}} \partial_{{\xi}}^{{\beta}} f(t,x_0,{\xi}_0,\eta_0) = 0$, $\forall \, \alpha\, \beta $, for $t \in I' $.
Let $I $ be the interval containing $ I'$ such that $  \partial_{x}^{{\alpha}} \partial_{{\xi}}^{{\beta}} \partial_{\eta} f \in I(\lambda^{-1/k - \delta})$ and $  \partial_{x}^{{\alpha}} \partial_{{\xi}}^{{\beta}} \partial_{\eta}^2 f \in I(\lambda^{1-2/k - \delta})$ for some $\delta > 0 $ by  Lemma~\ref{lemclaim}.
We obtain that $x_0'  = \xi_0' = 0$ on~$I' $  by~\eqref{2} which gives $ w_0'  = 0$  on $I' $  by \eqref{2a}. We also obtain $ \eta_0' \cong 0 $ modulo   $I(\lambda^{-\kappa})$  for some $\kappa > 0  $ by \eqref{detaeq}, since all the coefficients are in $ I(\lambda^{-\kappa})$.  We find from equations~\eqref{w2eq} and~\eqref{wjeq}  that $ w_{j,0}'  =   0$  on $I' $  for $j \ge 2 $.
By \eqref{wjleq} we find when $\ell  > 0 $ that
\begin{equation}
w_{j, \ell}' \cong  w_{j, \ell+1}\big(y_0' - j! \, \ell ! \, \partial_\eta r \big)    \qquad \text {on $I' $ modulo   $I(\lambda^{-\kappa})$}
\end{equation}  
where $y'_0 \in  I(1) $. Since $w_{j, \ell} \equiv 0 $ when $ j + \ell > K$ and $w_{ j, \ell}(0) = 0 $ when $ j + \ell  > 2 $ we recursively find that  $w_{j, \ell} \cong 0 $ when $  j + \ell  > 2 $ and $w_{j, \ell}'  \cong  0 $ when  $ j + \ell =  2 $ modulo   $I(\lambda^{-\kappa})$. By~\eqref{detaeq0} we find that 
\begin{equation}
y_0' \cong (\im w_{0,2}(0))^{-1}  \im \partial_\eta r w_{0,2}(0)  \qquad \text {on $I' $ modulo   $I(\lambda^{-\kappa})$}
\end{equation}
which gives $y_0' = o(1)$ on $ I$.
In fact, we assume that $\partial_\eta r = 0$ when  $ \im f$ vanishes of infinite order. We may choose $I' $ as the largest interval containing~0 such that $w_0  $ vanishes on $I' $. 
Then in any neighborhood of an endpoint of $ I' $ there exists points where $w_0 > c \ge \lambda^{\kappa -1} $ for $ \lambda \gg 1$.

Now  $f $ and $ r$ are independent of $ y$ near the semibicharacteristic, so the  coefficients of the system of equations are independent of $y_0(t)$ modulo $  I(\lambda^{- \kappa})$. (If the symbols are independent of $ y$ in an arbitrarily large $ y$ neighborhood we don't need the vanishing condition on $ \partial_\eta r$.) 
Thus, we obtain the solution $\omega_\lambda $ in a neighborhood of ${\gamma}' = \set {(t,x_0(t),y_0(t)): \ t \in I'}$ for $\lambda \gg 1 $. 
In fact, by scaling we see that the $ I(\lambda^{-\kappa}) $ perturbations do not change the local solvability of the ordinary differential equation for large enough~$ \lambda$. 

But we also have to show that $t \mapsto \im f(t,x_0(t), {\xi}_0(t),\eta_0) = f_0(t)$
changes sign from $+$ to $-$ as $t$ increases
for some choice of initial data $ (x_0,{\xi}_0) $ and $w_{2,0}(0) $. Then we obtain that  $\im w_0(t) \ge 0$ for the solution to $ \im w_0'(t) = - \im f (t, x_0(t), \xi_0(t), \eta_0) $ with suitable initial data.
 Observe that~\eqref{2}--\eqref{wjeq} only involve $x_0 $, $ \xi_0$ and $w_{j,0} $ with $j \le K $ and are uniformly integrable.

First we shall consider the case when $t \mapsto \im f(t,x_0, {\xi}_0,\eta_0)$ changes sign from $ +$ to $ -$ of first order. Then
\begin{multline}\label{df0}
f_0'(0) = \im \partial_t f (0, x_0,{\xi}_0, \eta_0)  + \im \partial_x f (0, x_0,{\xi}_0, \eta_0) x_0'(0)  \\ + \im \partial_\xi f (0, x_0,{\xi}_0, \eta_0)  \xi_0'(0) 
\end{multline}
where $\im  \partial_t f (0, x_0,{\xi}_0, \eta_0) < 0 $.

\begin{rem}\label{initclaim}
From \eqref{2} we find that we may choose $ w_{2,0}(0) $ so that $\im w_{2,0}(0) > 0$ and $|(x_0'(0), \xi_0'(0))| \ll 1$.
\end{rem}

In fact, if $ \im  \partial _{\xi}f \ne 0$ then we may choose $ \re w_{2,0}(0) $ 
so that 
$ 
\im \partial_x f + \im \partial_\xi f  \re w_{2,0} = 0
$ 
at $ (0,x_0,{\xi}_0,\eta_0) $. Since $\re f \equiv 0$ when $\eta = \eta_0 $ we find from~\eqref{2} that $x_0'(0) = 0 $ and 
\begin{equation*}
\xi_0'(0) =  \im \partial_\xi  f  (0,x_0,\xi_0, \eta_0)\im  w_{2,0}(0)  =  o(1)
\end{equation*} 
if $ \im  w_{2,0}(0) \ll 1$. 
On the other hand, if  $ \im  \partial _{\xi}f(t,x_0,{\xi}_0,\eta_0) = 0$ then by putting $\re w_{2,0}(0)  = 0 $ we find  from~\eqref{2} that $\xi_0'(0) = 0 $ and if $ \im  w_{2,0}(0) \gg 1$ we obtain
\begin{equation*}
x_0'(0) = (  \im  w_{2,0}(0) )^{-1} \im \partial_x f(0,x_0,{\xi}_0,\eta_0)  =  o(1)
\end{equation*} 

If  $(x'_0(0), \xi'_0(0))  = o(1)  $ then $f_0'(0)  < 0 $ by \eqref{df0}  so $t \mapsto f_0(t) $ has a sign change from $ +$ to $ -$ of first order at $ t = 0$.
By~\eqref{2a} we obtain that the asymptotic solution $t \mapsto \im
w_0(t)$ has a local minimum on~$I$, which is also true when $\lambda \gg 1 $, and 
the minima can be made equal to $0$ by subtracting a constant depending on $\lambda $.

We also have to consider the general case when $t \mapsto \im f(t,x_0,{\xi}_0,\eta_0)$
changes sign from $+$ to $-$ of higher order as $ t$ increases near $ I'$. If  there exist points  in any $(x,\xi) $ neighborhood
of~${\Gamma}' $ for $\eta=  {\eta_0}$
where $\im f = 0 $ and $\partial_t \im f < 0$, then by changing the initial data we can as before construct approximate solutions  for which $t \mapsto
\im w_0(t)$ has a local minimum equal to 0 on~$I$ when $\lambda \gg 1 $.
Otherwise, $\partial_t \im f \ge 0$ when $\im f = 0$ in some $(x,\xi) $ neighborhood of~${\Gamma}' $ for $\eta = \eta_0$. Then we take the asymptotic solution $ (w(t), w_{j,0}(t) ) = (x_0(t),\xi_0(t), w_{j,0}(t) )$ to \eqref{2}--\eqref{wjeq} when $\lambda \to \infty $ with $ \eta_0(t) \equiv \eta_0$ and initial data $w(0) = (x,{\xi})$ but fixed  $w_{j,0}(0)$. This gives a 
change of coordinates $(t,w) \mapsto (t,w(t))$ near $\gamma'  = \set{(t,x_0,{\xi}_0): \ t \in I'}$ if $\partial_{x}^{{\alpha}} \partial_{{\xi}}^{{\beta}} f(t,x_0,{\xi}_0,\eta_0) = 0 $ when $ t \in I'$. In fact, the solution is constant on $\Gamma' $ since all the coefficients of \eqref{2}--\eqref{wjeq} vanish there.
By the invariance of  
condition (${\Psi}$) there will still exist a change of sign  from 
$+$ to $-$ of $t
\mapsto \im f(t,w(t),\eta_0)$ in any neighborhood of~${\Gamma}'$  after the change of
coordinates, see~\cite[Theorem~26.4.12]{ho:yellow}. (Recall that 
conditions  \eqref{kcond}, \eqref{cond2} and \eqref{cond1} hold in some
neighborhood of ${\Gamma}'$.)  By choosing suitable initial values
$(x_0,{\xi}_0)$ at $t = t_0$ we obtain that
$\im w_0'(t) = - \im f(t,w(t),\eta_0)$ has a sign change from $ -$ to $ +$ and
$t \mapsto \im w_0(t) $ has a
local minimum on  $I$ for $\lambda \gg 1 $, which can be assumed to be
equal to 0 after subtraction. 
Thus, we obtain that
\begin{equation}\label{11}
e^{i{\lambda}{\omega}_\lambda(t,x)} \le e^{-c({\lambda}(\im w_0(t) + c|{\Delta}|^2) + {\lambda}^\varrho|{\Delta}y |^2)}
\qquad  |{\Delta}x + |{\Delta}y | \ll 1
\end{equation}
where $\min_I \im w_0(t) = 0$ with $\im w_0(t) >
0$ for $t \in \partial I$.
This gives an approximate solution to~\eqref{eikeq}, and summing up, we have proved the following result.

\begin{prop}\label{eikprop0}
Let \/ ${\Gamma}' = \set{(t,x_0,y_0; 0,{\xi}_0, \eta_0): \ t \in I' }$ so that
$ 
\partial_{x}^{{\alpha}} \partial_{{\xi}}^{{\beta}} f(t,x_0,{\xi}_0,\eta_0) = 0$, $\forall \, \alpha\, \beta$, for $t \in I' $ if  $ |I'| \ne 0$.
Then for  $ \varrho > 0$ small enough
we may solve~\eqref{eikeq} modulo terms that are $ \Cal O(1)$
with ${\omega}_\lambda (t,x)$  given by~\eqref{omegaexp} in a neighborhood of
${\gamma}' = \set {(t,x_0(t),y_0(t)): \ t \in I'}$ modulo $\Cal O( \lambda(|x-x_0(t)| + |y - y_0(t)|)^M)$, $\forall\, M$, such that when $t \in
I' $ we have that $(x_0(t), {\xi}_0(t), \eta_0(t) ) = (x_0,{\xi}_0, \eta_0)$,  $w_0(t) = 0$, $w_{1,1}(t) \cong 0$  and $w_{j,k}(t) \cong 0$ for $j + k > 2$ modulo $ \Cal O(\lambda^{-\kappa})$ for some $ \kappa > 0$, 
$\im w_{2,0}(t) > 0$ and $\im w_{0,2}(t) > 0$. 

Assume that $t \mapsto f(t,x_0,{\xi}_0,\eta_0)$ changes sign from $+$ to $-$
as $t$ increases near~$I'$. Then by changing the initial values we may obtain that
the curve $t \mapsto (t,x_0(t),y_0(t);0, {\xi}_0(t),\eta_0(t))$, $t \in I$, is
arbitrarily close to~${\Gamma}$,  $\min_{t \in I}
\im w_0(t) = 0$ and $\im w_0(t)  > 0$ for $t \in \partial I$.
\end{prop}

Since \eqref{11} holds near $\Gamma' $ the errors in the eikonal
equation will give terms that are bounded by $C_M{\lambda}^{1-M\varrho/2}$, $\forall \, M $.
Observe that cutting off where $\im w_0 > 0$ will give errors that are
$\Cal O({\lambda}^{-M})$, $\forall\, M$.

\section{The bicharacteristics on $ \Sigma_2$}\label{eta0} 

We shall also consider the case when $\eta_0 = 0 $ on the bicharacteristics, including the case $k = \infty $. As before, the eikonal equation is given by 
\begin{equation}\label{eikeq0}
\lambda \partial_t\omega_\lambda + F_1(t,x,y,\lambda \partial_{x,y} {\omega}_\lambda) - i\lambda\partial_{\eta}^2 F_1(t,x,y,{\lambda}\partial_{x,y} {\omega}_\lambda)  \partial_{y}^2{\omega}_\lambda = 0
\end{equation}
modulo bounded terms.  By Proposition~\ref{prepprop} we have $F_1 \in S(\Lambda, g_k) $, $F_1 \circ \chi =  f  \in S^1 $  modulo $ S^{1 - {1}/{k}}$ when $|\eta | \ls |\xi |^{1- {1}/{k}} $ for $k < \infty $, and $F_1 \cong p_1 = f   \in S^1 $ modulo terms in $ S^2$ vanishing of infinite order at $\st  $ when $ k = \infty$. We have assumed that $ f$ is independent of~$ y$, $\re  f(t,x,{\xi},0) \equiv 0$,
$t \mapsto \im  f(t,x,{\xi},0)$ changes sign from $+$ to $-$ as $t$ increases in a neighborhood of $I' $ and $f = 0$  at ${\Gamma}' = \set{(t,x_0, {\xi}_0,0):\ \xi_0 \ne 0 \  \ t \in
I'}$. If $|I' | \ne 0$ then by
Proposition~\ref{prepprop} we
may assume that~$f\restr \st$ vanishes of infinite order 
at~${\Gamma}'$. When $k < \infty $ there exists $c \in S(1,g_k) $ so that $F_1 - c \partial_\eta F_1 $ is independent of~$y $ modulo $S(1,g_k) $ when  $|\eta | \ls |\xi |^{1- {1}/{k}} $ and when $k = \infty $ we may assume that $ f$ is independent of~$y $.  

We will use the phase function
\begin{multline}\label{omegaexp01}
{\omega}_\lambda (t,x,y) = w_0(t) + \w{x-x_0(t), {\xi}_0(t)}  + \sum_{2 \le i
\le K} w_{i,0}(t) (x-x_0(t))^{i}/i!  \\ + \lambda^{\varrho- 1}\Big( \w{y- y_0(t), {\eta}_0(t)} + \sum_{\substack{2 \le i + j 
\le K \\ j \ne 0}} w_{i,j}(t) (x-x_0(t))^{i}(y- y_0(t))^j/i! j! \Big)
\end{multline}
for sufficiently large $ K$, where we will choose $0 < \varrho < 1/2  $, $\xi_0(0) = \xi_0 \ne 0$, $\eta_0(0)= 0 $, $\im w_{2,0}(0)> 0$, $\im w_{1,1}(0) =  0$ and $\im w_{0,2}(0) > 0$, which will give $ \partial_{x,y}^2 \im{\omega}_\lambda >
0$ when $  t= 0 $ and $|x - x_0(0)| + | y- y_0(0)| \ll 1$ which then holds in a neighborhood.
Here as before we use the multilinear forms $w_{i,j} = \set {w_{\alpha,\beta}i!j!/\alpha ! \beta !}_{|\alpha| = i, |\beta| = j}$,  $ (x-x_0(t))^j = \set{ (x-x_0(t))^\alpha}_{|\alpha| = j}$ and $ (y-y_0(t))^j = \set{ (y-y_0(t))^\alpha}_{|\alpha| = j}$ to simplify the notation. Observe that $ x_0(t)$, $ y_0(t)$, $ \xi_0(t)$, $ \eta_0(t)$ and $ w_{j,k}(t)$ will depend uniformly  on $ \lambda$.

Putting ${\Delta}x = x-x_0(t)$ and $\Delta y = y - y_0(t)$ we find that 
\begin{multline}\label{dtomega0}
\partial_t{\omega}_\lambda(t,x,y) = w_0'(t)  - \w{x_0'(t) ,{\xi}_0(t) } -  \lambda^{\varrho - 1}\w{y_0'(t) ,{\eta}_0(t) } \\ + \w{{\xi}'_0(t) - w_{2,0}(t)x_0'(t)- \lambda^{\varrho - 1}w_{1,1}(t)y_0'(t), {\Delta}x}  
\\ + \lambda^{\varrho - 1}\w{{\eta}'_0(t)  - w_{1,1}(t)x_0'(t) -  w_{0,2}(t)y_0'(t), {\Delta}y}
\\  + \sum_{2 \le i  \le K}(w_{i,0}'(t)- w_{i+1,0}(t)x_0'(t)- \lambda^{\varrho - 1} w_{i,1}(t)y_0'(t)) ({\Delta}x)^{i}/i!  
\\ + \lambda^{\varrho - 1} \sum_{\substack{2 \le i + j 
\le K \\ j\ne 0}} (w_{i,j}'(t) - w_{i+1,j}(t)x_0'(t) - w_{i,j+1}(t)y_0'(t)) ({\Delta}x)^{i}({\Delta}y)^j/i! j!
\end{multline}
where the terms $w_{i,j}(t) \equiv 0 $ for  $i + j > K $.  We have
\begin{multline}\label{dxomega0}
\partial_{x} {\omega}_\lambda(t,x,y) = {\xi}_{0}(t) +
\sum_{1 \le i \le K-1}   w_{i + 1,0}(t)({\Delta}x)^{i}/i! \\ + \lambda^{\varrho - 1} \sum_{\substack{1 \le i + j 
\le K \\ j \ne 0}} w_{i+1,j}(t) ({\Delta}x)^{i}({\Delta}y)^j/i! j! 
 ={\xi}_{0}(t)  +  {\sigma}_{0}(t,x) +  \lambda^{\varrho - 1}{\sigma}_{1}(t,x,y)
\end{multline}
Here  ${\sigma}_0$ is a finite expansion in powers of
${\Delta}x$ and  ${\sigma}_1$ is a finite expansion in powers of
${\Delta}x $ and $\Delta y $. We also find
\begin{multline}\label{dyomega0}
\partial_{y} {\omega}_\lambda(t,x,y) = 
 \lambda^{\varrho - 1} \Big( {\eta}_0 + \sum_{1 \le  i + j \le K-1 } w_{i ,j+1}(t) ({\Delta}x)^{i}({\Delta}y)^j/i! j! \Big) \\ =
 \lambda^{\varrho - 1}\big({\eta}_{0}(t)  +  {\sigma}_{2}(t,x,y) \big)
\end{multline}
where  ${\sigma}_2$ is a finite expansion in powers of
${\Delta}x$ and $\Delta y$. 
The main change from Sect.~\ref{eiksect} is that $\lambda^{{1}/{k} }\eta_0 $
get replaced by $\lambda^{\varrho }\eta_0 $ in~\eqref{dyomega}. 
Since the phase function is complex valued, the values will be given by a formal Taylor expansion of the symbol at the real values. 

In the case $k < \infty $, the blowup  $F_1\circ \chi $  gives the Taylor expansion of $ F_1$ at $\eta = 0 $, for example $ f$ is the $ k$:th Taylor term of $ p$ with the constant term of $ p_1$. 
By expanding, we find
\begin{multline}\label{f1exp}
F_1(t, x,y,\lambda \partial_{x,y}\omega_\lambda)  \cong  \lambda F_1(t, x,y,\xi_0 + \sigma_0, 0)  + \lambda^{\varrho} \partial _\xi F_1(t, x,y,\xi_0 + \sigma_0 , 0){\sigma}_{1}\\  +  \lambda^{\varrho} \partial _\eta F_1(t, x,y,\xi_0 + \sigma_0 , 0)({\eta}_{0}  +  {\sigma}_{2}) 
 + \lambda^{2\varrho} ({\eta}_{0}  +  {\sigma}_{2})\partial _\eta^2 F_1(t, x,y,\xi_0 + \sigma_0, 0)({\eta}_{0}  +  {\sigma}_{2}) /2
\end{multline}
modulo $\Cal O(\lambda^{-\kappa}) $ for some $\kappa > 0 $ if $\varrho \ll 1$.  In fact,   on $ \st $ we have $p =  \partial p= \partial_\xi^2 p = 0$  so $ \partial _{\xi, \eta} F_1  \in S^0$. The last term of \eqref{f1exp} is  $\Cal O(\lambda^{2\varrho -1}) $  if $k > 2 $ since then $ \partial _{\eta}^2 F_1  \in S^{-1}$ on $ \st $.  Similarly, we find
\begin{multline}\label{df1exp}
\partial _\eta F_1(t, x,y,\lambda \partial_{x,y}\omega_\lambda)  \cong  \partial _\eta F_1(t, x,y,\xi_0 + \sigma_0, 0) 
\\ +  \lambda^{\varrho} \partial _\eta^2 F_1(t, x,y,\xi_0 + \sigma_0, 0)({\eta}_{0}  +  {\sigma}_{2})  
\end{multline}
modulo $\Cal O(\lambda^{-\kappa}) $, where the last term vanish if $k > 2 $.

In the case $k = \infty $ we have that the principal symbol $p \in S^2$ vanishes of infinite order when $\eta = 0 $, which gives $ p(t, x,y,\lambda \partial_{x,y}\omega_\lambda)  = \Cal O(\lambda^{2 -j(1-\varrho)}) $ for any $ j$. Thus, we may assume that $ F_1 \cong f$ modulo $ S^0$ when $k = \infty $ which gives $\partial_y F_1 \cong 0 $ modulo $S^0 $. 
Since $ \partial_\eta ^2 F_1$ is bounded  and $\partial_y^2 \omega_\lambda =  \Cal O(\lambda^{\varrho - 1}) $ by~\eqref{dyomega0} we obtain that last term in   \eqref{eikeq0} is $ \Cal O(\lambda^{\varrho}) $. Thus we find from  \eqref{eikeq0} and  \eqref{f1exp} that
\begin{equation}\label{eikeq00}
\partial_t\omega_\lambda + f(t,x,\xi_0 + \sigma_0, 0) \cong 0
\end{equation}
modulo  $\Cal O(\lambda^{-\kappa}) $.  Observe we shall solve~\eqref{eikeq00} modulo terms that are $ \Cal O(\lambda^{-1}) $. 
When $x = x_0$  we obtain from \eqref{f1exp}  and \eqref{eikeq00} that
\begin{equation}\label{dw0}
w_0'  - \w{x_0'(t) ,{\xi}_0(t) } + f(t,  x_0,\xi_0, 0)  = 0
\end{equation}
modulo  $\Cal O(\lambda^{-\kappa}) $, which gives the equations~\eqref{2a} with $\re f \equiv \eta_0 \equiv 0$. 

Similarly, since $ F_1 = f$ when $\eta = 0 $ the first order terms in $\Delta x$ of~\eqref{eikeq0} vanish if
\begin{equation}\label{220}
{\xi}_0'(t) -  w_{2,0}(t)x_0'(t) + 
\partial_x f(t,x_0(t),{\xi}_0(t),0) 
 + \partial _{\xi} f(t,x_0(t),{\xi}_0(t),0)w_{2,0}(t)  = 0
\end{equation}
modulo  $\Cal O(\lambda^{-\kappa}) $.
By taking real and imaginary parts we find from~\eqref{220} that \eqref{2} holds
with  $\eta_0 \equiv 0$. We put $(x_0(0), {\xi}_0(0)) = (x_0, {\xi}_0)$, which will determine $x_0(t)$ and ${\xi}_0(t)$ if $\im w_{2,0}(t) \ne 0$. 
The second order terms in $\Delta x$ vanish if
\begin{equation}\label{221}
w'_{2,0} - w_{3,0}x_0'
+ \partial_{\xi}f w_{3,0} + \partial_x^2 f + 
2 \Re \left( \partial_x\partial_{\xi}f w_{2,0}\right) + w_{2,0} \partial_{\xi}^2f w_{2,0} = 0
\end{equation}
modulo $\Cal O(\lambda^{-\kappa}) $, where $ \Re A = ( A + A^t)/2 $  is the symmetric part of $ A$. 
Here and in what follows, the values of the symbols are taken at $(t,x_0(t),y_0(t),\xi_0(t),0)  $.
This gives the equation \eqref{w2eq} modulo $\Cal O (\lambda^{-\kappa}) $
with $\eta_0 \equiv 0$ and we choose initial data $ w_{2,0}(0) $ such that 
$\im w_{2,0}(0) > 0$ which then holds in a neighborhood. 
Similarly, for $ j > 2$ we obtain 
\begin{equation}\label{222}
w'_{j,0}(t)  =  w_{j+1,0}(t) x_0'(t) 
- \Big(f\big(t, x, \xi_0(t) + \sigma_0(t,x) ,0 \big)\Big)_j 
\end{equation}
modulo  $ \Cal O(\lambda^{-\kappa})$, where we have taken the $ j$:th term of the expansion in $ \Delta x$.
Observe that \eqref{220}--\eqref{222} only involve $x_0 $, $ \xi_0$ and $w_{j,0} $ with $j \le K $.

When $k < \infty $, we expand $F_1 \circ \chi \cong f + r $ modulo $S^{1 - {2}/{k}} $ when $|\eta | \ls |\xi |^{1- {1}/{k}} $, where $ r \in S^{1 - {1}/{k}}$ is homogeneous, independent of $ y$ and $\partial_{{\eta}} r = 0 $ when $ f$ vanishes (of infinite order).
Since $\partial_\eta f = 0 $ at $\st  $, we find that  $\partial_\eta F_1 \circ \chi =  |\xi |^{{1}/{k}} \partial_\eta r \in S^0$ at $ \st$. Observe that $r $ consists of the Taylor terms of $p $ of order $ k + 1$ and the  first order Taylor terms of $p_1 $ at $\st  $.
Now at $\st $  we find $ \partial_{\xi} F_1 \cong \partial_{\xi }f $  and $ \partial_\eta^2 F_1  \cong |\xi |^{{2}/{k}}\partial_\eta^2 f \in S^0$  modulo $ S^{-1}$.
We also have $\partial_y F_1 \cong \partial_y c \partial_\eta F_1 \in S(\Lambda^{1/k}, g_k)$  modulo $S(1,g_k)$  so we find
\begin{equation}\label{dyf1}
\partial_y \partial_\eta F_1 \ \cong \partial_\eta ( \partial_y c \partial_\eta F_1) \cong 0
\end{equation} 
modulo $ S(1,g_k)$ when $|\eta| \ls |\xi|^{1 - {1}/{k}}$, which gives $\partial_y^\alpha F_1 \cong \partial_y^\alpha c \partial_\eta F_1 \in S(\Lambda^{1/k}, g_k)$  modulo $S(1,g_k)$.  
In the case $ k = \infty$, we have $ r \equiv 0$, $\partial_y F_1 \cong \partial_y f = 0$ modulo $ S^0$ and we may formally put $ 1/k = 0$ in the formulas.

By \eqref{dyf1},  \eqref{f1exp}  and \eqref{df1exp} the first order terms in $\Delta y$ of~\eqref{eikeq0} are equal to
\begin{multline}\label{1y}
\lambda^\varrho\eta_0'   - \lambda^\varrho w_{0,2}y_0' - \lambda^{\varrho} w_{1,1}x_0'  +  \lambda^{\varrho} \partial_\xi f w_{1,1}  +  \lambda^{\varrho} \partial_\eta r  w_{0,2} 
\\+ 2\lambda^{ {2}/{k} + 2\varrho - 1} \eta_0\partial_\eta^2 f  w_{0,2}
 + \lambda^{ {2}/{k} + \varrho - 1} \partial_y c\partial_\eta^2 f\eta_0 - i \lambda^{ {2}/{k} + \varrho  - 1}\partial_\eta^2 f  w_{0,3} 
\end{multline}
modulo $ \Cal O(1)$ since $\partial_y F_1 \cong  \partial_y c \partial_\eta F_1 \cong  |\xi |^{{1}/{k}} \partial_y c \partial_\eta r \cong 0$ on $\st $ modulo bounded terms.  In the case $ k = \infty$, we put $ r \equiv 0$ and  $\partial_\eta^2 f \equiv 0 $. 
The terms in \eqref{1y}  vanish if
\begin{multline}\label{2y}
\eta_0'   - w_{0,2}y_0' - w_{1,1}x_0'  +  \partial_\xi f w_{1,1}  +  \partial_\eta r  w_{0,2} 
+ 2\lambda^{{2}/{k} + \varrho  - 1} \eta_0\partial_\eta^2 f  w_{0,2} 
\\ + \lambda^{{2}/{k} - 1} \partial_y c\partial_\eta^2 f\eta_0 - i \lambda^{{2}/{k} - 1}\partial_\eta^2 f  w_{0,3} = 0
\end{multline}
modulo $\Cal O(\lambda^{-\varrho})$. 
The real part of~\eqref{2y} gives
\begin{multline}\label{detaeq1}
\eta_0'  = \re w_{0,2}y_0' + \re w_{1,1}x_0'  - \re \partial_\xi f    w_{1,1}  - \re \partial_\eta r    w_{0,2} 
\\ - 2\lambda^{{2}/{k} + \varrho - 1}\re  \eta_0 \partial_\eta^2 f  w_{0,2} 
 - \lambda^{{2}/{k} - 1}  \re \partial_y c\partial_\eta^2 f\eta_0 + \lambda^{{2}/{k} - 1}\im  \partial_\eta^2 f  w_{0,3} 
\end{multline}
modulo $\Cal O(\lambda^{-\kappa}) $, and we will choose initial data $ \eta_0(0) = 0$.

By taking the imaginary part of  \eqref{2y} we find  
\begin{multline}\label{detaeq3}
\im w_{0,2}y_0' = - \im w_{1,1}x_0' + \im \partial_\xi f  w_{1,1} 
 + \im \partial_\eta r  w_{0,2}
\\ +  2\lambda^{ {2}/{k} + \varrho  - 1} \im \eta_0 \partial_\eta^2 f  w_{0,2} 
 + \lambda^{{2}/{k} - 1}  \im \partial_y c\partial_\eta^2 f\eta_0   - \lambda^{{2}/{k} - 1} \re  \partial_\eta^2 f  w_{0,3}
 \end{multline} 
modulo $\Cal O(\lambda^{-\kappa}) $. 
When $k > 2 $ we have $\partial_\eta^2 f \equiv 0 $ on $ \st$
and when $k = 2 $   we shall use  Lemma~\ref{lemclaim} with $ \eta_0 = 0$ to obtain that~\eqref{detaeq1} and~\eqref{detaeq3} are uniformly integrable if $\varrho \ll 1 $.

By using the expansions~\eqref{f1exp} and~\eqref{df1exp}, we can obtain the coefficients for the term $ \Delta x^j \Delta y^\ell$  in~\eqref{eikeq0} from the expansion of
\begin{multline}
\lambda^{\varrho}\sum_{j,\ell \ne 0}(w_{j,\ell}'  - w_{j+1,\ell}x_0'  - w_{j,\ell+1}y_0') \Delta x^j \Delta y^\ell/j!\ell! 
 + \lambda^{\varrho} \partial_\xi f  \sigma_1 +  \lambda^{\varrho} \partial_\eta r \sigma_2 
\\ + \lambda^{ {2}/{k}+ 2\varrho - 1 }  ({\eta}_{0}  +  {\sigma}_{2})\partial_\eta^2 f({\eta}_{0}  +  {\sigma}_{2})/2  
 + \lambda^{{2}/{k}+ \varrho -1}  \partial_y c \Delta y   \partial_\eta^2 f  ({\eta}_{0}  +  {\sigma}_{2})  -i \lambda^{ {2}/{k}+ \varrho  - 1 } \partial_\eta^2 f  \partial_y \sigma_2
\end{multline}
modulo $ \Cal O(1)$. In the case $ k = \infty$ we put $ r \equiv 0$ and  $\partial_\eta^2 f \equiv 0 $. 
Here the last terms can be expanded in~$ \Delta x$ and $ \Delta y$ which also  involves the $\xi $ derivatives.
Taking the coefficient for $ \Delta x^j \Delta y^\ell$ and dividing by $ \lambda^{\varrho }$ we obtain that these terms vanish if
\begin{multline}\label{wij}
w_{j,\ell}'  =  w_{j+1,\ell}x_0' + w_{j,\ell+1}y_0' 
 -  j!\ell!\Big( \partial_\xi f  \sigma_1 +   \partial_\eta r \sigma_2 
\\ + \lambda^{{2}/{k}+ \varrho -1} ({\eta}_{0}  +  {\sigma}_{2}) \partial_\eta^2 f({\eta}_{0} +  {\sigma}_{2})/2 
 + \lambda^{{2}/{k} - 1}  \partial_y c \Delta y \partial_\eta^2 f  ({\eta}_{0}  +  {\sigma}_{2}) -i  \lambda^{{2}/{k} - 1} \partial_{\eta}^2 f \partial_y \sigma_2\Big)_{j,\ell}
\end{multline}
modulo $ \Cal O(\lambda^{-\kappa})$, for some $\kappa > 0 $.

When $ k = \infty$ we find that these equations form a uniformly integrable system of nonlinear ODE. When $ k < \infty$ and  $\im w_0(t) \ge 0 $ then by using  Lemma~\ref{lemclaim} with $\eta_0 = 0 $,  $ \lambda \gg 1$ and $\varrho \ll 1 $ we obtain a uniformly integrable system, which gives a local solution near $(0, x_0,y_0,\xi_0,0)$. 
When $f(t,x_0,{\xi}_0,0) = 0 $ for $t \in I'$ when  $|I' |\ne 0$, we have assumed that $ \partial_{x}^{{\alpha}} \partial_{{\xi}}^{{\beta}} f(t,x_0,{\xi}_0,0) = 0$, $\forall \, \alpha\, \beta $,
for  $t \in I'$.  When $k=2 $ we use Lemma~\ref{lemclaim} to obtain that
$ \partial_{x}^{{\alpha}} \partial_{{\xi}}^{{\beta}}\partial_\eta^2 f(t,x_0,{\xi}_0,0)  \in   I(\lambda^{-\delta})$  for $t \in I'$, $\forall \, \alpha\, \beta $, where $I(\lambda^{-\delta}) $ is given by Definition~\ref{defI}. 
Then~\eqref{220} gives that $x_0' = \xi_0' = 0 $ on $ I' $ and \eqref{dw0} gives that $w_0' = 0$  on $ I' $. Equations~\eqref{221} and~\eqref{222} give that $ w'_{j,0} = 0$  on $ I' $ for $j \ge 2 $.
By~\eqref{wij} we find when $\ell > 0 $ that 
\begin{equation*}
w_{j,\ell}' \cong  w_{j,\ell+1}\big(y_0' - j!\, \ell ! \, \partial_\eta r \big)  \qquad \text{on $I'$ modulo   $ I(\lambda^{-\kappa})$}
\end{equation*} 
for some $\kappa > 0  $ where $y_0' = I(1) $. Since $w_{j,\ell} \equiv 0 $ when $ j + \ell > K$ and $w_{j,\ell}(0) = 0 $ when $j + \ell > 2 $ we find by recursion that  $w_{j,\ell}(t) \cong  0 $ when $j + \ell > 2 $ and $w_{j,\ell}'(t)  \cong  0 $ for  $t \in I' $ modulo   $ I(\lambda^{-\kappa})$ when  $j + \ell =  2 $. By~\eqref{detaeq3} we find that 
\begin{equation}
y_0' \cong (\im w_{0,2}(0))^{-1}\im \partial_\eta r w_{0,2}(0)   \qquad \text{on $ I'$  modulo   $ I(\lambda^{-\kappa})$} 
\end{equation}
which gives $y_0'  =  o(1) $ in $ I$, and \eqref{detaeq1} gives $ \eta_0'  \cong \re(y_0'  -  \partial_\eta r) w_{0,2}(0)   =  o(1) $ modulo   $ I(\lambda^{-\kappa})$. 
In fact, we assume that $\partial_\eta r = 0$ when  $ \im f$ vanishes of infinite order. We may choose $I' $ as the largest interval containing 0 such that $w_0  $ vanish on $I' $. 
Then in any neighborhood of an endpoint of $ I' $ there exists points where $w_0 > c \ge \lambda^{\kappa -1} $ for $ \lambda \gg 1$.

Now  $f $ and $ r$ are independent of $ y$ near the semibicharacteristic, so the  coefficients of the system of equations are independent of $y_0(t)$ modulo $  I(\lambda^{- \kappa})$.
 (If the symbols are independent of $ y$ in an arbitrarily large $ y$ neighborhood we don't need the vanishing condition on $ \partial_\eta r$.) 
Since we restrict $f $ and $ r$ to $\eta = 0 $ and these functions are independent of $ y$, the  coefficients of the system of equations are independent of $(y_0(t),\eta_0(t)) $ modulo $  I(\lambda^{- \kappa})$.
Thus for $\lambda \gg 1 $  the system has a solution $\omega_\lambda $ in a neighborhood of ${\gamma}' = \set {(t,x_0(t),y_0(t)): \ t \in I'}$. 
As before, the Lagrange error term in the Taylor expansion of~\eqref{eikeq0}   is $ \Cal O(\lambda (|x-x_0(t)| + |y - y_0(t)|)^{K+1})$.

But we have to show that $t \mapsto \im f(t,x_0(t), {\xi}_0(t),0) = f_0(t)$  changes sign from $+$ to $-$ as $t$ increases for some choice of initial values $ (t_0,x_0,{\xi}_0) $ and $w_{2,0}(0) $. 
Then we obtain that  $\im w_0(t) \ge 0$ for the solution to $ \im w_0'(t) = - \im f (t, x_0(t), \xi_0(t), 0) $ with suitable initial data.
We shall use the same argument as in Sect.~\ref{eiksect}.  
Observe that \eqref{220}--\eqref{222} only involve $x_0 $, $ \xi_0$ and $w_{j,0} $ with $j \le K $ and are uniformly integrable.
When the sign change is of first order we can use Remark~\ref{initclaim} to choose $w_{2,0}(0) $ so that $|\big (x'_0(0), \xi_0'(0)\big )| \ll 1 $ and  $\im w_{2,0}(0) > 0$. We have
\begin{equation}\label{df00}
f_0'(0) = \im \partial_t f (0, x_0,{\xi}_0, 0)  + \im \partial_x f (0, x_0,{\xi}_0, 0) x_0'(0)   + \im \partial_\xi f (0, x_0,{\xi}_0, 0)  \xi_0'(0)
\end{equation}
and since $\im \partial_t f(0, x_0,{\xi}_0, 0) < 0 $ we obtain that $t \mapsto f_0(t) $ has a sign change from $+$ to $-$ of first order  as $t$ increases if $ |\big (x'_0(0), \xi_0'(0)\big )| \ll 1$. 

We also have to consider the general case when $t \mapsto \im f(t,x_0,{\xi}_0,0)$
changes sign from $+$ to $-$ of higher order as $ t$ increases near $ I'$. If  there exist points  in any $(x,\xi) $ neighborhood
of~${\Gamma}' $ for $\eta=  {0}$
where $\im f = 0 $ and $\partial_t \im f < 0$, then by changing the initial data we can as before construct approximate solutions  for which $t \mapsto
\im w_0(t)$ has a local minimum equal to 0 on~$I$ when $\lambda \gg 1 $.
Otherwise we have  $\im \partial_t f \ge 0$ when $\im f = 0$ in some $(x,\xi) $ neighborhood of~${\Gamma}'$. Then we take the asymptotic solution $ w(t) = (x_0(t),\xi_0(t), w_{j,0}(t) )$  to \eqref{220}--\eqref{222}  when $\lambda \to \infty $ with $ \eta_0(t) \equiv 0$ and initial data $w = (x,{\xi})$ but fixed $w_{2,0}(0) $ and  $w_{j,0}(0)$. This gives a change of coordinates $(t,x,\xi) \mapsto (t,w(t))  $ near ${\Gamma}'$. 
In fact, the solution is constant on $ \Gamma'$ when $|I'| \ne 0 $ since all the coefficients of \eqref{220}--\eqref{222} vanish there.
By the invariance of  
condition (${\Psi}$) there would then exist a change of sign of $t
\mapsto \im f(t,w(t),0) $ from 
$+$ to $-$ in any neighborhood of~${\Gamma}'$. 
Thus by choosing suitable initial values  $ (t_0,x_0,{\xi}_0) $ arbitrarily close to~${\Gamma}'$ we obtain that   $t \mapsto  f_0(t)$  changes sign from  $+$ to $-$  as $t$ increases.

Since $\im w_0'(t) = - \im f (t,x_0(t), {\xi}_0(t),0) $ we obtain that
\begin{equation}\label{11a}
e^{i{\lambda}{\omega}_\lambda(t,x)} \le e^{-c({\lambda}(\im w_0(t) + c|{\Delta}x |^2) + {\lambda}^\varrho|{\Delta}y |^2)}
\qquad | {\Delta}x| + |{\Delta}y | \ll 1
\end{equation}
where  $\min_I \im w_0(t) = 0$ with $\im w_0(t) >
0$ for $t \in \partial I$. 
This gives the following result.

\begin{prop}\label{eikprop00}
Let ${\Gamma}' = \set{(t,x_0,y_0; 0,{\xi}_0, 0): \ t \in I'}$ so that
$\partial_{x}^{{\alpha}} \partial_{{\xi}}^{{\beta}} f(t,x_0,{\xi}_0, 0) = 0$,  $\forall \, \alpha\, \beta $,
for all $t \in I'$ in the case  $|I' |\ne 0$.
Then for $ \varrho \ll 1$ we may solve~\eqref{eikeq0} modulo $\Cal O(\lambda(|x-x_0(t)| + |y - y_0(t)|)^M)$, $\forall\, M$,
with ${\omega}_{\lambda}(t,x)$  given by~\eqref{omegaexp01} in a neighborhood of
${\gamma}' = \set {(t,x_0(t),y_0(t)): \ t \in I'}$.
When $t \in I'$ we find that $(x_0(t), {\xi}_0(t)) = (x_0,{\xi}_0)$,  $w_0(t) = 0$, 
$w_{1,1}(t) \cong 0$  and $w_{j,k}(t) \cong 0$ for $j + k > 2$ modulo $ \Cal O(\lambda^{-\kappa})$ for some $ \kappa > 0$, $\im w_{2,0}(t) > 0$ and $\im w_{0,2}(t) > 0$. 

If $t \mapsto f(t,x_0,{\xi}_0, 0)$ changes sign from $+$ to $-$
as $t$ increases near~$I'$ then by choosing initial values we may obtain that
$\set { (t,x_0(t),y_0(t);0, {\xi}_0(t),0): \ t \in I}$ is
arbitrarily close to~${\Gamma}$,  $\min_{t \in I}
\im w_0(t) = 0$ and $\im w_0(t) > 0$ for $t  \in \partial I_0$.
\end{prop}

\section{The Transport Equations}\label{trans}

Next, we shall solve the transport equations for the amplitudes $\phi  \in C^\infty $, first in the case when $k < \infty $ and $\eta_0 \ne 0 $ as in Sect.~\ref{eiksect}.
Then we use the phase function~\eqref{omegaexp}, by expanding the transport equation 
using~\eqref{dxomega}--\eqref{F1exp} and \eqref{detaf1} we find that it is given by the following terms in~\eqref{exp0}:
\begin{multline}\label{transexp}
  D_t{\phi}  + \left(\lambda^{1/k}\partial_{\eta} f(t,x,\xi_0 + \sigma_0, \eta_0)
 + \lambda^{ {2}/{k} + \varrho -1 }\partial_{\eta}^2 f(t,x,\xi_0 + \sigma_0, \eta_0)\sigma_2 \right) D_{y} {\phi} 
   \\ + \lambda^{{2}/{k}-1}\partial^2_{\eta} f(t,x,\xi_0 + \sigma_0, \eta_0) D_{y}^2 {\phi}/2 
 + \partial_{\xi} f(t,x,\xi_0 + \sigma_0, \eta_0)D_x \phi \\ + F_0(t,x,y,D_y){\phi} = 0
\end{multline}
modulo $ \Cal O(\lambda^{-\kappa})$ for some $\kappa > 0 $ near ${\gamma}' = \set{(t,x_0(t),y_0(t)):\ t \in I'}$ given by Proposition~\ref{eikprop0}. 
Here $ f \in S(\Lambda, g_k) $,  $0 < \varrho < 1/2 $, $ x_0(t)$, $ y_0(t)$, $ \xi_0(t)$, $\eta_0(t) $ and $ \sigma_j$ are given by \eqref{omegaexp}, \eqref{dxomega} and \eqref{dyomega}, and $F_0(t,x,y,D_y)$ is a uniformly bounded first order differential operator. 
In fact, by Proposition~\ref{prepprop} we have that $\partial_y F_1 \cong \partial_{y}c \partial_{\eta} F_1 \in S(\Lambda^{1/k},g_k) $ modulo $S(1,g_k) $ which gives that
$\partial_y\partial_{\eta} F_1(t,x,y,\lambda \partial_{x,y} {\omega}_\lambda)$ is uniformly bounded
by Remark~\ref{symbrem}.

We shall choose the initial value of the amplitude ${\phi} = 1$ for $t =  t_0$ such that $\im w_0(t_0) = 0$, and because of \eqref{11} we only have to solve the equation modulo  $\Cal O(\lambda^\mu (|x-x_0(t)|+ |y- y_0(t)|)^M)$ for some $ \mu $ and any $ M$.  We first solve~\eqref{transexp}, but because of the lower order terms in~\eqref{transexp} 
we will expand $ \phi =  \phi_0 + \lambda^{-\kappa}\phi_1 + \lambda^{-2\kappa}\phi_2 + \dots$ in an asymptotic series with $\phi_j \in C^\infty $, which we will use in \eqref{udef}.

By making  Taylor expansions in $ \Delta x = x-x_0(t)$ and $\Delta y = y- y_0(t)$  of  $\phi_0  $ and  the coefficients of  \eqref{transexp} we obtain a system of ODE's in the Taylor coefficients of $ \phi_0$. 
Observe that the Lagrange error terms of the Taylor expansions in the transport equation give terms that are 
$\Cal O( \lambda^{1/k}(|x-x_0(t)|+ |y- y_0(t)|)^{M+1} )$ since $\varrho \le 1/k $. 
By taking  $\varrho $ small enough and using 
Lemma~\ref{lemclaim} with Remark~\ref{rrem}  as in  Sect.~\ref{eiksect}, we may assume that this system has uniformly integrable coefficients.  Thus we get a uniformly bounded solution $ \phi_0 $ to~\eqref{transexp} 
modulo  $\Cal O(\lambda^{1/k}(|x-x_0(t)|+ |y- y_0(t)|)^M + \lambda^{-\kappa})$ for any $ M$ such that ${\phi}_0(t_0) \equiv 1$. By induction we can successively make the lower order terms in \eqref{transexp} to be  $\Cal O(\lambda^{1/k}(|x-x_0(t)|+ |y- y_0(t)|)^M+ \lambda^{-\ell\kappa})$ by solving \eqref{exp0} for $ \phi_\ell$  with right hand side depending on $ \phi_j$, $ j < \ell$, such that  $ \phi_\ell(t_0) \equiv 0$. Thus, we get a solution to \eqref{exp0} modulo  $\Cal O(\lambda^{1/k} (|x-x_0(t)|+ |y- y_0(t)|)^M + \lambda^{-N})$ for any $ M$ and $ N$.

In the case $\eta_0 = 0 $, we use the  phase function~\eqref{omegaexp01}. By expanding \eqref{exp0} and using \eqref{dxomega0}--\eqref{df1exp}, the transport equation for $ \phi$ becomes:
\begin{multline}\label{transexp2}
D_t{\phi}  + \left( \partial_{\eta} F_1 (t,x,\xi_0 + \sigma_0, 0) + \lambda^{\varrho }\partial_{\eta}^2 F_1 (t,x,\xi_0 + \sigma_0, 0) ( \eta_0 + \sigma_2) \right ) D_{y} {\phi} 
\\ + \partial_{\xi} F_1 (t,x,\xi_0 + \sigma_0, 0) D_{x} {\phi_0} + \partial^2_{\eta} F_1(t,x,\xi_0 + \sigma_0,0 ) D_{y}^2 {\phi}/2 
\\+ F_0(t,x,y,D_y){\phi} = 0
\end{multline}
near ${\gamma}' = \set{(t,x_0(t),y_0(t)):\ t \in I'}$ modulo $ \Cal O(\lambda^{-\kappa})$ for some $\kappa > 0 $ if $\varrho \ll  1 $.  Since  $\partial_y \partial_{\eta} F_1 \cong \partial_{\eta}(\partial_{y}c \partial_{\eta} F_1) $ is bounded we find that $F_0(t,x,y,D_y)$ is a  uniformly first order bounded differential operator by  Remark~\ref{symbrem0}. 
On $ \st $ we have $\partial_{\eta}F_1  \in S^{0}$, $\partial_{\eta}^2 F_1 \in S^{-1}$ when  $ k > 2$, and $\partial_{\eta}^2 F_1 = \partial_{\eta}^2 f $ when $ k = 2$.
We shall solve  \eqref{transexp2}  with initial value  ${\phi} \equiv 1$ when $t =  t_0$.

As before we expand  $ \phi =  \phi_0 + \lambda^{-\kappa}\phi_1 + \lambda^{-2\kappa}\phi_2 + \dots$ in an asymptotic series with $\phi_j \in C^\infty $, which we will use in \eqref{udef}. Observe that the Lagrange term of the Taylor's expansions in the transport  equation is $\Cal O(\lambda^{\varrho} (|x-x_0(t)|+ |y- y_0(t)|)^K )$ for any $ K$. 
By taking  the Taylor expansions  in $ \Delta x$ and $\Delta y$ of $ \phi_0$ and the coefficients of \eqref{transexp2}, we obtain a system of ODE's in the Taylor coefficients of $\phi_0 $.  As in Sect.~\ref{eta0}, we find from Lemma~\ref{lemclaim} that this system has uniformly integrable coefficients when  $\varrho \ll 1$. So by choosing  ${\phi}_0(t_0) \equiv 1$ we obtain a uniformly bounded solution  to~\eqref{transexp2} 
modulo $\Cal O(\lambda^{\varrho}(|x-x_0(t)|+ |y- y_0(t)|)^M + \lambda^{-\kappa})$ for any $ M$. 

We can successively make the lower order terms in \eqref{exp0}  to be  $\Cal O(\lambda^{\varrho}(|x-x_0(t)|+ |y- y_0(t)|)^M + \lambda^{-\ell \kappa})$ by solving the equation \eqref{transexp2} for~$\phi_\ell $ with right hand side depending on $ \phi_j$ for $j < \ell $ such that $ \phi_\ell(t_0) \equiv 0$. 
Thus we find that~\eqref{exp0} holds modulo  $\Cal O(\lambda^{\varrho} (|x-x_0(t)|+ |y- y_0(t)|)^M+ \lambda^{-N})$  for any $ M$ and $ N$ and we have  ${\phi}(t,x,y) = 1$ when $t =  t_0$.

\begin{prop}\label{transprop}
Assume that Propositions~\ref{prepprop},  \ref{eikprop0} and~\ref{eikprop00} hold. 
Then for $\varrho \ll 1 $ and any $ M$ and $ N$ we can solve the transport 
equations so that the expansion  \eqref{exp0}  is 
$ \Cal O(\lambda (|x-x_0(t)|+ |y- y_0(t)|)^M + \lambda^{-N} )$  near  $\gamma' = \set {(t,x_0(t),y_0(t)): \ t \in I'}$. We have ${\phi}\in S(1, g_{1-\varrho})$ uniformly with support
in a neighborhood of  $\gamma'$ where $x -x_0(t) = \Cal O(\lambda^{\varrho - {1}/{k}}) $, $y-y_0(t) = \Cal O(\lambda^{-{\varrho}/{4}}) $ and  $\im w_0(t) =  \Cal O( \lambda^{\varrho - 1}) $. We also have
${\phi}(t_0,x_0(t_0),y_0(t_0)) = 1$, $\lambda \gg 1 $, for some $t_0
\in I'$ such that   $\im w_0(t_0) = 0$.
\end{prop}

In fact, we obtain this by cutting off the solution $ \phi$ near  $\gamma'$. The cutoff in $ (x,y)$ can be done for $\varrho \ll 1/k $  by the cutoff function 
\begin{equation*}
 \psi \big((x -x_0(t))\lambda^{{1}/{k} - \varrho}, (y -y_0(t))\lambda^{{\varrho}/{4}} \big) \in S(1,g_{1 - \varrho})
\end{equation*}
where $\psi(x,y)\in C^\infty_0 $ such that $\psi = 1 $ in a neighborhood of the origin. In fact,
differentiation in $ x$ and $y $ gives factors that are $\Cal O(\lambda^{{1}/{k} - \varrho} + \lambda^{{\varrho}/{4}}) =   \Cal O(\lambda^{{1} - \varrho})$.  Differentiation in $ t$ gives factors $x_0' \lambda^{{1}/{k} - \varrho} $ and $y_0' \lambda^{{\varrho}/{4}} $. Here $x_0' \in C^\infty $ uniformly by \eqref{2} and  \eqref{220},  and $y_0' =  \Cal O(\lambda^{1/k}) $ by  \eqref{detaeq0} and  \eqref{detaeq3}. Repeated differentiation of $y_0' $ gives at most factors $ \Cal O(\lambda^{1/k}) $ by \eqref{detaeq}, \eqref{detaeq0}, \eqref{wjleq}, \eqref{2y} and \eqref{wij}.

The cutoff in $ t$ can be done where $\im w_0(t) \cong \lambda^{\varrho - 1}$ by the function $\chi \big (\im w_0(t)\lambda^{1 - \varrho} \big ) \in S(1, \lambda^{2-2\varrho}dt^2) $ with $ \chi \in   C_0^\infty(\br)$ such that $\chi = 1 $ near 0. By~\eqref{11}  and~\eqref{11a}  the cutoff errors will be $ \Cal O(\lambda^{-N}) $ for any $ N$. We obtain that  ${\phi}(t,x,y)\in S(1,g_{1-\varrho})$ uniformly, ${\phi}(t_0,x_0(t_0),y_0(t_0)) = 1$ and  $\im w_0(t_0) = 0$  for some $t_0 \in I' $.

\section{Proof of Lemma  \ref{lemclaim}}\label{lempf}  

Observe that  $k < \infty $ and that if  Lemma~\ref{lemclaim} holds for some $\delta$ and
$C$, then it trivially holds for smaller~$\delta$ and $\kappa $ and larger~$C$. 
Assume that~\eqref{detaass1}  (or \eqref{detaass2} when  $k=2 $) holds at $t$, by switching $t$ and $-t$ we may
assume $t > 0$. 
Assume that $ \im w_0(t) \ge 0$ satisfies $\im w_0'(t) = - \im f(t,x_0(t),\xi_0(t),\eta_0)$ and put 
\begin{align}
&f_0(t) =  | \im f(t,x_0(t),\xi_0(t),\eta_0)| \\
&f_1(t) =  | \partial_{x,\xi}^\alpha \partial_{\eta} f(t,x_0(t),\xi_0(t),\eta_0)| \\
&f_2(t) =  |\partial_{x,\xi}^\alpha \partial_{\eta}^2 f(t,x_0(t),\xi_0(t),\eta_0)|
\end{align}
for a fixed $ \alpha$.

We shall first consider the case when $ \im w_0(t)$ has a zero of \emph{finite order} at $ t = 0$. Then since $ 0$ is a minimum, $ \im w_0'(t)$ has a sign change of  finite order from $ -$  to $ +$ at  $t= 0$.  Since $t\im w_0'(t) \ge 0 $ we have that $ \im w_0(t) = \int_0^t f_0(s) \, ds$ for $t > 0 $ so \eqref{detaass1}, \eqref{deta0} and the Cauchy-Schwarz inequality give
\begin{equation}
\lambda^{-{1}/{k} -\delta} \ls \int_0^t f_1(s) \,ds \ls \int_0^t f^{{1}/{k} + \varepsilon}_0(s) \,ds \ls \im w_0(t)^{{1}/{k} + \varepsilon}
\end{equation}
for $0 < t \ll 1 $. Thus $\im w_0(t)  \gs \lambda^{-({1 + k\delta})/({1 + k\varepsilon})}  $ and since $\delta < \varepsilon $ we obtain $\lambda \im w_0(t) \gs {\lambda}^{\kappa}$  for some $\kappa > 0 $.

In the case when $ k = 2$ and \eqref{detaass2} holds, we similarly find from \eqref{ddeta0} that
\begin{equation}
\lambda^{-\delta}  \ls \int_0^t f_2(s) \,ds \ls \int_0^t f^{ \varepsilon}_0(s) \,ds \ls \im w_0(t)^{\varepsilon} \qquad 0 <  t  \ll 1
\end{equation}
which implies that $\lambda \im w_0(t) \gs \lambda^{1- {\delta}/{\varepsilon} } \gs {\lambda}^{\kappa}$  for some $\kappa > 0 $  since $\delta < \varepsilon $.

Next we consider the general case when $\im w_0(t) $ vanishes of \emph{infinite order} at $ t=0$, then $f_0(t) $ also vanishes of infinite order. 
For $\varepsilon \ge 0 $ let $I_\varepsilon $ be the maximal interval containing 0 such that $\im w_0 \le \varepsilon $ on $I_\varepsilon $. By assumtion \eqref{deta0} (and \eqref{ddeta0} when $k = 2 $) holds in a neighborhood $ I$ of $I_0 $. By continuity, we have $I_\varepsilon \downarrow I_0 $ when $\varepsilon \downarrow 0 $. Since $ \im w_0 = \varepsilon$ on $\partial I_\varepsilon $ where $ \varepsilon \gs \lambda^{\kappa - 1} = o(1) $  for $\lambda \gg 1 $ it suffices to prove the result in $ I$ for large enough $ \lambda$.
Observe that if $f_1 \ll \lambda^{-{1}/{k} - \delta} $ and $f_2 \ll \lambda^{1-{2}/{k} - \delta} $
in $ [0,t]$ then neither~\eqref{detaass1} nor~\eqref{detaass2} can hold.
If for  some $s \in
[0, t]$ we have that $f_1(s) \gs \lambda^{-{1}/{k} - \delta} $ 
(or $f_2(s) \gs \lambda^{- \delta}  $  when $ k = 2$) 
then by~\eqref{deta0} we find that 
\begin{equation}
\lambda^{-{1}/{k}  - \delta} \ls f_1(s) \ls |f_0(s)|^{{1}/{k}  +  \varepsilon }
\end{equation}
(or  $\lambda^{- \delta} \ls f_2(s) \ls |f_0(s)|^{ \varepsilon } $  by~\eqref{ddeta0}).
Since $ \delta <  \varepsilon $ we find that in both cases 
$ f_0(s)  \ge c \lambda^{-1 + \varrho}$ for some $ \varrho > 0$ and $c > 0 $.
Now we define $t_0$ as the smallest $t > 0$  such that $|\im w_0' (t_0)| =  f_0(t) \ge c\lambda^{-1 + \varrho} $. 
Since $ f_0(t)$ vanishes of infinite order at $ $t=0, we find that $ f_0(t_0)\le C_N|t_0|^N$ for any $ N \ge 1$,
which gives $|t_0| \gs \kappa^{1/N} $. 
Thus, we can use Lemma~\ref{intlem} below with $ \kappa =  c \lambda^{-1 + \varrho}$ for $\lambda \gg 1 $ to obtain that 
\begin{equation}
\max_{0 \le s \le t_0} \im w_0(s) \gs {\kappa}^{1 + {1}/{N}} \cong {\lambda}^{(-1 + \varrho)(1 + {1}/{N})} 
\end{equation}
where $(-1 + \varrho)(1 + {1}/{N}) = -1 + \varrho - (1-\varrho)/{N} > -1$ if we choose $N > {1}/{\varrho}  - 1 $, which gives the result. \qed

\begin{lem}\label{intlem}
Assume that $0 \le F(t) \in C^\infty$ has local minimum at $t = 0$,
and let $I_{t_0}$ be the closed interval
joining $0$ and $t_0\in \br$. If 
$$\max_{I_{t_0}}|F'(t)| = |F'(t_0)| = {\kappa}\le 1$$ 
with $|t_0| \ge c{\kappa}^{\varrho}$ for some  ${\varrho} > 0$ and $c > 0 $,
then we have $\max_{I_{t_0}}  F(t) \ge 
C_{\varrho, c}{\kappa}^{1+{\varrho}}$.
The constant $C_{\varrho, c}> 0$ only depends on
${\varrho}$, $ c$  and the bounds on $F$ in $C^\infty$.
\end{lem}

\begin{proof} 
Let $f = F'$ then $F(t) = F(0) + \int_0^t f(s)\,ds \ge \int_0^t
f(s)\,ds$ so assuming  the minimum is $F(0)=0$  only improves the estimate. By switching $t$ to $-t$
we may assume $t_0 \le -c{\kappa}^{\varrho} < 0$. Let
\begin{equation}
g(t) = {\kappa}^{-1}f(t_0 + tc{\kappa}^{\varrho}) 
\end{equation}
then $|g(0)| = 1$, $|g(t)| \le 1$ for $0 \le t \le 1$ and 
\begin{equation}
|g^{(N)}(t)| = c^N{\kappa}^{{\varrho}N
-1}|f^{(N)} (t_0 + tc{\kappa}^{\varrho})| \le C_N \qquad 0 \le t \le 1
\end{equation}
when $N \ge 1/{\varrho}$. By using the Taylor
expansion at $t= 0$ for $N \ge 1/{\varrho}$ we find 
\begin{equation}
g(t) = p(t) + r(t)
\end{equation}
where $p$ is the Taylor polynomial of order $N-1$ of $g$ at $0$, and 
\begin{equation}
r(t) = t^N \int_0^1g^{(N)}(ts)(1-s)^{N-1}\, ds/(N-1)! 
\end{equation}
is uniformly bounded in $C^\infty$ for $0 \le t \le 1$ and $ r(0) = 0 $. Since $g$
also is bounded on 
the interval, we find that $p(t)$ is uniformly bounded in $0 \le t \le 1$. Since
all norms on the finite dimensional space of polynomials of fixed
degree are equivalent, we find that $p^{(k)}(0) = g^{(k)}(0)$ are
uniformly bounded for $0 \le 
k < N$ which implies that $g(t)$ is uniformly bounded in $C^\infty$ for $0 \le t \le
1$. Since $|g(0)|= 1$ there exists a uniformly bounded ${\delta}^{-1} \ge
1$ such that  $|g(t)| \ge 1/2$
when $0 \le t \le {\delta}$, thus $g$ has the same sign in that
interval. Since $ g(t) = {\kappa}^{-1}f(t_0 +
tc{\kappa}^{\varrho})$  we find
\begin{equation}
{\delta}/2 \le  \left|\int_0^{\delta} g(s)\,ds \right| = \left|
{\kappa}^{-\varrho}\int_{t_0}^{t_0 +
c\delta {\kappa}^{{\varrho}}} {\kappa}^{-1}f(t)\,dt/c \right| 
\end{equation}
Since $t_0 + c{\delta}{\kappa}^{\varrho} \le 0$ we find that the variation of
$F(t)$ on $[t_0,0]$ is greater than $c{\delta}{\kappa}^{1+{\varrho}}/2$ 
and since $F \ge 0$ we find that the maximum of
$F$ on $I_{t_0}$ is greater than  $c{\delta}{\kappa}^{1+{\varrho}}/2$.
\end{proof}

\section{The proof of Theorem  \ref{mainthm}}\label{pfsect} 

We shall use the following modification
of Lemma 26.4.15 in~\cite{ho:yellow}. Recall that $\mn{u}_{(k)}$ is
the $L^2$ Sobolev norm of order $k$ of $u \in C_0^\infty$ and let
$\Cal D'_{{\Gamma}} = \set{u \in \Cal D': \wf (u) \subset {\Gamma}}$ for
$ \Gamma \subseteq T^*\br^n $.

\begin{lem}\label{estlem}
Let 
\begin{equation}\label{estlem0}
 u_{\lambda}(x) = \exp(i{\lambda}{\omega_\lambda}(x))
 \sum_{j=0}^M 
 {\varphi}_{j,\lambda} (x){\lambda}^{-j{\kappa}} \qquad {\lambda} \ge 1
\end{equation}
with ${\kappa} > 0$, 
${\omega_\lambda} \in C^\infty (\br^n)$ satisfying $\im 
{\omega_\lambda}\ge 0$, $|\partial \re{\omega_\lambda}| \ge c > 0$, and
${\varphi}_{j,\lambda} \in S(1, \lambda^{2 - 2\varrho}|dx|^2) = S(1, g_{1 - \varrho})$, $ \forall \, j\, \lambda$, for some $\varrho > 0 $.
We assume that  ${\omega_\lambda} \to \omega_\infty $ when $ \lambda \to \infty$, and that
${\varphi}_{j,\lambda}$  has  support in a
compact set $\Omega$, $\forall \,  j\, \lambda$.
Then we have
\begin{equation}\label{estlem1}
 \mn{u_{\lambda}}_{(-N)} \le C {\lambda}^{-N} \qquad {\lambda} \ge 1 \qquad \forall \, N 
\end{equation}
If  $\lim_{\lambda \to \infty}{\varphi}_{0,\lambda} (x_0) \ne 0$ and $\im {\omega_\infty}(x_0) = 0$ for some $x_0$ then
there exists $c > 0$ so that
\begin{equation}\label{estlem2}
  \mn{u_{\lambda}}_{(-N)} \ge c
  {\lambda}^{-N-{n}/{2}} \qquad {\lambda} \ge 1
  \qquad \forall\, N
\end{equation}
Let ${\Sigma} =  \lim_{\kappa \to \infty} \overline {\bigcup_{j,\lambda \ge \kappa}  \supp \varphi_{j,\lambda}} \subset \Omega$ and let $ {\Gamma}$ be the cone
generated by 
\begin{equation}\label{estlem3}
 \set{(x,\partial{\omega_\infty}(x)),\ x
   \in {\Sigma}} 
\end{equation}
Then for any $m$ we find ${\lambda}^m u_{\lambda} \to 0$ in $\Cal
D'_{{\Gamma}}$ so ${\lambda}^m Au_{\lambda} \to 0$ in $C^\infty$ if
$A$ is a pseudodifferential operator such that $\wf(A) \cap {\Gamma} =
\emptyset$. The estimates are uniform  if ${\varphi}_{j,\lambda}$ is uniformly bounded in $  S(1, g_{1 - \varrho})$  with fixed compact support $\forall\, j \, \lambda$ and 
${\omega_\lambda} \in C^\infty$ uniformly with fixed lower bound on
$|\partial \re{\omega_\lambda}|$.
\end{lem}

Observe that by Propositions~\ref{eikprop0} and~\ref{eikprop00} the phase functions $ {\omega_\lambda}$ in~\eqref{omegaexp} or~\eqref{omegaexp01} satisfy the
conditions in Lemma~\ref{estlem} near
$\set{(t,x_0(t), y_0(t)): \ t \in I'}$  
since ${\xi}_0(t) \ne 0$ and $\im {\omega_\lambda}(t,x) \ge 0$. Also, 
the functions ${\phi}_j$ in the expansion~\eqref{udef}
satisfy the conditions in Lemma~\ref{estlem} uniformly in $\lambda $ by Proposition~\ref{transprop}.  
Then  $\Sigma = \set {(t,x_0(t),y_0(t)): t \in I'} $ and
the cone $ {\Gamma}$ is generated by  
\begin{equation}
\set{(t,x_0(t),y_0(t),0,{\xi}_0(t), 0): \ t \in I'}
\end{equation}
In fact, in both the expansions \eqref{omegaexp} and~\eqref{omegaexp01} we have that $\partial{\omega_\infty(t,x,y)} = (0,\xi_0(t) + \sigma_0(t,x), 0)$, and we find by Proposition~\ref{transprop} that the supports of $ \phi_j$ in~\eqref{udef} shrink to the curve $\set{(t,x_0(t),y_0(t)): \ t \in I'}$  as $ \lambda \to \infty$ for any $ j$. 

\begin{proof}[Proof of Lemma \ref{estlem}]
We shall modify the proof of~\cite[Lemma 26.4.15]{ho:yellow} to this case. 
We have that 
\begin{equation}
 \hat u_{\lambda}({\xi}) = \sum_{j=0}^{M}
 {\lambda}^{-j \kappa} \int
 e^{i{\lambda}{\omega_\lambda} (x)
 - i\w{x,{\xi}}} {\varphi}_{j,\lambda}(x)\,dx 
\end{equation}
Let $U$ be a neighborhood of the projection on the second component of the set
in~\eqref{estlem3}. When ${\xi}/{\lambda} \notin U $  for
${\lambda} \gg 1 $ we find that
$$
\overline {\bigcup_j \supp \varphi_{j,\lambda}} \ni  x \mapsto
({\lambda}{\omega_\lambda}(x) - \w{x,{\xi}})/({\lambda} + |{\xi}|)
$$
is in 
a compact set of functions with nonnegative imaginary part with a fixed
lower bound on the gradient of the real part. Thus, by integrating by
parts we find for any positive integer $k$ that   
\begin{equation}\label{pfest}
 |\hat u_{\lambda}({\xi})| \le C_k({\lambda}
 +|{\xi}|)^{-k}\qquad 
 {\xi}/{\lambda} \notin U \qquad {\lambda} \gg 1
\end{equation}
which gives any negative power of ${\lambda}$ for $k$ large enough. If $V$ is bounded and $0 \notin \ol V$ then since $u_{\lambda}$ is
uniformly bounded in $L^2$ we find
\begin{equation}
 \int_{{\lambda}\, V}  |\hat u_{\lambda}({\xi})|^2 (1 +
 |{\xi}|^2)^{-N}\,d{\xi} \le C_V{\lambda}^{-2N}
\end{equation}
which together with~\eqref{pfest} gives~\eqref{estlem1}. If ${\chi}
\in C_0^\infty$ then we may apply~\eqref{pfest} to
${\chi}u_{\lambda}$, thus we find for any
positive integer $k$ that
\begin{equation}
  |\widehat {{\chi}u}_{\lambda}({\xi})| \le
  C ({\lambda}+
  |{\xi}|)^{-k} \qquad  {\xi} \in W \qquad {\lambda} \gg 1
\end{equation}
if $W$ is any closed cone with $(\supp
{\chi}\times W) \bigcap  {\Gamma}  = \emptyset$. Thus we find that
${\lambda}^m u_{\lambda} \to 0$ in $\Cal D'_{{\Gamma}}$ for every $m$.
To prove \eqref{estlem2} we may assume that $x_0 = 0$ and take ${\psi}\in
C_0^\infty$. If $\im {\omega_\infty}(0) = 0$ and $\lim_{\lambda \to \infty}{\varphi}_{0,\lambda} (0) \ne 0$ then since ${\varphi}_{j,\lambda}(x/{\lambda}) = {\varphi}_{j,\lambda}(0) + \Cal O(\lambda^{-\varrho}) $  in $\supp  \psi $ $\forall \, j $ we find that
\begin{multline}\label{limit}
 {\lambda}^{n} e^{- i \lambda \re \omega_\lambda(0)}\w{u_{\lambda}, {\psi}({\lambda}\cdot)} 
  = \int
 e^{i{\lambda}(\omega_\lambda(x/{\lambda}) - \re \omega_\lambda(0))}{\psi}(x)
 \sum_{j}{\varphi}_{j,\lambda}(x/{\lambda})
 {\lambda}^{-j \kappa }\,dx \\ \to \int 
 e^{i\w{\re \partial_x{\omega_\infty}(0),x}}{\psi}(x)
 {\varphi}_{0,\infty}(0)\,dx \qquad {\lambda} \to \infty
\end{multline}
which is not equal to zero for some suitable ${\psi}  \in
C^\infty_0$. Since 
\begin{equation}
 \mn{{\psi}({\lambda}\, \cdot)}_{(N)} \le C_N {\lambda}^{N-n/2}
\end{equation}
we obtain from~\eqref{limit} that $0 < c \le  {\lambda}^{N + n/2}
\mn{u_{\lambda}}_{(-N)}$ which gives~\eqref{estlem2} and the lemma. 
\end{proof}

\begin{proof}[Proof of Theorem~\ref{mainthm}] 
By conjugating with elliptic Fourier integral operators and
multiplying with pseudodifferential operators, we obtain 
that $P^* \in {\Psi}^{2}_{\mathrm{cl}}$ is of the form given by Proposition~\ref{prepprop}
microlocally near~${\Gamma} = \set{(t,x_0,y_0,0,{\xi}_0,0):\ t \in I}$. Thus we may assume
\begin{equation} 
P^* = D_t + F(t,x,y, D_x,D_y) + R
\end{equation}
where  $R\in
{\Psi}^2_{\mathrm{cl}}$ satisfies $\wf_{g_k} (R) \bigcap
{\Gamma}\times \set {\eta_0} = \emptyset$ when $\kappa < \infty $, vanishes of infinite order at $ \st $ if $ \kappa =  \infty$, and the form of the symbol of $ F$ depends on whether $k < \infty $ or  $k = \infty $.

Then we can construct approximate solutions $u_{\lambda}$ to $P^* u_\lambda = 0 $ of the
form~\eqref{udef} for  $ \lambda \to \infty$ by using the 
expansion~\eqref{exp0}. The phase function $\omega_\lambda $ is given by \eqref{omegaexp} in the case when $k < \infty $ and $\eta_0 \ne 0 $ or by \eqref{omegaexp01} in the case when  $\eta_0 = 0 $.

First we solve the eikonal equation~\eqref{eikeq}  modulo $ \Cal O\big(\lambda (|x-x_0(t)|+ |y - y_0(t)|)^M\big)$  for any $M $ by using
Propositions~\ref{eikprop0}  when $k < \infty $ and  $\eta_0 \ne 0 $ or Proposition~\ref{eikprop00} when  $\eta_0 = 0 $. By using Proposition~\ref{transprop} we can solve the transport
equations so that the expansion \eqref{exp0} is $ \Cal O\big(\lambda (|x-x_0(t)|+ |y - y_0(t)|)^M + \lambda^{-N}\big)$  for any $M $ and $ N$ and  $\phi_0(t_0,x_0(t_0),y_0(t_0)) = 1$ for some $t \in I' $.
Because of the phase functions \eqref{11} or \eqref{11a} this gives approximate solutions
$u_{\lambda}$ of the  
form~\eqref{estlem0} in Lemma~\ref{estlem}.  In fact, for any $ N$ we may choose $M$ in
Proposition~\ref{transprop} so that 
$|(D_t + F)u_{\lambda}| \ls {\lambda}^{-N}$.
Now differentiation of $(D_t + F)u_{\lambda}$ can at most give a
factor~${\lambda}$. In fact, differentiating the exponential gives a factor $\lambda $ and  differentiating the amplitude gives either a factor $\lambda^{1 - \varrho} $, or a loss of 
a factor $x-x_0(t)$ or  $y - y_0(t)$ in the expansion, which gives at most a factor ${\lambda}^{{1}/{2} - \varrho}$. 
Because of the bounds on the
support of $u_{\lambda}$ we obtain that 
\begin{equation} \label{8.13}
\mn{(D_t + F){u_{\lambda}}}_{({\nu})} = \Cal
O({\lambda}^{-N-n})
\end{equation} 
for any chosen ${\nu}$. 
Since Propositions  \ref{eikprop0}, \ref{eikprop00} and~\ref{transprop} gives $t_0 $ so that ${\phi}_0(t_0,x_0(t_0),y_0(t_0)) = 1$ 
and $\im \omega_\lambda(t_0,x_0(t_0),y_0(t)) =
0$ when $\lambda \gg 1$, we find
by~\eqref{estlem1} and~\eqref{estlem2} that
\begin{equation}\label{lastest}
{\lambda}^{-N - {n}/{2}} \ls \mn{u}_{(-N)}  
\ls  {\lambda}^{-N} \qquad \forall\, N \qquad {\lambda} \gg 1
\end{equation}  
Since $u_{\lambda}$ has support in a fixed
compact set, 
we find from Remark~\ref{Rrem} and Lemma~\ref{estlem} that $\mn{Ru}_{({\nu})}$ and
$\mn{Au}_{(0)}$ are $\Cal O({\lambda}^{-N-n})$ if $\wf(A)$ does not
intersect ${\Gamma}$. Thus we find from~\eqref{8.13}
and~\eqref{lastest} that~\eqref{solvest} does 
not hold when ${\lambda} \to \infty$, so $P$ is not
solvable at~${\Gamma}$ by Remark~\ref{solvrem}. 
\end{proof}

\bibliographystyle{plain}

%\bibliography{nec}

\end{document}